\documentclass{amsart}
\usepackage{amsthm,amssymb,mathtools}
\usepackage[british]{babel}
\usepackage{enumitem}
\usepackage[latin1]{inputenc}
\usepackage{longtable}
\usepackage{url}
\usepackage{hyperref}

\newtheorem{theorem}{Theorem}
\numberwithin{theorem}{section}
\newtheorem{corollary}[theorem]{Corollary}
\newtheorem{lemma}[theorem]{Lemma}
\newtheorem{proposition}[theorem]{Proposition}

\theoremstyle{definition}
\newtheorem{definition}[theorem]{Definition}
\newtheorem{remark}[theorem]{Remark}
\newtheorem{example}[theorem]{Example}

\newtheorem{notation}[theorem]{Notation}

\theoremstyle{plain}
\newtheorem*{theorem*}{Theorem}

\allowdisplaybreaks

\newcommand{\lpr}[1]{\operatorname{pr_{\mathbf L}^{#1}}}

\newcommand{\dom}{\operatorname{dom}}
\newcommand{\rng}{\operatorname{rng}}
\newcommand{\len}{\operatorname{len}}
\newcommand{\hth}{\operatorname{ht}}
\newcommand{\ordi}{{\operatorname{Ord}}}
\newcommand{\wop}{\operatorname{WOP}}
\newcommand{\bh}{\xrightarrow{{\operatorname{BH}}}}
\newcommand{\bhp}{\operatorname{BH}}
\newcommand{\ax}{{\operatorname{Ax}}}
\newcommand{\cut}{{\operatorname{Cut}}}
\newcommand{\rep}{{\operatorname{Rep}}}
\newcommand{\rref}{{\operatorname{Ref}}}

\newcommand{\prog}{\operatorname{Prog}}

\title{A Higher Bachmann-Howard Principle}
\author{Anton Freund}
\thanks{This paper is no longer up to date: It is superseded by the author's PhD thesis (available at \url{http://etheses.whiterose.ac.uk/20929/}) and the streamlined presentation in \href{https://arxiv.org/abs/1809.06759}{arXiv:1809.06759}. In contrast to the present abstract, we have now found a computable version of our well-ordering principle. Thus the conjecture by Montalb\'an and Rathjen can be considered as fully solved.}

\begin{document}

\begin{abstract}
We present a higher well-ordering principle which is equivalent (over Simpson's set theoretic version of $\mathbf{ATR_0}$) to the existence of transitive models of Kripke-Platek set theory, and thus to $\Pi_1^1$-comprehension. This is a partial solution to a conjecture of Montalb\'an and Rathjen: partial in the sense that our well-ordering principle is less constructive than demanded in the conjecture.
\end{abstract}

\maketitle 
\tableofcontents

The present work is situated at the intersection of set theory, proof theory and reverse mathematics. Accordingly, we motivate our result from two different viewpoints: Beginning with the set theoretic perspective, we invite the reader to recall axiom beta (see~e.g.~\cite[Definition~I.9.5]{barwise-admissible}): It states that any well-founded relation can be collapsed to the $\in$-relation. In particular this turns well-foundedness into a $\Delta$-notion. Not surprisingly, then, axiom beta adds considerable strength to theories which admit $\Delta$-separation and $\Sigma$-collection, such as Kripke-Platek set theory: We can now form the set of arithmetical well-orderings on the natural numbers. This implies the existence of the Church-Kleene ordinal $\omega_1^{\operatorname{CK}}$ and of the set $\mathbb L_{\omega_1^{\operatorname{CK}}}$, which is a transitive model of Kripke-Platek set theory with infinity (see~\cite[Corollary~V.5.11]{barwise-admissible}; such models will henceforth be called admissible sets). Also, the combination of axiom beta and $\Sigma$-collection is known to imply $\Delta_2^1$-comprehension for subsets of the natural numbers (see \cite[Theorem~3.3.4.7]{pohlers98}). For weaker set theories, which do not contain $\Delta$-separation, the contribution of axiom beta can be quite different: Consider for example the set-theoretic version of $\mathbf{ATR_0}$ introduced by Simpson (see~\cite[Section~VII.3]{simpson09}). This set theory contains axiom beta but does not prove $\Pi_1^1$-comprehension or the existence of admissible sets. From a set theoretic perspective, the higher well-ordering principle that we shall introduce can be seen as a strengthening of axiom beta: one that implies the existence of admissible sets, even when the base theory does not contain $\Delta$-separation.

From a different viewpoint, the present paper is part of an ongoing investigation into the reverse mathematics of well-ordering principles. A typical result in this area (due to Marcone and Montalb\'an \cite{marcone-montalban}, and re-proved by Afshari and Rathjen \cite{rathjen-afshari} using different methods) says that the following statements are equivalent over $\mathbf{RCA}_0$:
\begin{enumerate}[label=(\roman*)]
\item the $\omega$-th jump of any set exists;
\item for any well-ordering $X$, a certain term system $\varepsilon_X$ is well-ordered as well;
\item any set is contained in a countable coded $\omega$-model of $\mathbf{ACA}$.
\end{enumerate}
A function such as $X\mapsto\varepsilon_X$ in (ii), which maps any well-ordering to some other well-ordering, is called a well-ordering principle. Many other set existence axioms have been characterized in terms well-ordering principles: arithmetical comprehension (\mbox{Girard} \cite[Theorem~5.4.1]{girard87}, Hirst \cite{hirst94}), $\Pi_{\omega^\alpha}^0$-comprehension (Marcone and Montalb\'an \cite{marcone-montalban}), arithmetical trans\-finite recursion (Friedman, Montalb\'an and Weiermann \cite{friedman-mw}, Rathjen and Weiermann \cite{rathjen-weiermann-atr}, Marcone and Montalb\'an \cite{marcone-montalban}), the existence of $\omega$-models of arithmetical transfinite recursion (Rath\-jen \cite{rathjen-atr}), and the existence of $\omega$-models of bar induction (Rathjen and Vizca\'{i}no~\cite{rathjen-model-bi}). While these results have reached considerable proof theoretic strength they share a limitation in terms of logical complexity: Statements analogous to (ii) and (iii) are equivalent to $\Pi^1_2$-formulas. Thus we cannot expect to replace (i) by a genuine $\Pi^1_3$-statement, such as the axiom of $\Pi^1_1$-comprehension. To overcome this limitation Rathjen \cite{rathjen-wops-chicago,rathjen-atr} and Montalb\'an \cite[Section~4.5]{montalban-open-problems} have proposed to use well-ordering principles of higher type, i.e.~functionals $\mathcal F$ which map each well-ordering principle $f$ to a well-ordering $\mathcal F(f)$. At the place of (ii) above one would state that
\begin{equation}\label{eq:higher-wop-general}
 \text{``for any well-ordering principle $f$, the set $\mathcal F(f)$ is a well-ordering''.}\tag{$\star$}
\end{equation}
Alternatively, one could consider functionals which map well-ordering principles to well-ordering principles (i.e.~increasing the type of the co-domain as well). At the place of (iii) one demands the existence of $\beta$-models. Note that the resulting statements are $\Pi^1_3$-formulas. By \cite[Theorem~VII.2.10]{simpson09} the existence of countable coded $\beta$-models is equivalent to $\Pi^1_1$-comprehension, as desired. A strategy to construct $\beta$-models from higher well-ordering principles has also been suggested by Rathjen (personal communication; cf.~also \cite{rathjen-wops-chicago} and \cite[Section~6]{rathjen-atr}): In the results cited above, Rathjen and collaborators construct the required $\omega$-models via Sch\"utte's method of search trees (or ``deduction chains'', see~\cite[Section~II.3]{schuette77}). In order to construct $\beta$-models, Rathjen proposed to extend these ideas to Girard's~\cite[Section~6]{girard-intro} notion of $\beta$-proof. This fits well with Montalb\'an's \cite{montalban-open-problems} remark that the variable $f$ in (\ref{eq:higher-wop-general}) should range over dilators (note that the quantification needs to be restricted in some way if we are to remain within the realm of second order arithmetic). In the present paper we do not work with $\beta$-proofs and dilators in the strict sense, but the underlying ideas are still of central importance (for more on this point see Remark~\ref{rmk:search-trees-almost-beta-proofs} below). Let us point out that we will construct transitive models of Kripke-Platek set theory rather than $\beta$-models, thus realizing a modification of Rathjen's idea. As Kripke-Platek set theory does not prove axiom beta this does not immediately yield $\beta$-models of second order arithmetic (in contrast, the theory $\mathbf{ATR}_0^{\operatorname{set}}$ in \cite[Theorem~VII.3.27]{simpson09} contains axiom beta). Nevertheless, the existence of transitive models of Kripke-Platek set theory is known to imply $\Pi^1_1$-comprehension (see~\cite[Theorem~3.3.3.5]{pohlers98}).

It is time to describe our higher well-ordering principle in detail. The meta-theory of the present paper will be primitive recursive set theory with infinity (see e.g.~\cite[Section~6]{rathjen-set-functions}). Note that ``primitive recursive'' will always mean ``primitive recursive relative to $\omega$''. In the language of this theory we have function symbols for all primitive recursive (class) functions. We cannot quantify over all of these functions, but we can use Skolemization to quantify over parametrized families: Given a primitive recursive function $(u,x)\mapsto F(u,x)$ quantification over the functions $x\mapsto F(u,x)$ amounts to first-order quantification over the set parameter~$u$, which is of course permitted. In particular we can implement well-ordering principles as follows (cf.~Definition~\ref{def:wop} below): Consider primitive recursive functions $T:(u,\alpha)\mapsto (T_\alpha^u,<_{T_\alpha^u})$ and $|\cdot|_T:(u,s)\mapsto|s|_T^u$. We say that $T^u$ is a (ranked compatible) well-ordering principle, abbreviated as $\wop(T^u)$, if the following holds:
\begin{enumerate}[leftmargin=*,label=(WOP\arabic*)]
 \item For every ordinal $\alpha$ the set $T_\alpha^u$ is well-ordered by $<_{T_\alpha^u}$.
 \item For $\alpha<\beta$ we have $T_\alpha^u\subseteq T_\beta^u$ and ${<_{T_\alpha^u}}={<_{T_\beta^u}}\cap(T_\alpha^u\times T_\alpha^u)$; furthermore we have $T_\lambda^u=\bigcup_{\alpha<\lambda}T_\alpha^u$ for each limit ordinal $\lambda$.
 \item We have
\begin{equation*}
 |s|_T^u=\begin{cases}
             \min\{\alpha\in\ordi\,|\,s\in T_{\alpha+1}^u\}\quad&\text{if such an $\alpha$ exists},\\
             \{1\}\quad&\text{otherwise}.
            \end{cases}
\end{equation*}
\end{enumerate}
As the set $\{1\}$ is not an ordinal (which is its sole purpose) this does, in particular, make the class $T^u=\bigcup_{\alpha\in\ordi}T_\alpha^u$ primitive recursive. Considering the general form~(\ref{eq:higher-wop-general}) of a higher well-ordering principle, we should now describe a functional $\mathcal F$ which transforms each well-ordering principle $T^u$ into a well-ordering $\mathcal F(T^u)$. Viewing the set $u$ rather than the class function $T^u$ as the argument we should strive to define $\mathcal F$ as a primitive recursive class function. In the present paper we take a more abstract approach: Rather than constructing a term system $\mathcal F(T^u)$ explicitly we axiomatize ordinals of the desired order-type. The following notion achieves this (cf.~Definition~\ref{def:bh-collapse} below): Given a well-ordering principle $T^u$ and an ordinal~$\alpha$, a function $\vartheta:T_\alpha^u\rightarrow\alpha$ is called a Bachmann-Howard collapse, abbreviated as $\vartheta:T_\alpha^u\bh\alpha$, if the following holds for all $s,t\in T_\alpha^u$:
\begin{enumerate}[leftmargin=*,label=(BH\arabic*)]
 \item $|s|_T^u<\vartheta(s)$;
 \item if $s<_{T_\alpha^u}t$ and $|s|_T^u<\vartheta(t)$ then $\vartheta(s)<\vartheta(t)$.
\end{enumerate}
These conditions are motivated by Rathjen's notation system for the Bachmann-Howard ordinal (see e.g.\ \cite{rathjen-model-bi}). Now we can define the central notion of the present paper: The higher Bachmann-Howard principle for $T$ is the statement
\begin{equation*}
 \bhp(T):\equiv\forall_u(\wop(T^u)\rightarrow\exists_\alpha\exists_\vartheta\,\vartheta:T_\alpha^u\bh\alpha).
\end{equation*}
From a set theoretic perspective one may view this as a strengthening of axiom beta (for linear orderings): Rather than collapsing the single well-ordering $T_0^u$ we demand a ``compatible collapse of a compatible family''. The reader may wish to consider the simple Example~\ref{ex:no-monotone-collapsing}, the existence proof in Remark~\ref{rmk:Bachmann-Howard-semantically}, and the connection with axiom beta established in Remark~\ref{rmk:comparison-beta}.

Let us now state our main results. In Section~\ref{sect:wo-proof} we will show that primitive recursive set theory proves the following, for each primitive recursive function $T$: If $\wop(T^u)$ holds and $\mathbb A$ is an admissible set with $u\in\mathbb A$ then there is a Bachmann-Howard collapse $\vartheta:T_{o(\mathbb A)}^u\bh o(\mathbb A)$, where~$o(\mathbb A)=\mathbb A\cap\ordi$. The other sections are devoted to the converse direction: We will define a primitive recursive function which maps each countable transitive set $u=\{u_i\,|\,i\in\omega\}$ (with fixed enumeration) and each ordinal $\alpha$ to a linear ordering~$\varepsilon(S_{\omega^\alpha}^u)$, such that the following holds:
\begin{theorem*}
 Working in primitive recursive set theory, consider a countable transitive set $u$. If the implication
\begin{equation*}
 \wop(\alpha\mapsto\varepsilon(S_{\omega^\alpha}^u))\rightarrow\exists_\alpha\exists_\vartheta\,\vartheta:\varepsilon(S_{\omega^\alpha}^u)\bh\alpha
\end{equation*}
holds then there is an admissible set $\mathbb A$ with $u\subseteq\mathbb A$.
\end{theorem*}
If $u$ is hereditarily countable then we may apply the theorem to the transitive closure of $\{u\}$, to get an admissible set $\mathbb A$ with $u\in\mathbb A$. This completes the equivalence between (ii) and (iii) of the following theorem. The equivalence between (i) and (iii) is known (see \cite[Theorem~3.3.3.5]{pohlers98} and \cite[Lemma~7.5]{jaeger-admissibles}, and use \cite[Section~VII.3]{simpson09} to relate second order arithmetic and set theory).
\begin{theorem*}
 The following axioms resp.~axiom schemes are equivalent over primitive recursive set theory, extended by the axiom of countability and axiom beta:
\begin{enumerate}[label=(\roman*)]
 \item $\Pi_1^1$-comprehension for subsets of the natural numbers;
 \item the higher Bachmann-Howard principle, i.e.~the collection of axioms $\bhp(T)$ for all primitive recursive functions $T$;
 \item the statement that every set is an element of an admissible set.
\end{enumerate}
\end{theorem*}
Let us make the connection with second order arithmetic: Eliminating the primitive recursive function symbols, it should be straightforward to show that the base theory of the theorem is conservative over Simpson's \cite[Section~VII.3]{simpson09} set theoretic version of $\mathbf{ATR_0}$. Thus the equivalence between (i), (ii) and (iii) translates into a theorem of $\mathbf{ATR_0}$. Note that (ii) and (iii) are $\Pi_2$-statements in the language of set theory. According to \cite[Theorem~VII.3.24]{simpson09} they correspond to $\Pi^1_3$-statements of second order arithmetic, as expected. In a sense, we have thus solved the conjecture formulated by Montalb\'an \cite[Section~4.5]{montalban-open-problems} and Rathjen \cite[Section~6]{rathjen-atr}. To see why our solution is not completely satisfying, let us once again compare our Bachmann-Howard principles $\bhp(T)$ with the general form (\ref{eq:higher-wop-general}) of a well-ordering principle: The point is that $\bhp(T)$ merely asserts the existence of ordinals with certain properties, while (\ref{eq:higher-wop-general}) requires to compute these ordinals (or some well-orderings of the appropriate order type) by a concrete functional $\mathcal F$. The author currently works on a more constructive version of the present paper, including an explicit definition of $\mathcal F$.

\section{Search Trees for Admissible Sets}\label{sect:search-trees}

In this section we give a primitive recursive construction $(u,\alpha)\mapsto S_\alpha^u$ of ``search trees" for each ordinal $\alpha$ and each countable transitive set $u$. To be more precise, $S_\alpha^u$ will be primitive recursive in $\alpha$ and a given enumeration $u=\{u_i\,|\,i\in\omega\}$ of $u$; for the sake of readability this enumeration will often be left implicit. One may think of $S_\alpha^u$ as an attempted proof of a contradiction in $\mathbb L_\alpha^u$-logic, with the axioms of Kripke-Platek set theory as open assumptions. Recall that the distinctive rule of $\mathbb L_\alpha^u$-logic allows to infer $\forall_x\,\varphi(x)$ from the assumptions $\varphi(a)$ for all $a\in\mathbb L_\alpha^u$. If $\mathbb L_\alpha^u$ satisfies the Kripke-Platek axioms then the construction of $S_\alpha$ cannot be successful, by the correctness of $\mathbb L_\alpha^u$-logic; this will manifest itself in the fact that $S_\alpha^u$ turns out ill-founded. Search trees are distinguished by a converse property: We will be able to transform an infinite branch of $S_\alpha^u$ into a standard model $u\subseteq M\subseteq\mathbb L_\alpha^u$ of the Kripke-Platek axioms. The method of search trees (or ``deduction chains'') is due to Sch\"utte, who used it to prove the completeness theorem for first-order logic (see~\cite[Section~II.3]{schuette77}). The present paper is mainly influenced by Rathjen's use of search trees in the construction of $\omega$-models (cf.~the introduction). Another application of search trees in $\omega$-logic is due to J\"ager and Strahm \cite{jaeger-strahm-bi-reflection}. The author knows of one application of ``$\beta$-search trees'': This is Buchholz' \cite{buchholz-inductive-dilator} construction of a dilator that bounds the stages of inductive definitions.

Recall that the constructible hierarchy relative to $u$ can be given as a primitive recursive function $(u,\alpha)\mapsto\mathbb L_\alpha^u$ (see \cite[Definition~2.3]{rathjen-set-functions}). Let us emphasize that the enumeration $u=\{u_i\,|\,i\in\omega\}$ does \emph{not} come into play at this point: The stage $\mathbb L_0^u$ is simply the set $u$ and we do not require the function $i\mapsto u_i$ to lie in the class $\mathbb L^u=\bigcup_{\alpha\in\ordi}\mathbb L_\alpha^u$. We would like to define our search tree $S^u_\alpha$ as a labelled subtree of $(\mathbb L_\alpha^u)^{<\omega}$, the tree of finite sequences with entries in $\mathbb L_\alpha^u$. However, there is one technical obstruction: We will later need a primitive recursive notion of $\mathbb L^u$-rank. This is problematic, for not even membership in the class $\mathbb L^u$ seems to be primitive recursively decidable. To fix this we replace $\mathbb L^u_\alpha$ by its ranked version
\begin{equation*}
 \mathbf L_\alpha^u=\begin{cases}
 \{\langle 0,a\rangle\,|\,a\in\mathbb L_0^u\} & \text{if $\alpha=0$},\\
 \{\langle\beta,a\rangle\,|\,\beta<\alpha\text{ minimal with }a\in\mathbb L_{\beta+1}^u\}\quad & \text{if $\alpha>0$}.
 \end{cases}
\end{equation*}
Note that $\mathbf L^u=\bigcup_{\alpha\in\ordi}\mathbf L_\alpha^u$ is now a primitive recursive class. It is trivial to define a rank function
\begin{equation*}
 |\cdot|_{\mathbf L}^u:\mathbf L^u\rightarrow\ordi,\qquad |\langle\beta,a\rangle|_{\mathbf L}^u=\beta
\end{equation*}
and a projection
\begin{equation*}
 \lpr{u}:\mathbf L^u\rightarrow\mathbb L^u,\qquad\lpr{u}(\langle\beta,a\rangle)=a
\end{equation*}
such that we have
\begin{equation*}
 |c|_{\mathbf L}^u=\min\{\alpha\in\ordi\,|\,\lpr{u}(c)\in\mathbb L_{\alpha+1}^u\}
\end{equation*}
and
\begin{equation*}
 |c|_{\mathbf L}^u=\min\{\alpha\in\ordi\,|\,c\in\mathbf L_{\alpha+1}^u\}
\end{equation*}
for all $c\in\mathbf L^u$. Note that we have $\mathbf L_\alpha^u\subseteq\mathbf L_\beta^u$ for $\alpha<\beta$ (if $\alpha=0$ use $\mathbb L_0^u\subseteq\mathbb L_1^u$), as well as $\mathbf L_\lambda^u=\bigcup_{\alpha<\lambda}\mathbf L_\alpha^u$ for any limit ordinal $\lambda$. The projection $\lpr{u}:\mathbf L^u\rightarrow\mathbb L^u$ is bijective but has no primitive recursive inverse. However, given a bound $\alpha>0$ with $a\in\mathbb L_\alpha^u$ we can primitive recursively compute the minimal ordinal $\beta<\alpha$ with $a\in\mathbb L_{\beta+1}^u$, leading to $\langle\beta,a\rangle\in\mathbf L_\alpha^u$. Misusing notation we will often write $a$ at the place of $\langle\beta,a\rangle$; the first component can be ``recovered'' by the notation $\beta=|a|_{\mathbf L}^u$. In particular, since $u$ is equal to $\mathbb L_0^u$ we can view each $u_i\in u$ as the element $\langle 0,u_i\rangle$ of $\mathbf L_\alpha^u$ (for any $\alpha$). It will be convenient to assume $0,1\in u$, for then we have $\langle 0,0\rangle,\langle 0,1\rangle\in\mathbb L_0^u$ (we will use $0,1$ as markers for the two conjuncts / disjuncts of a formula).

As stated above, the search tree $S_\alpha^u$ will be a subtree of $(\mathbf L_\alpha^u)^{<\omega}$. Each node of $S_\alpha^u$ will be labelled by an $\mathbf L_\alpha^u$-sequent, a notion that we shall define next: By a formula (of the object language) we shall mean a first-order formula with relation symbols $\in$ and $=$. As common in proof theory we only consider formulas in negation normal form: These are built from negated and unnegated prime formulas by the connectives $\land$ and $\lor$ and the quantifiers $\forall$ and $\exists$. To negate a formula one pushes negation down to the level of prime formulas (applying de~Morgan's laws) and deletes any double negations. Other connectives will be used as abbreviations, such that e.g.~$\varphi\rightarrow\psi$ stands for $\neg\varphi\lor\psi$. An occurrence of a quantifier is called bounded if it is of the form $\forall_x(x\in y\rightarrow\cdots)$ resp.\ $\exists_x(x\in y\land\cdots)$, where $y$ is not the variable~$x$. We will abbreviate this as $\forall_{x\in y}\cdots$ resp.\ $\exists_{x\in y}\cdots$ but we shall not consider bounded quantifiers as separate quantifiers in their own right; rather, they are bounded occurrences of normal quantifiers. A $\Delta_0$-formula is a formula in which all quantifiers are bounded. Formulas may contain arbitrary sets as parameters. By an $\mathbf L_\alpha^u$-formula we mean a closed formula with parameters from $\mathbf L_\alpha^u$. Note that for parameters in $\mathbf L^u$ it is really the projection into $\mathbb L^u$ that counts: e.g.\ the $\mathbf L_\alpha^u$-formula $\langle\beta_0,a_0\rangle\in\langle\beta_1,a_1\rangle$ should really be interpreted as the formula $a_0\in a_1$, together with information on the ranks of the sets $a_0,a_1\in\mathbb L^u$. An $\mathbf L_\alpha^u$-sequent is a finite sequence of $\mathbf L_\alpha^u$-formulas. As usual we write $\Gamma,\varphi$ for the sequent that arises from $\Gamma$ by appending the formula~$\varphi$ as last entry. Concerning semantics, we have a primitive recursive notion of satisfaction of a formula in a set model~$(m,\in\restriction_{m\times m})$ (cf.~\cite[Section~III.1]{barwise-admissible}). In particular we get a primitive recursive truth predicate for $\Delta_0$-formulas with parameters. We stress once more that the parameters of an $\mathbf L_\alpha^u$-formula must be projected into $\mathbb L_\alpha^u$ before its satisfaction or truth are evaluated.

As a final ingredient for the definition of $S_\alpha^u$ we need an enumeration $\langle\theta_k\rangle_{k\in\omega}$ of the axioms of Kripke-Platek set theory with infinity, excluding the instances of foundation (which will hold in any transitive model). To give a concrete description of the axioms, recall the usual $\Delta_0$-formula expressing that a given set is a limit ordinal (cf.~\cite[Chapter~I]{barwise-admissible}: there the formula has complexity $\Delta_1$, but only because of the implementation of urelements). Then $\langle\theta_k\rangle_{k\in\omega}$ lists the axioms
\begin{gather*}
\forall_{x}\forall_{x'}\forall_{y}\forall_{y'}(x=x'\land y=y'\land x\in y\rightarrow x'\in y')\tag{Equality}\\
\forall_x\forall_y(\forall_{z\in x}z\in y\land\forall_{z\in y}z\in x\rightarrow x=y)\tag{Extensionality}\\
\forall_x\forall_y\exists_z(x\in z\land y\in z)\tag{Pairing}\\
\forall_x\exists_y\forall_{z\in x}\forall_{z'\in z}z'\in y\tag{Union}\\
\exists_x\text{``$x$ is a limit ordinal"}\tag{Infinity}
\end{gather*}
and the instances of the axiom schemata
\begin{gather*}
\begin{multlined}[t][0.7\textwidth]
\forall_{v_1}\cdots\forall_{v_k}\forall_x\exists_y(\forall_{z\in x}(\theta(x,z,v_1,\dots,v_k)\rightarrow z\in y)\land\\
\land\forall_{z\in y}(z\in x\land\theta(x,z,v_1,\dots ,v_k)))\end{multlined}\tag{$\Delta_0$-separation}\\
\begin{multlined}[t][0.7\textwidth]
\forall_{v_1}\cdots\forall_{v_k}\forall_x(\forall_{y\in x}\exists_z\theta(x,y,z,v_1,\dots,v_k)\rightarrow\\
\rightarrow\exists_w\forall_{y\in x}\exists_{z\in w}\theta(x,y,z,v_1,\dots,v_k))\end{multlined}\tag{$\Delta_0$-collection}
\end{gather*}
for $\Delta_0$-formulas $\theta$. It will be convenient to make two restrictions on the formula $\theta$ in the axiom schemata: Firstly, we fix some global bound on the number $k$ of parameters in $\Delta_0$-separation and $\Delta_0$-collection axioms. This is a harmless restriction because all other instances can be derived via an encoding of tuples. Secondly, we require the formula $\theta$ in a $\Delta_0$-collection axiom to be a disjunction. This ensures that the existential quantifier in the subformula $\exists_z\theta$ is unbounded. To deduce $\Delta_0$-collection for an arbitrary $\Delta_0$-formula $\theta$ one replaces $\theta$ by the equivalent disjunction $z\neq z\lor\theta$. Now we are ready to define the desired search trees:

\begin{definition}\label{def:construct-search-tree}
 Given a transitive set $u=\{u_i\,|\,i\in\omega\}\supseteq\{0,1\}$ and an ordinal $\alpha$ we define a search tree $S_\alpha^u\subseteq(\mathbf L_\alpha^u)^{<\omega}$ and a labelling $l:S_\alpha^u\rightarrow\text{``$\mathbf L_\alpha^u$-sequents''}$ by recursion on sequences $\sigma\in(\mathbf L_\alpha^u)^{<\omega}$. In the base case $\sigma=\langle\rangle$ we set
\begin{equation*}
 \langle\rangle\in S_\alpha^u\qquad\text{and}\qquad l(\langle\rangle)=\langle\rangle.
\end{equation*}
In the recursion step we assume that $\sigma\in S_\alpha^u$ holds. If $\sigma$ has even length $2k$ then we add the negation of an axiom of Kripke-Platek set theory, setting
\begin{equation*}
 \sigma^\frown a\in S_\alpha^u\Leftrightarrow a=0\qquad\text{and}\qquad l(\sigma^\frown  0)=l(\sigma),\neg\theta_k.
\end{equation*}
If the length of $\sigma$ is odd we analyze the previous sequent: If $l(\sigma)$ contains a true $\Delta_0$-formula then we stipulate that $\sigma$ is a leaf of $S_\alpha^u$. If $l(\sigma)$ consist of false (negated) prime formulas we set
\begin{equation*}
 \sigma^\frown a\in S_\alpha^u\Leftrightarrow a=0\qquad\text{and}\qquad l(\sigma^\frown 0)=l(\sigma).
\end{equation*}
Otherwise we write $l(\sigma)=\Gamma,\varphi,\Gamma'$ such that $\Gamma$ consists of (negated) prime formulas and $\varphi$ is not a (negated) prime formula. For later reference we call $\varphi$ the redex of~$l(\sigma)$. The recursion step depends on the form of $\varphi$ as follows:
{\def\arraystretch{1.75}\setlength{\tabcolsep}{6pt}
\setlength{\LTpre}{\baselineskip}\setlength{\LTpost}{\baselineskip}
\begin{longtable}{lp{0.78\textwidth}}\hline
If\dots & \dots\ then \dots\\ \hline
$\varphi\equiv\psi_0\land\psi_1$ & $\sigma^\frown a\in S_\alpha^u$ iff $a\in\{0,1\}$, and $l(\sigma^\frown i)=\Gamma,\Gamma',\varphi,\psi_i$ for $i=0,1$,\\
$\varphi\equiv\psi_0\lor\psi_1$ & $\sigma^\frown a\in S_\alpha^u$ iff $a=0$, and $l(\sigma^\frown 0)=\Gamma,\Gamma',\varphi,\psi_i$ where $i=0$ if $\psi_0$ does not already occur in $l(\sigma)$ and $i=1$ otherwise,\\
$\varphi\equiv\forall_x\psi(x)$ & $\sigma^\frown a\in S_\alpha^u$ for all $a\in\mathbf L_\alpha^u$, and $l(\sigma^\frown a)=\Gamma,\Gamma',\varphi,\psi(a)$,\\
$\varphi\equiv\exists_x\psi(x)$ & $\sigma^\frown a\in S_\alpha^u$ iff $a=0$, and\newline $l(\sigma^\frown 0)=\Gamma,\Gamma',\varphi,\psi(b)$ where $b$ is the first entry of the list $u_0,\sigma_0,\dots,u_{\dom(\sigma)-1},\sigma_{\dom(\sigma)-1},u_{\dom(\sigma)},u_{\dom(\sigma)+1},\dots$ such that $\psi(b)$ does not already occur in $l(\sigma)$.\\ \hline
\end{longtable}}
\end{definition}

Consider a function $f:\omega\rightarrow\mathbf L_\alpha^u$ and write $f[n]=\langle f(0),\dots ,f(n-1)\rangle$ for $n\in\omega$. If $f[n]\in S_\alpha^u$ holds for all $n\in\omega$ then $f$ is called a branch of $S_\alpha^u$. We say that a formula occurs on $f$ if it occurs in some sequent $l(f[n])$. The construction of search trees ensures the following crucial properties:

\begin{lemma}\label{lem:properties-search-tree}
 For any branch $f$ of the search tree $S_\alpha^u$ the following holds:
\begin{enumerate}[label=(\alph*)]
 \item Any parameter in a formula on $f$ lies in $\rng(f)\cup u$.
 \item Any (negated) prime formula on $f$ is false.
 \item If $\psi_0\land\psi_1$ occurs on $f$ then either $\psi_0$ or $\psi_1$ occurs on $f$.
 \item If $\psi_0\lor\psi_1$ occurs on $f$ then both $\psi_0$ and $\psi_1$ occur on $f$.
 \item If $\forall_x\psi(x)$ occurs on $f$ then $\psi(b)$ occurs on $f$ for some $b\in\rng(f)$.
 \item If $\exists_x\psi(x)$ occurs on $f$ then $\psi(b)$ occurs on $f$ for all $b\in\rng(f)\cup u$.
\end{enumerate}
\end{lemma}
\begin{proof} (a) By induction on $n$ we show that all parameters that occur in the sequent $l(f[n])$ lie in $\rng(f[n])\cup u$: For $n=0$ we have $l(f[0])=\langle\rangle$ and thus no parameters. In the induction step we consider $f[n+1]=f[n]^\frown f(n)$: If $n=2k$ is even then the only new formula in $l(f[n+1])$ is the negated axiom $\neg\theta_k$, which contains no parameters. Now assume that $n$ is odd. The first interesting case is that of a redex $\varphi\equiv\forall_x\psi(x)$. Then $l(f[n+1])$ contains the new formula $\psi(f(n))$. As the formula $\forall_x\psi(x)$ occurs in $l(f[n])$ its parameters lie in $\rng(f[n])\cup u$ by induction hypothesis. The only new parameter in the formula $\psi(f(n))$ is $f(n)$, which is indeed an element of $\rng(f[n+1])$. The other interesting case is a redex $\varphi\equiv\exists_x\psi(x)$. Here $l(f[n])$ contains a new parameter from the list $u_0,f(0),u_1,\dots,f(n-1),u_n,u_{n+1},u_{n+2},\dots$, which is thus an element of $\rng(f[n])\cup u$.\\
(b) Aiming at a contradiction, assume that $l(f[n])$ contains a true (negated) prime formula. By construction this means that $f[n]$ is a leaf of the search tree, contradicting the assumption that $f$ is a branch.\\
(c) Similar to (e), and easier.\\
(d) Similar to (f), and easier.\\
(e) Assume that $\forall_x\psi(x)$ is a formula in $l(f[n])$. By construction of the search tree $\forall_x\psi(x)$ will be the redex of $l(f[m])$ for some $m\geq n$ (a formula to the right of the redex moves a position to the left in each odd step). Then the formula $\psi(f(m))$ occurs in $l(f[m+1])$, again by construction.\\
(f) As in (e) we may assume that $\exists_x\psi(x)$ is the redex of $l(f[m])$. By construction the formula $\psi(u_0)$ occurs in $l(f[m+1])$, and so does $\exists_x\psi(x)$. Now observe that $\psi(u_0)$ and $\exists_x\psi(x)$ remain in $l(f[m'])$ for all $m'\geq m+1$. Pick an $m'$ such that $\exists_x\psi(x)$ is again the redex of $l(f[m'])$. Since $\psi(u_0)$ is already contained in $l(f[m'])$ we may conclude that $\psi(f(0))$ occurs in $l(f[m'+1])$. Inductively we can verify that $\psi(b)$ occurs on $f$ for any $b$ in the list $u_0,f(0),u_1,f(1),\dots$, which enumerates the set $\rng(f)\cup u$.
\end{proof}

To understand the following proposition, recall that $\lpr{u}$ projects the ranked constructible hierarchy $\mathbf L^u$ onto the usual hierarchy $\mathbb L^u$, and that $\mathbf L_\alpha^u$-formulas are to be evaluated under this projection.

\begin{proposition}
 Assume that $f$ is a branch of the search tree $S_\alpha^u$. Any formula that occurs on $f$ is false in the structure $(\rng({\lpr{u}}\circ f)\cup u,\in)$.
\end{proposition}
\begin{proof}
 We establish the claim by induction on the height of formulas. Note that the formulas on $f$ form a set, so that the induction statement is primitive recursive. For a (negated) prime formula the claim holds by the previous lemma. Now let us consider a formula $\forall_x\psi(x)$ that occurs on $f$. By the lemma $\psi(b)$ occurs on $f$ for some $b\in\rng(f)$. The induction hypothesis tells us that $\psi(b)$ is false in the structure $(\rng({\lpr{u}}\circ f)\cup u,\in)$; thus $\forall_x\psi(x)$ must also be false in this structure. Next, consider a formula $\exists_x\psi(x)$ on $f$. The lemma tells us that $\psi(b)$ occurs on $f$ for all $b\in\rng(f)\cup u$. By the induction hypothesis $(\rng({\lpr{u}}\circ f)\cup u,\in)$ does not satisfy any of these formulas. Thus it cannot satisfy $\exists_x\psi(x)$ either. The remaining cases are similar and easier.
\end{proof}

\begin{corollary}
 If $f$ is a branch of the search tree $S_\alpha^u$ then $(\rng({\lpr{u}}\circ f)\cup u,\in)$ is a model of Kripke-Platek set theory.
\end{corollary}
\begin{proof}
 Recall that $\langle\theta_k\rangle_{k\in\omega}$ enumerates the Kripke-Platek axioms, except for foundation. By construction of the search tree all formulas $\neg\theta_k$ occur on $f$. The proposition tells us that $\neg\theta_k$ is false in $(\rng({\lpr{u}}\circ f)\cup u,\in)$. Thus this structure satisfies all axioms $\theta_k$. As the satisfaction relation is primitive recursive, foundation for arbitrary formulas in $(\rng({\lpr{u}}\circ f)\cup u,\in)$ reduces to foundation for primitive recursive predicates in the meta-theory.
\end{proof}

Recall that transitive models of Kripke-Platek set theory are also called admissible sets.

\begin{corollary}\label{cor:branch-to-admissible-set}
 If the search tree $S_\alpha^u$ has a branch then there is an admissible set~$\mathbb A$ with $u\subseteq\mathbb A$.
\end{corollary}
\begin{proof}
 The previous corollary provides a model in which $\in$ is interpreted by the actual membership relation. In particular this is a model of extensionality. Thus it is isomorphic to its Mostowski collapse. As $u$ is transitive it is still contained in the collapsed model. 
\end{proof}

If $u$ is hereditarily countable we can apply the same construction to the transitive closure of $\{u\}$, to obtain an admissible set $\mathbb A$ with $u\in\mathbb A$. To conclude that $u$ is indeed contained in an admissible set one must exclude the case that the search trees $S_\alpha^u$ are well-founded for all ordinals $\alpha$. This will require an additional assumption, namely a ``higher well-ordering principle'' to be described in the next section. In the rest of this section we exhibit some further properties of the trees $S_\alpha^u$, which will link them to such well-ordering principles.

First, let us clarify our notion of well-ordering: We say that $<$ well-orders $a$ if every non-empty subset of $a$ has a $<$-minimal element. Below we will relate this to the existence of infinitely decreasing sequences (in general this requires some amount of choice). Now we want to define linear orderings on the search trees $S_\alpha^u$. Note that the set $u$ admits a well-ordering which is primitive recursive in the given enumeration $u=\{u_i\,|\,i\in\omega\}$. This can be extended to a primitive recursive family of compatible well-orderings $<_{\mathbf L_\alpha^u}$ on the stages $\mathbf L_\alpha^u$ of the (ranked) constructible hierarchy (cf.~\cite[Lemma~VII.4.19]{simpson09}; to show that these are well-orderings one constructs order embeddings $(\mathbf L_\alpha^u,<_{\mathbf L_\alpha^u})\rightarrow\ordi$, also by primitive recursion). Using these well-orderings we can define the Kleene-Brouwer ordering on the tree $S_\alpha^u$, namely as
\begin{equation*}
 \sigma<_{S_\alpha^u}\tau\quad :\Leftrightarrow\quad\left\{
                                \begin{array}{l}
                                  \text{either the sequence $\sigma$ properly extends the sequence $\tau$},\\
                                  \text{or $\sigma={\sigma_0}^\frown a^\frown \sigma'$ and $\tau={\sigma_0}^\frown b^\frown \tau'$ with $a<_{\mathbf L_\alpha^u} b$}.
                                \end{array}
                            \right.
\end{equation*}
Clearly $<_{S_\alpha^u}$ is a linear ordering on $S_\alpha^u$. The promised connection with well-orderedness goes as follows:

\begin{lemma}\label{lem:branch-Kleene-Brouwer}
 If the search tree $S_\alpha^u$ has no infinite branch then its Kleene-Brouwer ordering $<_{S_\alpha^u}$ is a well-ordering.
\end{lemma}
\begin{proof}
 The first step is to show that the ``choice function''
\begin{equation*}
 \textstyle\min_{\mathbf L_\alpha^u}(a):=\text{``the $<_{\mathbf L_\alpha^u}$-minimal element of $a\cap\mathbf L_\alpha^u$''}
\end{equation*}
is primitive recursive. Indeed the set
\begin{equation*}
 \{b\in a\cap\mathbf L_\alpha^u\,|\,\neg\exists_{y\in a\cap\mathbf L_\alpha^u}\,y<_{\mathbf L_\alpha^u}b\}
\end{equation*}
can be computed by a primitive recursive function. As $<_{\mathbf L_\alpha^u}$ is a well-ordering this set is a singleton (or empty, in which case we assign some default value), and we can extract its only element $\textstyle\min_{\mathbf L_\alpha^u}(a)$. Now the claim of the lemma is shown by contraposition: Assume that $a\subseteq S_\alpha^u$ has no $<_{S_\alpha^u}$-minimal element. Observe that we have $a\subseteq(\mathbf L_\alpha^u)^{<\omega}\subseteq \mathbf L_{\alpha+\omega}^u$. Thus we can use the choice function for subsets of $\mathbf L_{\alpha+\omega}^u$ to define a $<_{S_\alpha^u}$-descending sequence $g:\omega\rightarrow x$, setting
\begin{equation*}
 g(n+1):=\textstyle\min_{\mathbf L_{\alpha+\omega}^u}(\{b\in a\,|\,b<_{S_\alpha^u}g(n)\}).
\end{equation*}
To transform $g$ into an infinite branch $f$ of $S_\alpha^u$ we recursively define
\begin{multline*}
 f(n):=\textstyle\min_{\mathbf L_\alpha^u}(\{b\in\mathbf L_\alpha^u\,|\,f[n]^\frown b\in S_\alpha^u\text{ and $f[n]^\frown b$ lies below $g(m)$}\\ \text{for infinitely many $m\in\omega$}\}).
\end{multline*}
Here ``$f[n]^\frown b$ lies below $g(m)$'' means that the sequence $g(m)$ is an end-extension of the sequence $f[n]^\frown b$. Note that the property ``is an infinite subset of $\omega$'' is primitive recursive. To conclude it suffices to show that the required sets $b$ exist: Inductively we assume that $f[n]$ lies (strictly) below infinitely many nodes of the form $g(m)$. Define a strictly increasing sequence of numbers $m_k$ such that $f[n]$ lies below all nodes $g(m_k)$. Let $b_k$ be the unique set in $\mathbf L_\alpha^u$ such that $f[n]^\frown b_k$ lies below $g(m_k)$. From $g(m_{k+1})<_{S_\alpha^u}g(m_k)$ and the definition of the Kleene-Brouwer ordering we get $b_{k+1}\leq_{\mathbf L_\alpha^u} b_k$. As $<_{\mathbf L_\alpha^u}$ is well-founded there must be a bound $K$ such that $b_k=b_K$ holds for all $k\geq K$. It follows that $f[n]^\frown b_K$ lies below $g(m_k)$ for all $k\geq K$, so $b:=b_K$ is as required for the definition of $f(n)$.
\end{proof}

Next, let us observe that the linear orderings $(S_\alpha^u,<_{S_\alpha^u})$ are compatible:

\begin{lemma}\label{lem:search-trees-compatible}
 For $\alpha<\beta$ we have $S_\alpha^u\subseteq S_\beta^u$ and ${<_{S_\alpha^u}}={<_{S_\beta^u}}\cap(S_\alpha^u\times S_\alpha^u)$. If $\lambda$ is a limit ordinal then we have $S_\lambda^u=\bigcup_{\alpha<\lambda} S_\alpha^u$.
\end{lemma}
\begin{proof}
 By induction on sequences $\sigma\in\mathbf L_\alpha^u$ one verifies $\sigma\in S_\alpha^u\Leftrightarrow \sigma\in S_\beta^u$; simultaneously one needs to check that the labels in the two trees coincide. Thus we have $S_\alpha^u=S_\beta^u\cap(\mathbf L_\alpha^u)^{<\omega}$, from which the claims are easily deduced.
\end{proof}

It is important to observe that, given $\alpha<\beta$, the order $(S_\alpha^u,<_{S_\alpha^u})$ is \emph{not} an initial segment of $(S_\beta^u,<_{S_\beta^u})$: Indeed the root $\langle\rangle$ is the biggest element of any $S_\alpha^u$, and it already lies in $S_0^u$ (cf.~Example~\ref{ex:no-monotone-collapsing} below).

\begin{remark}\label{rmk:search-trees-almost-beta-proofs}
 Recall Girard's notion of (pre-)dilator \cite{girard-pi2}. In particular a pre-dilator is a functor from the category of well-orders to the category of linear orders. To turn our construction of search trees
\begin{equation*}
 \alpha\mapsto (S_\alpha^u,<_{S_\alpha^u})
\end{equation*}
into such a functor we would have to assign an embedding $(S_\alpha^u,<_{S_\alpha^u})\rightarrow (S_\beta^u,<_{S_\beta^u})$ to each order preserving map $\alpha\rightarrow\beta$. The previous lemma yields such an embedding for the inclusion map of $\alpha$ into $\beta>\alpha$. The obvious extension to arbitrary maps requires a functorial version of the constructible hierarchy (cf.~the notion of $\beta$-proof in \cite[Section~6]{girard-intro}). Dilators have the foundational advantage that they are finitistically meaningful (see \cite[Section~0.2.1]{girard-pi2}). Apart from that there is no need to work with dilators in the present study --- interestingly enough, though, the embeddings $S_\alpha^u\subseteq S_\beta^u$ from the previous lemma will play a key role. Many ideas that we use come from Girard's work on $\Pi^1_2$-logic, even if we do not work with dilators and $\beta$-proofs in the strict sense. 
\end{remark}

To conclude this section, let us show that the rank function for the constructible hierarchy yields a rank function for the sequence $\alpha\mapsto S_\alpha^u$ of search trees:

\begin{lemma}\label{lem:search-trees-rank}
 The class $S^u:=\bigcup_{\alpha\in\ordi} S_\alpha^u$ is primitive recursive in the given enumeration of $u$. There is a primitive recursive rank function $|\cdot|_S^u:S^u\rightarrow\ordi$ such that we have $|\sigma|_S^u=\min\{\alpha\in\ordi\,|\,\sigma\in S_{\alpha+1}^u\}$ for all $\sigma\in S^u$.
\end{lemma}
\begin{proof}
 First define $|\cdot|_S^u$ on the full tree $(\mathbf L^u)^{<\omega}$, namely by
\begin{equation*}
|\sigma|_S^u=\begin{cases}
            0\quad & \text{if $\sigma=\langle\rangle$},\\
            \max\{|a_0|_{\mathbf L}^u,\dots ,|a_n|_{\mathbf L}^u\}\quad & \text{if $\sigma=\langle a_0,\dots ,a_n\rangle$.}
           \end{cases}
\end{equation*}
Then
\begin{equation*}
 |\sigma|_S^u=\min\{\alpha\in\ordi\,|\,\sigma\in(\mathbf L_{\alpha+1}^u)^{<\omega}\}
\end{equation*}
follows from the corresponding property of the ranked constructible hierarchy. In the proof of Lemma \ref{lem:search-trees-compatible} we have seen $S^u\cap(\mathbf L_\alpha^u)^{<\omega}=S_\alpha^u$. Thus $\sigma\in S^u$ is equivalent to $\sigma\in S^u_{|\sigma|_S^u+1}$, which is a primitive recursive relation. Now we may restrict $|\cdot|_S^u$ to the class $S^u$, and $|\sigma|_S^u=\min\{\alpha\in\ordi\,|\,\sigma\in S_{\alpha+1}^u\}$ follows from the above.
\end{proof}

Depending on the situation it may be more intuitive to think of the compatible family $\alpha\mapsto S_\alpha^u$ of set-sized trees or of the single class-sized tree $S^u$.

\section{A Higher Bachmann-Howard Construction}\label{sect:higher-wop}

In the previous section we have constructed, for each countable transitive set $u$, a family of search trees $\langle S_\alpha^u\rangle_{\alpha\in\ordi}$ with the following property: If there is an ordinal $\alpha$ such that $S_\alpha^u$ is ill-founded then there is an admissible set $\mathbb A$ with $u\subseteq\mathbb A$. It remains to consider the case where all search trees $S_\alpha^u$ are well-founded. In this case the primitive recursive function $\alpha\mapsto S_\alpha^u$ is a \emph{well-ordering principle}, in a sense to be defined below. The goal of this section is to define a notion of Bachmann-Howard ordinal relative to a well-ordering principle. The assertion that ``the Bachmann-Howard ordinal relative to any given well-ordering principle exists'' can itself be described as a \emph{higher well-ordering principle}. In the following sections we will use this higher well-ordering principle to exclude the case that all $S_\alpha^u$ are well-founded. This will finally establish the existence of an admissible set $\mathbb A$ with $u\subseteq\mathbb A$.

We begin with a general notion of well-ordering principle:

\begin{definition}\label{def:wop}
 Consider primitive recursive functions $T:(u,\alpha)\mapsto (T_\alpha^u,<_{T_\alpha^u})$ and $|\cdot|_T:(u,s)\mapsto|s|_T^u$. We say that $T^u$ is a (ranked compatible) well-ordering principle, abbreviated as $\wop(T^u)$, if the following holds:
\begin{enumerate}[label=(\roman*)]
 \item For every ordinal $\alpha$ the set $T_\alpha^u$ is well-ordered by $<_{T_\alpha^u}$.
 \item For $\alpha<\beta$ we have $T_\alpha^u\subseteq T_\beta^u$ and ${<_{T_\alpha^u}}={<_{T_\beta^u}}\cap(T_\alpha^u\times T_\alpha^u)$; furthermore we have $T_\lambda^u=\bigcup_{\alpha<\lambda}T_\alpha^u$ for each limit ordinal $\lambda$.
 \item We have
\begin{equation*}
 |s|_T^u=\begin{cases}
             \min\{\alpha\in\ordi\,|\,s\in T_{\alpha+1}^u\}\quad&\text{if such an $\alpha$ exists},\\
             \{1\}\quad&\text{otherwise}.
            \end{cases}
\end{equation*}
As the set $\{1\}$ is not an ordinal (which is its sole purpose) this does, in particular, make the class $T^u=\bigcup_{\alpha\in\ordi}T_\alpha^u$ primitive recursive.
\end{enumerate}
\end{definition}

The above is in fact a definition scheme: For fixed function symbols~$T$ and~$|\cdot|_T$ we obtain a statement $\wop(T^u)$ with parameter $u$. Observe that $\wop(T^u)$ is a $\Pi_1$-formula in the language of primitive recursive set theory. Also, note that $s\in T_\alpha^u$ implies $|\cdot|_T^u<\alpha$ for all $\alpha>0$, by the minimality of the rank and the ``continuity'' in clause (ii). Let us now define a notion of collapse for ranked well-orderings:

\begin{definition}\label{def:bh-collapse}
 Adding to Definition~\ref{def:wop}, a function $\vartheta:T_\alpha^u\rightarrow\alpha$ is called a Bachmann-Howard collapse of $T^u_\alpha$, abbreviated as $\vartheta:T_\alpha^u\bh\alpha$, if the following holds for all $s,t\in T_\alpha^u$:
\begin{enumerate}[label=(\roman*)]
 \item $|s|_T^u<\vartheta(s)$,
 \item if $s<_{T_\alpha^u}t$ and $|s|_T^u<\vartheta(t)$ then $\vartheta(s)<\vartheta(t)$.
\end{enumerate}
\end{definition}

Observe that $\vartheta:T_\alpha^u\bh\alpha$ is a primitive recursive property of $\vartheta,\alpha,u$. We shall motivate the definition in a moment, but let us first use it to state our higher well-ordering principle:

\begin{definition}\label{def:higher-bh-principle}
 The higher Bachmann-Howard principle for $T$ is the statement
\begin{equation*}
 \bhp(T):\equiv\forall_u(\wop(T^u)\rightarrow\exists_\alpha\exists_\vartheta\,\vartheta:T_\alpha^u\bh\alpha).
\end{equation*}
An ordinal $\alpha$ with $\exists_\vartheta\vartheta:T_\alpha^u\bh\alpha$ is called a Bachmann-Howard ordinal for $T^u$.
\end{definition}

Observe that $\wop(T^u)\rightarrow\exists_\alpha\exists_\vartheta\,\vartheta:T_\alpha^u\bh\alpha$ is a $\Sigma_1$-statement in the language of primitive recursive set-theory. Thus $\bhp(T)$ is a $\Pi_2$-statement. The notion of Bachmann-Howard collapse is motivated by Rathjen's ordinal notation system for the Bachmann-Howard ordinal (see e.g.\ \cite{rathjen-model-bi}). However, we can also give some motivation without recourse to this background: First, note that condition~(i) of Definition~\ref{def:bh-collapse} excludes the trivial solution $\vartheta(s)=0$ for all $s\in T_\alpha^u$, which would fulfill condition (ii). Also note that (i) entails the implication
\begin{equation*}
 \vartheta(t)\leq |s|_T^u\quad\Rightarrow\quad\vartheta(t)<\vartheta(s),
\end{equation*}
familiar from Rathjen's ordinal notation system. Let us record an easy consequence:

\begin{lemma}
 If $T^u$ is a well-ordering principle then any Bachmann-Howard collapse $\vartheta:T_\alpha^u\bh\alpha$ is injective.
\end{lemma}
\begin{proof}
 Consider arbitrary elements $s,t\in T_\alpha^u$. As $T_\alpha^u$ is linearly ordered we may assume $s<_{T_\alpha^u}t$. Now distinguish the following cases: If we have $|s|_T^u<\vartheta(t)$ then condition (ii) of Definition~\ref{def:bh-collapse} implies $\vartheta(s)<\vartheta(t)$. If, on the other hand, we have $\vartheta(t)\leq |s|_T^u$ then we get $\vartheta(t)<\vartheta(s)$, as we have just seen.
\end{proof}

Note that any Bachmann-Howard collapse preserves the ordering between elements of the same rank: If $|s|_T^u\leq |t|_T^u$ then condition (i) of Definition~\ref{def:bh-collapse} yields $|s|_T^u<\vartheta(t)$. Together with condition (ii) this implies $\vartheta(s)<\vartheta(t)$. On the other hand $|s|_T^u\leq |t|_T^u$ is not a necessary condition for $\vartheta(s)<\vartheta(t)$. We will later see that the weaker condition $|s|_T^u<\vartheta(t)$ plays a crucial role. The following explains why we do not require $\vartheta$ to be completely order preserving:

\begin{example}\label{ex:no-monotone-collapsing}
Consider the well-ordering principle $T_\alpha:=\alpha\cup\{\star\}$ with the usual ordering on $\alpha$ and $\star$ as biggest element. Then $T=\bigcup_{\alpha\in\ordi}T_\alpha$ is the class of all ordinals with a maximal element added. We have $|\beta|_T=\beta$ and $|\star|_T=0$. The order type of $T_\alpha$ is $\alpha+1$, so there can be no order preserving map $\vartheta:T_\alpha\rightarrow\alpha$. However, if we demand $\vartheta(s)<\vartheta(t)$ only under the side condition $|s|_T<\vartheta(t)$, then such a map exists for $\alpha=\omega\cdot 2$: Set $\vartheta(\beta)=\beta+1$ and $\vartheta(\star)=\omega$. Indeed, $|\beta|_T<\vartheta(\star)$ now implies $\beta<\omega$ and thus $\vartheta(\beta)<\omega=\vartheta(\star)$. Also observe that $\vartheta(\star)\geq\omega$ must hold for any Bachmann-Howard collapse $\vartheta:T_{\omega\cdot 2}\bh\omega\cdot 2$: First, we must have $0=|\star|_T<\vartheta(\star)$. Inductively we assume $|n|_T=n<\vartheta(\star)$ and infer $n=|n|_T<\vartheta(n)<\vartheta(\star)$, which implies~$n+1<\vartheta(\star)$.
\end{example}

Having seen the example, the reader may rightly ask whether each well-ordering principle allows for a Bachmann-Howard collapse. In Section~\ref{sect:wo-proof} we will show that a Bachmann-Howard collapse can be constructed on the basis of an admissible set. The following foreshadows this construction, but in a strong meta-theory:

\begin{remark}\label{rmk:Bachmann-Howard-semantically}
 Consider a well-ordering principle $T^u$ with $u\in\mathbb L_{\aleph_1}$ (where $\aleph_1$ is the first uncountable cardinal). As $T$ is a primitive recursive function we have $T_\alpha^u\in\mathbb L_{\aleph_1}$ for each $\alpha<\aleph_1$; in particular the sets $T_\alpha^u$ are countable. We define a Bachmann-Howard collapse $\vartheta:T_{\aleph_1}^u\rightarrow\aleph_1$ by recursion over the well-ordering $T_{\aleph_1}^u$. Assuming that $\vartheta(s)$ is already defined for all $s<_{T_{\aleph_1}^u}t$ let us construct sets $C_n(t,\alpha)\subseteq\aleph_1$ by recursion over $n\in\omega$, for all $\alpha<\aleph_1$:
\begin{itemize}
 \item $C_0(t,\alpha)=\alpha\cup\{|t|_T^u\}$,
 \item $C_{n+1}(t,\alpha)=C_n(t,\alpha)\cup\{\vartheta(s)\,|\,s<_{T_{\aleph_1}^u}t\text{ and }|s|_T^u\in C_n(t,\alpha)\}$.
\end{itemize}
Now set $C(t,\alpha):=\bigcup_{n\in\omega}C_n(t,\alpha)$ and
\begin{equation*}
 \vartheta(t):=\min\{\alpha<\aleph_1\,|\,C(t,\alpha)\subseteq\alpha\}.
\end{equation*}
We must verify that such an $\alpha$ exists: For each countable $\beta$ there are only countably many $s\in T_{\aleph_1}^u$ with $|s|_T^u=\beta$, because we have $s\in T_{|s|_T^u+1}^u$. Thus if $C_n(t,\alpha)$ is countable then so is the set
\begin{equation*}
 \bigcup_{\beta\in C_n(t,\alpha)}\{\vartheta(s)\,|\,|s|_T^u=\beta\}.                                                                                               \end{equation*}
Inductively it follows that all $C_n(t,\alpha)$ are countable. So $C(t,\alpha)$ is countable as well. We can thus construct a sequence $0=\alpha_0<\alpha_1<\dots<\aleph_1$ with $C(t,\alpha_n)\subseteq\alpha_{n+1}$. Set $\alpha:=\sup_{n\in\omega}\alpha_n<\aleph_1$. It is easy to verify
\begin{equation*}
 C(t,\alpha)=\bigcup_{n\in\omega}C(t,\alpha_n)\subseteq\bigcup_{n\in\omega}\alpha_{n+1}=\alpha,
\end{equation*}
as required. Conditions (i) and (ii) from Definition~\ref{def:bh-collapse} are readily deduced: We have $|s|_T^u\in C(s,\vartheta(s))\subseteq\vartheta(s)$. Also, $|s|_T^u<\vartheta(t)$ implies $|s|_T^u\in C(t,\vartheta(t))$. Together with $s<_{T_{\aleph_1}^u}t$ this yields $\vartheta(s)\in C(t,\vartheta(t))\subseteq\vartheta(t)$. The proof-theorist will have noticed that the given argument is very similar to the usual construction of the Bachmann-Howard ordinal.
\end{remark}

We want to show that our higher Bachmann-Howard principle implies the existence of admissible sets. Let us compare this claim with some known results:

\begin{remark}\label{rmk:comparison-beta}
Recall axiom beta (see e.g.~\cite[Definition~I.9.5]{barwise-admissible}), which states that any well-founded relation can be collapsed to the $\in$-relation. It is easy to deduce axiom beta for linear orderings from our higher Bachmann-Howard principle: Define $T_\alpha^{(u,<_u)}:=(u,<_u)$ and
\begin{equation*}
|s|_T^{(u,<_u)}:=\begin{cases}
0 & \text{if $s\in u$},\\
\{1\} & \text{otherwise}.
\end{cases}
\end{equation*}
If $(u,<_u)$ is a well-ordering then $T^{(u,<_u)}$ is a well-ordering principle. The higher Bachmann-Howard principle provides a Bachmann-Howard collapse $\vartheta:u\rightarrow\alpha$ for some ordinal $\alpha$. In the present case $\vartheta$ is fully order preserving, because all elements of $u$ receive the same rank zero. Thus axiom beta is established. This observation sparks the following question: Can we construct admissible sets on the basis of axiom beta alone, making our higher Bachmann-Howard principle redundant? Indeed, axiom beta is powerful in the presence of $\Delta$-separation: As axiom beta turns well-foundedness into a $\Delta$-property we see that the definable well-orderings of the natural numbers form a set. Then $\Sigma$-collection ensures the existence of the Church-Kleene ordinal, which is well-known to be admissible. Also, the combination of axiom beta and $\Sigma$-collection implies $\Delta_2^1$-comprehension (see \cite[Theorem~3.3.4.7]{pohlers98}). In particular we get $\Pi_1^1$-comprehension, which is equivalent to the existence of countable admissible sets (see \cite[Theorem~3.3.3.5]{pohlers98} and \cite[Lemma~7.5]{jaeger-admissibles}). To summarize, our higher Bachmann-Howard principle does not add more strength than the bare axiom beta over a base theory that contains $\Sigma$-collection (such as Kripke-Platek set theory). For a base theory that does not contain $\Sigma$-collection and $\Delta$-separation the situation can be quite different: Consider for example the set-theoretic version of $\mathbf{ATR_0}$ introduced by Simpson (see~\cite[Section~VII.3]{simpson09}). This theory contains axiom beta but does not prove $\Pi_1^1$-comprehension or the existence of admissible sets. Over such base theories the higher Bachmann-Howard principle is thus a genuine strengthening of axiom beta (anticipating our construction of admissible sets based on a Bachmann-Howard collapse). We remark that theories without $\Delta$-separation and $\Sigma$-collection are particularly interesting in the context of reverse mathematics. Many interesting questions remain open: What precisely is responsible for the strength of the higher Bachmann-Howard principle? Can we weaken the conditions on a Bachmann-Howard collapse, e.g.~by replacing conditions (i,ii) of Definition~\ref{def:bh-collapse} with the weaker implication $s<_{T_\alpha^u}t\land|s|_T^u=|t|_T^u\Rightarrow\vartheta(s)<\vartheta(t)$? Can we find a set theoretic or recursion theoretic proof of our result? Assuming the higher Bachmann-Howard principle, is there a ``direct" construction of the Church-Kleene ordinal? 
\end{remark}

Recall the construction of search trees $S_\alpha^u$ from the previous section. We will not apply the higher Bachmann-Howard principle to the search trees themselves but rather to a modified well-ordering principle $\alpha\mapsto\varepsilon(S_{\omega^\alpha}^u)$, which combines ideas from \cite[Definition~2.1]{rathjen-afshari} and \cite[Definition~4.1]{rathjen92}. In the following term systems the reader may interpret $\Omega$ as the ordinal $\omega^\alpha$ or, alternatively, as the class of all ordinals: The first interpretation helps to understand each term system individually, while the second clarifies the relation of the term systems for different values of $\alpha$. The terms~$\varepsilon_\sigma$ should be imagined as $\varepsilon$-numbers above $\Omega$.

\begin{definition}\label{def:epsilon-of-ordering}
For each countable transitive set $u=\{u_i\,|\,i\in\omega\}$ and each ordinal~$\alpha$ we define a set of terms $\varepsilon(S_{\omega^\alpha}^u)$ and an order relation $<_{\varepsilon(S_{\omega^\alpha}^u)}$ by the following simultaneous recursion (which will be justified below):
\begin{enumerate}[label=(\roman*)]
 \item The symbol $0$ is a term in $\varepsilon(S_{\omega^\alpha}^u)$.
 \item For each $\sigma\in S_{\omega^\alpha}^u$ the symbol $\varepsilon_\sigma$ is a term in $\varepsilon(S_{\omega^\alpha}^u)$.
 \item Given terms $s_0,\dots ,s_n\in\varepsilon(S_{\omega^\alpha}^u)$ and ordinals $0<\beta_0,\dots ,\beta_n<\omega^\alpha$ the expression
\begin{equation*}
 \Omega^{s_0}\cdot\beta_0+\dots +\Omega^{s_n}\cdot\beta_n
\end{equation*}
is also a term in $\varepsilon(S_{\omega^\alpha}^u)$, provided that the following holds: If we have $n=0$ and $\beta_0=1$ then $s_0$ may not be of the form $\varepsilon_\sigma$. If we have $n>0$ then we require $s_{i+1}<_{\varepsilon(S_{\omega^\alpha}^u)} s_i$ for all $i<n$.
\end{enumerate}
We stipulate that $s<_{\varepsilon(S_{\omega^\alpha}^u)} t$ holds if and only if one of the following is satisfied:
\begin{enumerate}[label=(\roman*)]
 \item $s=0$ and $t\neq 0$ (equality of terms),
 \item $s=\varepsilon_\sigma$ and one of the following holds:
\begin{itemize}
 \item $t=\varepsilon_\tau$ for some $\tau\in S_{\omega^\alpha}^u$ with $\sigma<_{S_{\omega^\alpha}^u}\tau$,
 \item $t=\Omega^{t_0}\cdot\gamma_0+\dots +\Omega^{t_m}\cdot\gamma_m$ and $s<_{\varepsilon(S_{\omega^\alpha}^u)} t_0$ or $s=t_0$,
\end{itemize}
 \item $s=\Omega^{s_0}\cdot\beta_0+\dots +\Omega^{s_n}\cdot\beta_n$ and one of the following holds:
\begin{itemize}
 \item $t=\varepsilon_\tau$ and $s_0<_{\varepsilon(S_{\omega^\alpha}^u)} t$,
 \item $t=\Omega^{t_0}\cdot\gamma_0+\dots +\Omega^{t_m}\cdot\gamma_m$ and either
\begin{itemize}
 \item $n<m$ and $\langle s_i,\beta_i\rangle=\langle t_i,\gamma_i\rangle$ for all $i\leq n$, or
 \item there is a $j\leq\min\{n,m\}$ such that we have either $s_j<_{\varepsilon(S_{\omega^\alpha}^u)} t_j$ or $s_j=t_j$ and $\beta_j<\gamma_j$, and $\langle s_i,\beta_i\rangle=\langle t_i,\gamma_i\rangle$ holds for all $i<j$.
\end{itemize}
\end{itemize}
\end{enumerate}
\end{definition}

The given definition can be justified as follows: First construct a preliminary term system $\varepsilon^0(S_{\omega_\alpha}^u)$ which is defined as above but contains all terms of the form
\begin{equation*}
 \Omega^{s_0}\cdot\beta_0+\dots +\Omega^{s_n}\cdot\beta_n,
\end{equation*}
regardless of the condition $s_{i+1}<_{\varepsilon(S_{\omega^\alpha}^u)} s_i$. Clearly $\varepsilon^0(S_{\omega_\alpha}^u)$ can be constructed by primitive recursion (just as the set of formulas with parameters from a given set). Also by primitive recursion we can define the length of terms in $\varepsilon^0(S_{\omega_\alpha}^u)$, setting
\begin{gather*}
 \len(0):=\len(\varepsilon_\sigma):=0,\\
 \len(\Omega^{s_0}\cdot\beta_0+\dots +\Omega^{s_n}\cdot\beta_n):=\len(s_0)+\dots+\len(s_n)+n+1.
\end{gather*}
Then the conditions in Definition~\ref{def:epsilon-of-ordering} single out a subset $\varepsilon(S_{\omega_\alpha}^u)\subseteq\varepsilon^0(S_{\omega_\alpha}^u)$ and a relation ${<_{\varepsilon(S_{\omega^\alpha}^u)}}\subseteq\varepsilon(S_{\omega_\alpha}^u)\times\varepsilon(S_{\omega_\alpha}^u)$ in the following way:
\begin{itemize}[leftmargin=1cm]
 \item To determine whether we have $s\in\varepsilon(S_{\omega_\alpha}^u)$ we only need to check $t\in\varepsilon(S_{\omega_\alpha}^u)$ for $\len(t)<\len(s)$, and $t<_{\varepsilon(S_{\omega^\alpha}^u)}t'$ for $\len(t)+\len(t')<\len(s)$.
 \item To determine whether we have $s<_{\varepsilon(S_{\omega^\alpha}^u)}t$ we need to check $s'\in\varepsilon(S_{\omega_\alpha}^u)$ for $\len(s')\leq \len(s)+\len(t)$, and $s'<_{\varepsilon(S_{\omega^\alpha}^u)}t'$ for $\len(s')+\len(t')<\len(s)+\len(t)$.
\end{itemize}

It follows that there is a primitive recursive function which constructs $\varepsilon(S_{\omega^\alpha}^u)$ and~$<_{\varepsilon(S_{\omega^\alpha}^u)}$ from $\alpha$ and the given enumeration of $u$. Our next goal is to show that $\alpha\mapsto\varepsilon(S_{\omega^\alpha}^u)$ is a well-ordering principle if all the search trees $S_{\omega^\alpha}^u$ are well-founded. The following is a first step:

\begin{lemma}
 For all $u,\alpha$ the relation $<_{\varepsilon(S_{\omega^\alpha}^u)}$ is a linear ordering of $\varepsilon(S_{\omega^\alpha}^u)$.
\end{lemma}
\begin{proof}
 Recall that $<_{S_{\omega^\alpha}^u}$ is a linear ordering of $S_{\omega^\alpha}^u$. The claim for $<_{\varepsilon(S_{\omega^\alpha}^u)}$ follows by tedious but straightforward inductions on the length of terms (see above): By induction on $\len(s)$ one shows that $s<_{\varepsilon(S_{\omega^\alpha}^u)} s$ is false. To see that $s<_{\varepsilon(S_{\omega^\alpha}^u)} t$ and $t<_{\varepsilon(S_{\omega^\alpha}^u)} r$ imply $s<_{\varepsilon(S_{\omega^\alpha}^u)} r$ one argues by induction on $\len(s)+\len(t)+\len(r)$. Finally, by induction on $\len(s)+\len(t)$ one shows that one of the alternatives $s<_{\varepsilon(S_{\omega^\alpha}^u)} t$, $s=t$ (equality of terms) and $t<_{\varepsilon(S_{\omega^\alpha}^u)} s$ must hold.
\end{proof}

Primitive recursive set theory does not show that any well-ordering is isomorphic to an ordinal. Nevertheless it is instructive to assume that we have an order embedding $c:S_{\omega^\alpha}^u\rightarrow\ordi$, and to consider the following construction: Pick an $\varepsilon$-number $\varepsilon_\eta\geq\omega^\alpha$. Now define a function $o:\varepsilon(S_{\omega^\alpha}^u)\rightarrow\ordi$ by
\begin{align*}
 o(0)&:=0,\\
 o(\varepsilon_s)&:=\varepsilon_{\eta+1+c(s)},\\
 o(\Omega^{s_0}\cdot\beta_0+\dots +\Omega^{s_n}\cdot\beta_n)&:=(\omega^{1+\alpha})^{o(s_0)}\cdot\beta_0+\dots +(\omega^{1+\alpha})^{o(s_n)}\cdot\beta_n.
\end{align*}
By induction on $\len(s)+\len(t)$ one shows that $s<_{\varepsilon(S_{\omega^\alpha}^u)}t$ implies $o(s)<o(t)$. We have thus constructed an order embedding of $\varepsilon(S_{\omega^\alpha}^u)$ into the ordinals. As stated above, we do not in general have the function $c:S_{\omega^\alpha}^u\rightarrow\ordi$ required for this interpretation at our disposal. Nevertheless we will be able to show that $\varepsilon(S_{\omega^\alpha}^u)$ is a well-ordering, provided that the same holds for $S_{\omega^\alpha}^u$. First, we need to extend addition and exponentiation to the full term system $\varepsilon(S_{\omega^\alpha}^u)$. Let us assume $\alpha>0$ to have the coefficients $1,2<\omega^\alpha$ available. Exponentiation to the base $\Omega$ is easily defined, namely by
\begin{equation*}
 \Omega^s:=\begin{cases}
            s\quad&\text{if $s$ is of the form $\varepsilon_\sigma$},\\
            \Omega^s\cdot 1\quad&\text{otherwise}.
           \end{cases}
\end{equation*}
Iterated exponentiation is written as
\begin{align*}
 \Omega_0^s&:=s,\\
 \Omega_{n+1}^s&:=\Omega^{\Omega_n^s}.
\end{align*}
Extending addition to all terms in $\varepsilon(S_{\omega^\alpha}^u)$ is more tedious because we need to distinguish many cases. Luckily, the correct definition is evident if one thinks in terms of Cantor normal forms:
\begin{gather*}
\begin{aligned}
 0+s&=s+0=s,\\[1.5ex]
 \varepsilon_\sigma+\varepsilon_\tau&=\begin{cases}
                                       \varepsilon_\tau\quad&\text{if $\varepsilon_\sigma<\varepsilon_\tau$},\\
                                       \Omega^{\varepsilon_\sigma}\cdot 2\quad&\text{if $\varepsilon_\sigma=\varepsilon_\tau$},\\
                                       \Omega^{\varepsilon_\sigma}\cdot 1+\Omega^{\varepsilon_\tau}\cdot 1&\text{if $\varepsilon_\sigma>\varepsilon_\tau$},
                                      \end{cases}\\[2.5ex]
 \varepsilon_\sigma+(\Omega^{t_0}\cdot\gamma_0+\dots+\Omega^{t_m}\cdot\gamma_m)&=
                                      \begin{cases}
                                        \parbox[t]{0.5\textwidth}{$\Omega^{t_0}\cdot\gamma_0+\dots+\Omega^{t_m}\cdot\gamma_m$\newline \hspace*{3cm}if $\varepsilon_\sigma<t_0$,}\\[2.5ex]
                                        \parbox[t]{0.5\textwidth}{$\Omega^{t_0}\cdot(1+\gamma_0)+\Omega^{t_1}\cdot\gamma_1+\dots+\Omega^{t_m}\cdot\gamma_m$\\ \hspace*{3cm}if $\varepsilon_\sigma=t_0$,}\\[2.5ex]
                                        \parbox[t]{0.5\textwidth}{$\Omega^{\varepsilon_\sigma}\cdot 1+\Omega^{t_0}\cdot\gamma_0+\dots+\Omega^{t_m}\cdot\gamma_m$\\ \hspace*{3cm}if $t_0<\varepsilon_\sigma$,}
                                      \end{cases}\\[2.5ex]
 (\Omega^{s_0}\cdot\beta_0+\dots+\Omega^{s_n}\cdot\beta_n)+\varepsilon_\tau&=
                                      \begin{cases}
                                       \parbox[t]{0.5\textwidth}{\makebox[3cm][t]{$\varepsilon_\tau$}if $s_0<\varepsilon_\tau$,}\\[1.5ex]
                                       \parbox[t]{0.5\textwidth}{$\Omega^{s_0}\cdot\beta_0+\dots+\Omega^{s_{i-1}}\cdot\beta_{i-1}+\Omega^{s_i}\cdot(\beta_i+1)$\newline \hspace*{3cm}if $s_i=\varepsilon_\tau$,}\\[2.5ex]
                                       \parbox[t]{0.5\textwidth}{$\Omega^{s_0}\cdot\beta_0+\dots+\Omega^{s_i}\cdot\beta_i+\Omega^{\varepsilon_\tau}\cdot 1$\newline \hspace*{3cm}if $s_{i+1}<\varepsilon_\tau<s_i$,}\\[2.5ex]
                                       \parbox[t]{0.5\textwidth}{$\Omega^{s_0}\cdot\beta_0+\dots+\Omega^{s_n}\cdot\beta_n+\Omega^{\varepsilon_\tau}\cdot 1$\newline \hspace*{3cm}if $\varepsilon_\tau<s_n$,}
                                      \end{cases}
\end{aligned}\displaybreak[0]\\
\begin{multlined}
 (\Omega^{s_0}\cdot\beta_0+\dots+\Omega^{s_n}\cdot\beta_n)+(\Omega^{t_0}\cdot\gamma_0+\dots+\Omega^{t_m}\cdot\gamma_m)=\\
=\begin{cases}
  \parbox[t]{0.8\textwidth}{$\Omega^{t_0}\cdot\gamma_0+\dots+\Omega^{t_m}\cdot\gamma_m$\newline \hspace*{3cm}if $s_0<t_0$,}\\[2.5ex] \parbox[t]{0.8\textwidth}{$\Omega^{s_0}\cdot\beta_0+\dots+\Omega^{s_{i-1}}\cdot\beta_{i-1}+\Omega^{s_i}\cdot(\beta_i+\gamma_0)+\Omega^{t_1}\cdot\gamma_1+\dots+\Omega^{t_m}\cdot\gamma_m$\newline \hspace*{3cm}if $s_i=t_0$,}\\[2.5ex]
  \parbox[t]{0.8\textwidth}{$\Omega^{s_0}\cdot\beta_0+\dots+\Omega^{s_i}\cdot\beta_i+\Omega^{t_0}\cdot\gamma_0+\dots+\Omega^{t_m}\cdot\gamma_m$\newline \hspace*{3cm}if $s_{i+1}<t_0<s_i$,}\\[2.5ex]
  \parbox[t]{0.8\textwidth}{$\Omega^{s_0}\cdot\beta_0+\dots+\Omega^{s_n}\cdot\beta_n+\Omega^{t_0}\cdot\gamma_0+\dots+\Omega^{t_m}\cdot\gamma_m$\newline \hspace*{3cm}if $t_0<s_n$.}
 \end{cases}
 \end{multlined}
\end{gather*}%

In the last case distinction, observe that we have $\beta_i+\gamma_0<\omega^\alpha$ because $\omega^\alpha$ is additively closed. For~$\alpha>0$ the term system $\varepsilon(S_{\omega^\alpha}^u)$ is thus closed under addition. As in the usual ordinal notation systems (see e.g.~\cite[Section~V.14]{schuette77}) one can verify the expected relations:

\begin{lemma}
 Assume $\alpha>0$. The following holds for all $r,s,t\in\varepsilon(S_{\omega^\alpha}^u)$:
\begin{enumerate}[label=(\roman*)]
 \item $s\leq\Omega^s$, and $s<t$ implies $\Omega^s<\Omega^t$,
 \item if $t<t'$ then $s+t<s+t'$ and $t+s\leq t'+s$,
 \item $s+(t+r)=(s+r)+t$,
 \item if $s<\Omega^t$ then $s+\Omega^t=\Omega^t$,
 \item if $s\leq t$ then we have $t=s+r$ for some $r$.
\end{enumerate}
\end{lemma}

As promised, we can now show that well-orderedness is preserved:

\begin{lemma}\label{lem:eps-preserves-wo}
 If $(S_{\omega^\alpha}^u,<_{S_{\omega^\alpha}^u})$ is a well-ordering then so is $(\varepsilon(S_{\omega^\alpha}^u),<_{\varepsilon(S_{\omega^\alpha}^u)})$, for each $\alpha>0$.
\end{lemma}
Concerning the restriction $\alpha>0$, we will later show that $\varepsilon(S_{\omega^0}^u)$ is a sub-ordering of $\varepsilon(S_{\omega^1}^u)$. Thus the well-foundedness of $\varepsilon(S_{\omega^0}^u)$ is also covered.
\begin{proof}
 We adapt the argument from \cite[Lemma~VIII.5]{schuette77} to our context (cf.~also the result of \cite{rathjen-afshari}): As usual, well-foundedness is equivalent to induction, i.e.~it suffices to establish
\begin{equation*}
   \forall_{s\in\varepsilon(S_{\omega^\alpha}^u)}(\forall_{t\in\varepsilon(S_{\omega^\alpha}^u)}(t<_{\varepsilon(S_{\omega^\alpha}^u)}s\rightarrow t\in a)\rightarrow s\in a)\rightarrow S_{\omega^\alpha}^u\subseteq a                                                                                                                                                                                                                                                                                                         \end{equation*}
for an arbitrary set $a$. Let us abbreviate
\begin{equation*}
 \prog(a):\equiv\forall_{s\in\varepsilon(S_{\omega^\alpha}^u)}(\forall_{t\in\varepsilon(S_{\omega^\alpha}^u)}(t<_{\varepsilon(S_{\omega^\alpha}^u)}s\rightarrow t\in a)\rightarrow s\in a).
\end{equation*}
In the rest of the proof we will drop the subscript of $<_{\varepsilon(S_{\omega^\alpha}^u)}$. Quantifiers with bound variable $s,t$ or $r$ are always restricted to $\varepsilon(S_{\omega^\alpha}^u)$. It is easy to see that any term in $\varepsilon(S_{\omega^\alpha}^u)$ is smaller than a term of the form~$\Omega_n^{\varepsilon_\sigma+1}$. Thus it will be enough to show
\begin{equation*}
 \prog(a)\rightarrow\forall_{r<\Omega_n^{\varepsilon_\sigma+1}}\,r\in a
\end{equation*}
for all $\sigma\in S_{\omega^\alpha}^u$ and $n\in\omega$. We want to argue by induction on $(\sigma,n)\in S_{\omega^\alpha}^u\times\omega$, ordered alphabetically. The latter is a well-ordering since $S_{\omega^\alpha}^u$ is well-ordered by assumption. Now, induction over any well-ordering is available if the induction statement is primitive recursive (and thus, by separation, corresponds to a subset of the well-ordering). For fixed $a$ the statement above is indeed primitive recursive; however it ceases to be primitive recursive if we quantify over all subsets \mbox{$a\subseteq S_{\omega^\alpha}^u$}. To gain flexibility while keeping the induction statement primitive recursive we introduce the primitive recursive ``jump'' function
\begin{align*}
 J(0,a)&:=a,\\
 J(m+1,a)&:=\{s\in S_{\omega^\alpha}^u\,|\,\forall_r(\forall_{t<r}\,t\in J(m,a)\rightarrow\forall_{t<r+\Omega^s}\,t\in J(m,a))\}.
\end{align*}
Let us verify an auxiliary result that we will need later, namely the implication
\begin{equation*}
 \prog( J(m,a))\rightarrow\prog( J(m+1,a)).
\end{equation*}
Aiming at $\prog( J(m+1,a))$ we fix $s\in\varepsilon(S_{\omega^\alpha}^u)$ and assume $\forall_{t<s}\,t\in J(m+1,a)$. Our goal is to establish $s\in J(m+1,a)$, which is equivalent to
\begin{equation*}
 \forall_r(\forall_{t<r}\,t\in J(m,a)\rightarrow\forall_{t<r+\Omega^s}\,t\in J(m,a)).
\end{equation*}
 If $s=0$ this is easy: From $\forall_{t<r}\,t\in J(m,a)$ and the assumption $\prog( J(m,a))$ we get $r\in J(m,a)$, so that we have $\forall_{t<r+\Omega^0}\,t\in J(m,a)$. In case $s>0$ any $t<r+\Omega^s$ is smaller than some term $r+\Omega^{s_0}\cdot\beta$, with $s_0<s$ and $\beta<\omega^\alpha$. We establish $\forall_{t<r+\Omega^{s_0}\cdot\beta}\,t\in J(m,a)$ by induction on $\beta$. For $\beta=0$ it suffices to cite the assumption $\forall_{t<r}\,t\in J(m,a)$. If $\beta$ is a limit ordinal then any $t<r+\Omega^{s_0}\cdot\beta$ is smaller than $r+\Omega^{s_0}\cdot\beta_0$ for some $\beta_0<\beta$, and the induction step is immediate. Now assume that $\beta=\beta_0+1$ is a successor. As $s_0<s$, one of the assumptions above provides $s_0\in J(m+1,a)$, which implies
\begin{equation*}
 \forall_{t<r+\Omega^{s_0}\cdot\beta_0}\,t\in J(m,a)\rightarrow\forall_{t<r+\Omega^{s_0}\cdot\beta_0+\Omega^{s_0}}\,t\in J(m,a).
\end{equation*}
In view of $r+\Omega^{s_0}\cdot\beta_0+\Omega^{s_0}=r+\Omega^{s_0}\cdot\beta$ this completes the induction step. After these preparations, let us prove
\begin{equation*}
 \forall_{m\in\omega}(\prog( J(m,a))\rightarrow\forall_{t<\Omega_n^{\varepsilon_\sigma+1}}\,t\in J(m,a))
\end{equation*}
by induction on $(\sigma,n)$. As observed above the instance $m=0$ suffices to establish the lemma; the other instances are required to perform the induction, while keeping $a$ fixed and the statement primitive recursive. In the induction step we assume $\prog(J(m,a))$ for some $m$ and consider an arbitrary $t<\Omega_n^{\varepsilon_\sigma+1}$. First assume $n=0$, such that we have $\Omega_n^{\varepsilon_\sigma+1}=\varepsilon_\sigma+1$. If $t<\varepsilon_\sigma$ then we have $t<\Omega_k^{\varepsilon_\tau+1}$ for some $\tau<_{S_{\omega^\alpha}^u}\sigma$ and some $k\in\omega$. So $t\in J(m,a)$ holds by the induction hypothesis. Having shown this much, we can conclude $\varepsilon_\sigma\in J(m,a)$ by $\prog( J(m,a))$. Together we have established $t\in J(m,a)$ for all $t<\varepsilon_\sigma+1=\Omega_0^{\varepsilon_\sigma+1}$, as required. Now assume $n>1$ and write $n=k+1$. Our auxiliary result provides $\prog(J(m+1,a))$ and the induction hypothesis yields $\forall_{r<\Omega_k^{\varepsilon_\sigma+1}}\,r\in J(m+1,a)$. Using $\prog(J(m+1,a))$ again we get $\Omega_k^{\varepsilon_\sigma+1}\in J(m+1,a)$, which is equivalent to
\begin{equation*}
 \forall_r(\forall_{t<r}\,t\in J(m,a)\rightarrow\forall_{t<r+\Omega_n^{\varepsilon_\sigma+1}}\,t\in J(m,a)).
\end{equation*}
With $r=0$ the antecedent is trivial and we get $\forall_{t<\Omega_n^{\varepsilon_\sigma+1}}\,t\in J(m,a)$ as desired.
\end{proof}

To obtain a well-ordering principle in the sense of Definition~\ref{def:wop} we must also verify compatibility:

\begin{lemma}
 For $\alpha<\beta$ we have $\varepsilon(S_{\omega^\alpha}^u)\subseteq\varepsilon(S_{\omega^\beta}^u)$, and $<_{\varepsilon(S_{\omega^\alpha}^u)}$ is the restriction of $<_{\varepsilon(S_{\omega^\beta}^u)}$ to $\varepsilon(S_{\omega^\alpha}^u)$. Also, we have $\varepsilon(S_{\omega^\lambda}^u)=\bigcup_{\gamma<\lambda}\varepsilon(S_{\omega^\gamma}^u)$ for each limit $\lambda$.
\end{lemma}
\begin{proof}
 Recall the auxiliary set $\varepsilon^0(S_{\omega_\alpha}^u)\supseteq\varepsilon(S_{\omega^\alpha}^u)$ discussed just after Definition~\ref{def:epsilon-of-ordering}. For $s,t\in\varepsilon^0(S_{\omega_\alpha}^u)$ one verifies
\begin{equation*}
 s\in\varepsilon(S_{\omega^\alpha}^u)\quad\Leftrightarrow\quad s\in\varepsilon(S_{\omega^\beta}^u)
\end{equation*}
and
\begin{equation*}
 s<_{\varepsilon(S_{\omega^\alpha}^u)}t\quad\Leftrightarrow\quad s<_{\varepsilon(S_{\omega^\beta}^u)}t
\end{equation*}
by simultaneous induction on $|s|$ resp.~$|s|+|t|$. The base of the induction relies on the fact that the search trees $S_\gamma^u$ are compatible, as shown in Lemma~\ref{lem:search-trees-compatible}. The induction step is straightforward. To save some work, observe that it suffices to establish the implication ``$\Rightarrow$'' in the second biconditional, as we already know that both orderings are linear. We have thus established
\begin{equation*}
 \varepsilon(S_{\omega^\alpha}^u)=\varepsilon(S_{\omega^\beta}^u)\cap\varepsilon^0(S_{\omega^\alpha}^u)
\end{equation*}
and ${<_{\varepsilon(S_{\omega^\alpha}^u)}}={<_{\varepsilon(S_{\omega^\beta}^u)}}\cap(\varepsilon(S_{\omega^\alpha}^u)\times\varepsilon(S_{\omega^\alpha}^u))$. The remaining claim about limit ordinals is reduced to the inclusion
\begin{equation*}
 \varepsilon^0(S_{\omega^\lambda}^u)\subseteq\bigcup_{\gamma<\lambda}\varepsilon^0(S_{\omega^\gamma}^u),
\end{equation*}
which is readily verified by induction on the length $|t|$ of a term $t\in\varepsilon^0(S_{\omega^\lambda}^u)$. Again this relies on Lemma~\ref{lem:search-trees-compatible}, i.e.~the corresponding statement for search trees.
\end{proof}

Finally, we need to construct a rank function:

\begin{lemma}
 There is a primitive recursive function $(u,s)\mapsto |s|_{\varepsilon(S_{\omega^\cdot})}^u$ such that we have
\begin{equation*}
 |s|_{\varepsilon(S_{\omega^\cdot})}^u=\begin{cases}
                            \min\{\alpha\in\ordi\,|\,s\in\varepsilon(S_{\omega^{\alpha+1}}^u)\} & \text{if such an $\alpha$ exists},\\
                            \{1\} & \text{otherwise}.
                           \end{cases}
\end{equation*}
\end{lemma}
\begin{proof}
As in the previous proof we use the auxiliary sets $\varepsilon^0(S_{\omega^{\alpha+1}}^u)\supseteq\varepsilon(S_{\omega^{\alpha+1}}^u)$ introduced after Definition~\ref{def:epsilon-of-ordering}. These sets consist of simple terms, built from ``constants'' $\varepsilon_\sigma$ and $\beta$ with $\sigma\in S_{\omega^{\alpha+1}}^u$ and $1<\beta<\omega^{\alpha+1}$, respectively. Clearly we can check whether a given set $s$ represents such a simple term, and if it does we can extract the constants it contains. Next we must check whether these constants have the required form: Lemma~\ref{lem:search-trees-rank} provides a primitive recursive rank function for the search trees. Given an alleged constant $\varepsilon_\sigma$ we can thus determine whether~$\sigma$ lies in some search tree $S_{\gamma+1}^u$, and we can compute the minimal such $\gamma$ if it does. Bounded minimization gives $\alpha=\min\{\beta\leq\gamma\,|\,\gamma<\omega^{\beta+1}\}$, which is minimal with~$\sigma\in S_{\omega^{\alpha+1}}^u$. Doing this for all constants in the term $s$ we get the minimal $\alpha$ with $s\in\varepsilon^0(S_{\omega^{\alpha+1}}^u)$, or we come to decide that no such $\alpha$ exists. Once $\alpha$ is computed we check whether $s$ lies in $\varepsilon(S_{\omega^{\alpha+1}}^u)\subseteq\varepsilon^0(S_{\omega^{\alpha+1}}^u)$: If it does, set $|s|_{\varepsilon(S_{\omega^\cdot})}^u=\alpha$, otherwise $|s|_{\varepsilon(S_{\omega^\cdot})}^u=\{1\}$.
\end{proof}

Combining these lemmata with the results of the previous section we obtain the following:

\begin{proposition}\label{prop:no-admissible-gives-wop}
 Consider a countable transitive set $u=\{u_i\,|\,i\in\omega\}$. If there is no admissible set $\mathbb A$ with $u\subseteq\mathbb A$ then $\alpha\mapsto\varepsilon(S_{\omega^\alpha}^u)$ is a well-ordering principle.
\end{proposition}
\begin{proof}
 The previous lemmata show that $(\varepsilon(S_{\omega^\alpha}^u),<_{\varepsilon(S_{\omega^\alpha}^u)})$ are compatible linear orderings with a rank function. Now assume that $u$ is not contained in a transitive model of Kripke-Platek set theory. By Corollary~\ref{cor:branch-to-admissible-set} this implies that none of the search trees $S_{\omega^{\alpha}}^u$ has an infinite branch. Using Lemma~\ref{lem:branch-Kleene-Brouwer} we conclude that $S_{\omega^{\alpha}}^u$ is well-ordered by $<_{S_{\omega^\alpha}^u}$. Lemma \ref{lem:eps-preserves-wo} tells us that $(\varepsilon(S_{\omega^\alpha}^u),<_{\varepsilon(S_{\omega^\alpha}^u)})$ is a well-ordering for each $\alpha>1$. As $\varepsilon(S_{\omega^0}^u)\subseteq\varepsilon(S_{\omega^1}^u)$ the case $\alpha=0$ is covered as well.
\end{proof}

If $\alpha\mapsto\varepsilon(S_{\omega^\alpha}^u)$ is a well-ordering principle then the higher Bachmann-Howard principle yields a collapsing function $\vartheta:\varepsilon(S_{\omega^\alpha}^u)\bh\alpha$, for some ordinal $\alpha$. In the rest of this section we establish properties of such a collapse. To get started, observe that $0\in\varepsilon(S_{\omega^\alpha}^u)$ implies $\vartheta(0)\in\alpha$ and thus $\alpha>0$. So we have $1<\omega^1\leq\omega^\alpha$, which means that $\varepsilon(S_{\omega^\alpha}^u)$ contains the terms $1:=\Omega^0\cdot 1$ and $\Omega:=\Omega^1\cdot 1$. We observe the following:

\begin{lemma}\label{lem:embed-alpha-s-alpha}
The map
\begin{equation*}
 \beta\mapsto\hat\beta:=\begin{cases}
                         0\quad&\text{if $\beta=0$},\\
                         \Omega^0\cdot\beta\quad&\text{otherwise}
                        \end{cases}
\end{equation*}
is an order isomorphism between $\omega^\alpha$ and $\varepsilon(S_{\omega^\alpha}^u)\cap\Omega:=\{t\in\varepsilon(S_{\omega^\alpha}^u)\,|\,t<_{\varepsilon(S_{\omega^\alpha}^u)}\Omega\}$.
\end{lemma}
\begin{proof}
It is clear from the definition of $<_{\varepsilon(S_{\omega^\alpha}^u)}$ that $\beta\mapsto\hat\beta$ is an order embedding. Also, we can infer $\Omega^0\cdot\beta<_{\varepsilon(S_{\omega^\alpha}^u)}\Omega$ from $0<_{\varepsilon(S_{\omega^\alpha}^u)}1$. It remains to check that any term $t<_{\varepsilon(S_{\omega^\alpha}^u)}\Omega$ is of the form $t=0$ or $t=\Omega^0\cdot\beta$. First, aiming at a contradiction, assume that $t$ is of the form $\varepsilon_\sigma$. Then $t<_{\varepsilon(S_{\omega^\alpha}^u)}\Omega$ would imply $\varepsilon_\sigma<_{\varepsilon(S_{\omega^\alpha}^u)}1$ and thus $\varepsilon_\sigma<_{\varepsilon(S_{\omega^\alpha}^u)}0$, which is false. Now assume that $t<_{\varepsilon(S_{\omega^\alpha}^u)}\Omega$ is of the form
\begin{equation*}
 t=\Omega^{t_0}\cdot\gamma_0+\dots+\Omega^{t_m}\cdot\gamma_m.
\end{equation*}
We cannot have $t_0=1$ as this would require $\gamma_0<1$, which was not allowed in our term system. Thus we must have $t_0<_{\varepsilon(S_{\omega^\alpha}^u)}1$, which is easily seen to imply $t_0=0$. The latter makes $t_1<_{\varepsilon(S_{\omega^\alpha}^u)}t_0$ impossible, so that we see $m=0$ and $t=\Omega^0\cdot\gamma_0$.
\end{proof}

As we will argue in the same context for quite a while it is worth introducing some abbreviations:

\begin{notation}\label{notation:abs-and-star}
In view of the previous lemma we will write $\beta$ instead of $\hat\beta$ and $<$ instead of $<_{\varepsilon(S_{\omega^\alpha}^u)}$. Concerning the rank functions $|\cdot|_{\mathbf L}^u$ and $|\cdot|_S^u$ we will omit the sub- and superscript, writing $|a|$ resp.~$|\sigma|$ instead of $|a|_{\mathbf L}^u$ resp.~$|\sigma|_S^u$. This is harmless because these two rank functions are closely connected: For example we have $|a|_{\mathbf L}^u=|\langle a\rangle|_S^u$. The rank function $|\cdot|_{\varepsilon(S_{\omega^\cdot})}^u$ behaves quite differently and must be carefully distinguished: To make this visual we write $s^*$ at the place of $|s|_{\varepsilon(S_{\omega^\cdot})}^u$ (this is inspired by Rathjen's ordinal notation system for the Bachmann-Howard ordinal in \cite{rathjen-model-bi}). To get some intuition for the notation, let us consider two ways to view an ordinal $\beta<\omega^\alpha$: First, one can view $\beta$ as an element of $\mathbf L_{\omega^\alpha}^u$, with
\begin{equation*}
|\beta|=|\beta|_{\mathbf L}^u=\min\{\gamma\in\ordi\,|\,\beta\in\mathbb L_{\gamma+1}^u\}\leq\beta.
\end{equation*}
Note that inequality is possible, e.g.~if we have $\beta\in u$. Alternatively, one can view~$\beta$ as the term $\hat\beta$ in $\varepsilon(S_{\omega^\alpha}^u)$. Then we have
\begin{equation*}
\beta^*=|\beta|_{\varepsilon(S_{\omega^\cdot})}^u=\min\{\gamma\in\ordi\,|\,\beta<\omega^{\gamma+1}\}\leq\beta,
\end{equation*}
which implies $\omega^{\beta^*}\leq\beta<\omega^{\beta^*+1}$ in case $\beta>0$.
\end{notation}

The reader may wonder why we consider the well-ordering principle $\alpha\mapsto\varepsilon(S_{\omega^\alpha}^u)$ rather than $\alpha\mapsto\varepsilon(S_\alpha^u)$. The following two results should make this clear:

\begin{lemma}\label{lem:bh-epsilon-number}
 Assume $\vartheta:\varepsilon(S_{\omega^\alpha}^u)\bh\alpha$. Then we have $\beta\leq\vartheta(\beta)$ for each $\beta<\omega^\alpha$. In particular this implies $\alpha=\omega^\alpha$, i.e.~the ordinal $\alpha$ must be an $\varepsilon$-number.
\end{lemma}
\begin{proof}
 Clearly $\gamma<\beta<\omega^\alpha$ implies $\gamma^*\leq\beta^*$. By the definition of Bachmann-Howard collapse we conclude $\gamma^*\leq\beta^*<\vartheta(\beta)$ and then $\vartheta(\gamma)<\vartheta(\beta)$. Now $\beta\leq\vartheta(\beta)$ is established by the usual induction on $\beta$: We have
\begin{equation*}
 \vartheta(\beta)\geq\sup\{\vartheta(\gamma)+1\,|\,\gamma<\beta\}\geq\sup\{\gamma+1\,|\,\gamma<\beta\},
\end{equation*}
which implies the induction step in both successor and limit case.
\end{proof}

\begin{proposition}\label{prop:bh-epsilon-numbers}
 Assume $\vartheta:\varepsilon(S_{\omega^\alpha}^u)\bh\alpha$. For any $t\in\varepsilon(S_{\omega^\alpha}^u)$ with $\Omega\leq t$ the ordinal $\vartheta(t)$ is additively principal and bigger than $\omega$.
\end{proposition}
\begin{proof}
From $\beta<\omega^{\beta^*+1}$ and $\gamma<\omega^{\gamma^*+1}$ we get $\beta+\gamma<\omega^{\max\{\beta^*,\gamma^*\}+1}$ and thus
\begin{equation*}
 (\beta+\gamma)^*\leq\max\{\beta^*,\gamma^*\}.
\end{equation*}
Also, $\beta<\omega^{\beta+1}$ yields $\beta^*\leq\beta$. Thus $\beta,\gamma<\vartheta(t)$ implies
\begin{equation*}
 (\beta+\gamma)^*\leq\max\{\beta^*,\gamma^*\}\leq\max\{\beta,\gamma\}<\vartheta(t).
\end{equation*}
As $\omega^\alpha$ is additively principal we can form the term $\beta+\gamma=\Omega^0\cdot(\beta+\gamma)$. From $\beta+\gamma<\Omega\leq t$ and the above we infer $\vartheta(\beta+\gamma)<\vartheta(t)$. Together with the previous lemma we obtain $\beta+\gamma<\vartheta(t)$, as needed to show that $\vartheta(t)$ is additively principal. As for $\omega<\vartheta(t)$, note first that we have $0^*=0\leq t^*<\vartheta(t)$ and thus $0<\vartheta(0)<\vartheta(t)$, which means $1<\vartheta(t)$. In view of $\omega^*=1$ this implies $\vartheta(\omega)<\vartheta(t)$. Together with the previous lemma we get $\omega<\vartheta(t)$.
\end{proof}

Note that the condition $\Omega\leq t$ in the statement of the proposition is necessary: We could indeed set $\vartheta(\beta):=f(\beta)$ for any strictly increasing function $f:\aleph_1\rightarrow\aleph_1$ with $f(\beta)>\beta$, and then use the construction from Remark~\ref{rmk:Bachmann-Howard-semantically} to extend~$\vartheta$ to a Bachmann-Howard collapse $\varepsilon(S_{\omega^{\aleph_1}}^u)\rightarrow\aleph_1$. The value $\vartheta(\Omega)$ cannot be prescribed in the same way because a strictly increasing function $f:\aleph_1+1\rightarrow\aleph_1$ does not exist. It follows that the values $\vartheta(\beta)$ for $\beta<\Omega$ are not informative at all, in contrast to the usual notation systems for the Bachmann-Howard ordinal. This will not be a problem, however, as we can work with the ordinals $\vartheta(\Omega^{1+\beta})$ instead. Next, recall the above definition of exponentiation and addition for terms in $\varepsilon(S_{\omega^\alpha}^u)$. Crucially, these operations behave nicely with respect to the ranks:

\begin{lemma}\label{lem:exp-add-ranks}
 The following holds for all $s,t\in\varepsilon(S_{\omega^\alpha}^u)$:
\begin{enumerate}[label=(\roman*)]
 \item $(\Omega^s)^*\leq s^*$,
 \item $(s+t)^*\leq\max\{s^*,t^*\}$.
\end{enumerate}
\end{lemma}
\begin{proof}
 The argument is essentially the same for both claims, so we only consider part~(ii): Set $\delta:=\max\{s^*,t^*\}$. By definition of the rank we have $s,t\in\varepsilon(S_{\omega^{\delta+1}}^u)$. We have already observed that $\varepsilon(S_{\omega^{\delta+1}}^u)$ is closed under addition, i.e.~we also have $s+t\in\varepsilon(S_{\omega^{\delta+1}}^u)$. This implies $(s+t)^*\leq\delta$ by the minimality of the rank.
\end{proof}

The following notions will be of central importance in the next sections, where we extend the search trees $S_{\omega^\alpha}^u$ to proof trees and analyse them by proof-theoretic methods:

\begin{definition}\label{def:bh-to-operators}
 Assume $\vartheta:\varepsilon(S_{\omega^\alpha}^u)\bh\alpha$. For each $t\in\varepsilon(S_{\omega^\alpha}^u)$ with $\Omega\leq t$ and each subset $X\subseteq\omega^\alpha$ we define a set $C^\vartheta(t,X)\subseteq\varepsilon(S_{\omega^\alpha}^u)$: Put \begin{equation*}
 C^\vartheta(t,X):=\bigcup_{n\in\omega}C_n^\vartheta(t,X)
 \end{equation*}
 where $C_n^\vartheta(t,X)$ is inductively defined by
\begin{align*}
C_0^\vartheta(t,X)&=X\cup\{0\},\\
C_{n+1}^\vartheta(t,X)&=\begin{aligned}[t]C_{n}^\vartheta(t,X)&\cup\{s\in\varepsilon(S_{\omega^\alpha}^u)\,|\,s^*\in C_{n}^\vartheta(t,X)\}\\ &\cup\{\vartheta(s)\,|\,s\in C_{n}^\vartheta(t,X)\text{ and }s<t\}\\ &\cup\{s\,|\,s<s'\text{ for some }s'\in C_{n}^\vartheta(t,X)\cap\Omega\}.\end{aligned}
\end{align*} 
\end{definition}

Here we have abbreviated $C_{n}^\vartheta(t,X)\cap\Omega=\{s'\in C_{n}^\vartheta(t,X)\,|\,s'<\Omega\}$. The reader may wish to recall that $s^*$ and $\vartheta(s)$ are elements of $\omega^\alpha\cong\varepsilon(S_{\omega^\alpha}^u)\cap\Omega$, for any term $s\in\varepsilon(S_{\omega^\alpha}^u)$. Clearly $C^\vartheta(t,X)$ is primitive recursive in $\vartheta,t,X,\alpha$ and (the fixed enumeration~of)~$u$. Let us show some basic properties:

\begin{lemma}\label{lem:basic-c-sets}
 Assume $\vartheta:\varepsilon(S_{\omega^\alpha}^u)\bh\alpha$. Then the following holds:
\begin{enumerate}[label=(\roman*)]
 \item If $\Omega\leq t<t'$ then $C^\vartheta(t,X)\subseteq C^\vartheta(t',X)$.
 \item If $X\subseteq C^\vartheta(t,X')\cap\Omega$ then $C^\vartheta(t,X)\subseteq C(t,X')$.
 \item If $s\in C^\vartheta(t,X)$ then $s^*\in C^\vartheta(t,X)$.
\end{enumerate}
\end{lemma}
\begin{proof}
 (i) It is straightforward to establish $C_n^\vartheta(t,X)\subseteq C(t',X)$ by induction on $n$.\\
(ii) Note that $C^\vartheta(t,X)$ is defined because of $X\subseteq\varepsilon(S_{\omega^\alpha}^u)\cap\Omega\cong\omega^\alpha$. A straightforward induction on $n$ shows $C_n^\vartheta(t,X)\subseteq C(t,X')$.\\
(iii) Let us first establish the claim in the special case $s<\Omega$: In the context of Notation~\ref{notation:abs-and-star} we have observed $s^*\leq s$. Thus $s\in C^\vartheta(t,X)\cap\Omega$ implies $s^*\in C^\vartheta(t,X)$, by one of the closure properties of $C^\vartheta(t,X)$. As for the general case, let us show by induction on $n$ that $s\in C_n^\vartheta(t,X)$ implies $s^*\in C^\vartheta(t,X)$: For $n=0$ this reduces to the special case $s<\Omega$, because of $X\cup\{0\}\subseteq C^\vartheta(t,X)\cap\Omega$. Concerning the induction step, if $s\in C_{n+1}^\vartheta(t,X)$ holds because of $s^*\in C_n^\vartheta(t,X)$ then the claim is immediate. In all other cases we have $s<\Omega$, and the claim holds as before.
\end{proof}

The following recovers the usual construction of the Bachmann-Howard ordinal (cf.~Remark~\ref{rmk:Bachmann-Howard-semantically} above):

\begin{lemma}\label{lem:recover-c-sets}
 Assume $\vartheta:\varepsilon(S_{\omega^\alpha}^u)\bh\alpha$. For each $t\in\varepsilon(S_{\omega^\alpha}^u)$ with $\Omega\leq t$ we have
\begin{equation*}
 C^\vartheta(t,\vartheta(t))\cap\Omega=\vartheta(t).
\end{equation*}
\end{lemma}
\begin{proof}
First, $\vartheta(t)\subseteq C^\vartheta(t,\vartheta(t))\cap\Omega$ is immediate by the definition of $C_0^\vartheta(t,\vartheta(t))$. To establish the converse inclusion we prove
\begin{equation*}
C_n^\vartheta(t,\vartheta(t))\cap\Omega\subseteq\vartheta(t)
\end{equation*}
by induction on $n$. Simultaneously we show that $s\in C_n^\vartheta(t,\vartheta(t))$ implies $s^*<\vartheta(t)$. For $n=0$ the inclusion $\vartheta(t)\cup\{0\}\subseteq\vartheta(t)$ amounts to $0<\vartheta(t)$, which holds by Proposition~\ref{prop:bh-epsilon-numbers}. Also $s<\vartheta(t)$ implies $s^*\leq s<\vartheta(t)$, as in the proof of the previous lemma. In the induction step we distinguish several cases: First, assume that $s\in C_{n+1}^\vartheta(t,\vartheta(t))$ holds because of $s^*\in C_n^\vartheta(t,\vartheta(t))$. In view of $s^*<\Omega$ the induction hypothesis yields $s^*<\vartheta(t)$. In case $s<\Omega$ we also have $s<\Omega\leq t$. Then $\vartheta(s)<\vartheta(t)$ follows by the definition of Bachmann-Howard collapse. Together with Lemma~\ref{lem:bh-epsilon-number} we obtain $s\leq\vartheta(s)<\vartheta(t)$, as required for the induction step. Next, assume $s=\vartheta(s_0)$ with $s_0\in C_n^\vartheta(t,X)$ and $s_0<t$. The induction hypothesis provides ${s_0}^*<\vartheta(t)$, which implies $s=\vartheta(s_0)<\vartheta(t)$. Also, because of $s=\vartheta(s_0)<\Omega$ we have $s^*\leq s<\vartheta(t)$. Finally, assume that we have $s<s'$ for some $s'\in C_n^\vartheta(t,X)\cap\Omega$. By induction hypothesis we have $s'<\vartheta(t)$. This implies $s^*\leq s<s'<\vartheta(t)$, as required.
\end{proof}

For a proof-theoretic analysis of the Kripke-Platek axioms it is convenient to use the following ``controlling operators". This formalism is due to Buchholz \cite{buchholz-local-predicativity}:

\begin{definition}\label{def:operators}
Assume $\vartheta:\varepsilon(S_{\omega^\alpha}^u)\bh\alpha$. For a term $t\in\varepsilon(S_{\omega^\alpha}^u)$ with $\Omega\leq t$ and ordinals $\beta_1,\dots,\beta_k<\omega^\alpha$ we write
\begin{equation*}
\mathcal H_t^\vartheta[\beta_1,\dots ,\beta_k]:=C^\vartheta(t+1,\{\beta_1,\dots,\beta_n\})
\end{equation*}
We will write $\mathcal H_t^\vartheta$ rather than $\mathcal H_t^\vartheta[\,]$ (i.e.~in case $k=0$).
\end{definition}

Occasionally it is helpful to think of the operator $\mathcal H_t^\vartheta[\beta_1,\dots ,\beta_k]$ as the function
\begin{equation*}
\gamma\mapsto\mathcal H_t^\vartheta[\beta_1,\dots ,\beta_k,\gamma],
\end{equation*}
which maps an ordinal $\gamma<\omega^\alpha$ to a subset of $\varepsilon(S_{\omega^\alpha}^u)$. Note that this function exists as a set, as it is the restriction of a primitive recursive class function to a set. Our next goal is to recover properties of Buchholz' operators in our context:

\begin{lemma}\label{lem:operators-closure}
Assume $\vartheta:\varepsilon(S_{\omega^\alpha}^u)\bh\alpha$. The operators $\mathcal H_t^\vartheta[\beta_1,\dots ,\beta_k]$ are ``closure operators", in the sense that we have
\begin{enumerate}[label=(\roman*)]
\item $\{\beta_1,\dots ,\beta_k\}\subseteq\mathcal H_t^\vartheta[\beta_1,\dots ,\beta_k]$,
\item if $\{\gamma_1,\dots ,\gamma_l\}\subseteq\mathcal H_t^\vartheta[\beta_1,\dots ,\beta_k]$ then $\mathcal H_t^\vartheta[\gamma_1,\dots ,\gamma_l]\subseteq\mathcal H_t^\vartheta[\beta_1,\dots ,\beta_k]$.
\end{enumerate}
\end{lemma}
\begin{proof}
Part (i) is immediate by definition, and (ii) holds by Lemma~\ref{lem:basic-c-sets}(ii).
\end{proof}

Let us record an easy consequence:

\begin{lemma}\label{lem:operators-max-counts}
 Assume $\vartheta:\varepsilon(S_{\omega^\alpha}^u)\bh\alpha$. Then the following holds:
\begin{enumerate}[label=(\roman*)]
\item If $s\in\mathcal H_t^\vartheta[\beta_1,\dots ,\beta_k]\cap\Omega$ and $s'<s$ then $s'\in\mathcal H_t^\vartheta[\beta_1,\dots ,\beta_k]$,
\item if $\max\{\gamma_1,\dots,\gamma_l\}\leq\max\{\beta_1,\dots,\beta_k\}$ then $\mathcal H_t^\vartheta[\gamma_1,\dots ,\gamma_l]\subseteq\mathcal H_t^\vartheta[\beta_1,\dots ,\beta_k]$.
\end{enumerate}
\end{lemma}
\begin{proof}
 Part (i) holds by the definition of $\mathcal H_t^\vartheta[\beta_1,\dots ,\beta_k]=C^\vartheta(t+1,\{\beta_1,\dots,\beta_n\})$. As for (ii), the previous lemma implies $\max\{\beta_1,\dots,\beta_k\}\in\mathcal H_t^\vartheta[\beta_1,\dots ,\beta_k]$. Then $\max\{\gamma_1,\dots,\gamma_l\}\leq\max\{\beta_1,\dots,\beta_k\}$ implies $\{\gamma_1,\dots ,\gamma_l\}\subseteq\mathcal H_t^\vartheta[\beta_1,\dots ,\beta_k]$, by~(i). The previous lemma yields $\mathcal H_t^\vartheta[\gamma_1,\dots ,\gamma_l]\subseteq\mathcal H_t^\vartheta[\beta_1,\dots ,\beta_k]$, as desired.
\end{proof}

\begin{lemma}\label{lem:operators-nice}
The operators $\mathcal H_t^\vartheta[\beta_1,\dots ,\beta_k]$ are ``nice", in the sense that we have
\begin{enumerate}[label=(\roman*)]
\item $\{0,\omega\}\subseteq\mathcal H_t^\vartheta[\beta_1,\dots ,\beta_k]$,
\item if $\{s,s'\}\subseteq\mathcal H_t^\vartheta[\beta_1,\dots ,\beta_k]$ then $\{\Omega^s,s+s'\}\subseteq\mathcal H_t^\vartheta[\beta_1,\dots ,\beta_k]$,
\item if $\sigma\in S_{\omega^\alpha}^u$ and $|\sigma|\in\mathcal H_t^\vartheta[\beta_1,\dots ,\beta_k]$ then $\varepsilon_\sigma\in\mathcal H_t^\vartheta[\beta_1,\dots ,\beta_k]$.
\end{enumerate}
\end{lemma}
\begin{proof}
(i) By the definition of $\mathcal H_t^\vartheta[\beta_1,\dots ,\beta_k]=C^\vartheta(t+1,\{\beta_1,\dots ,\beta_k\})$ we immediately get $0\in\mathcal H_t^\vartheta[\beta_1,\dots ,\beta_k]$. Now $1<\omega^{0+1}$ implies $1^*=0$, so that we obtain $1\in\mathcal H_t^\vartheta[\beta_1,\dots ,\beta_k]$. Similarly, $\omega^1\leq\omega<\omega^2$ implies $\omega^*=1$, from which we can infer $\omega\in\mathcal H_t^\vartheta[\beta_1,\dots ,\beta_k]$.\\
(ii) Assume that $\mathcal H_t^\vartheta[\beta_1,\dots ,\beta_k]$ contains $s$ and $s'$. By Lemma~\ref{lem:basic-c-sets} it contains $s^*$ and $(s')^*$. Using Lemma~\ref{lem:exp-add-ranks} and the previous lemma we infer that it contains $(\Omega^s)^*$ and $(s+t)^*$. Finally, $\{\Omega^s,s+t\}\subseteq\mathcal H_t^\vartheta[\beta_1,\dots ,\beta_k]$ follows by definition.\\
(iii) By definition of the rank we have $\sigma\in S_{|\sigma|+1}^u\subseteq S_{\omega^{|\sigma|+1}}^u$ and thus $\varepsilon_\sigma\in\varepsilon(S_{\omega^{|\sigma|+1}}^u)$. This means ${\varepsilon_\sigma}^*\leq|\sigma|$. Thus $|\sigma|\in\mathcal H_t^\vartheta[\beta_1,\dots ,\beta_k]$ implies ${\varepsilon_\sigma}^*\in\mathcal H_t^\vartheta[\beta_1,\dots ,\beta_k]$, and then $\varepsilon_\sigma\in\mathcal H_t^\vartheta[\beta_1,\dots ,\beta_k]$ as desired.
\end{proof}

The above are general conditions for ``reasonable" operators. Now we show a result that is specific to the operators $\mathcal H_t^\vartheta[\beta_1,\dots ,\beta_k]$:

\begin{proposition}\label{prop:operators-special}
 Assume $\vartheta:\varepsilon(S_{\omega^\alpha}^u)\bh\alpha$. The operators $\mathcal H_t^\vartheta[\beta_1,\dots ,\beta_k]$ have the following properties:
\begin{enumerate}[label=(\roman*),leftmargin=1cm]
\item If $\Omega\leq t<t'$ then $\mathcal H_t^\vartheta[\beta_1,\dots ,\beta_k]\subseteq\mathcal H_{t'}^\vartheta[\beta_1,\dots ,\beta_k]$,
\item if $s\in\mathcal H_t^\vartheta[\beta_1,\dots ,\beta_k]$ and $s\leq t$ then $\vartheta(s)\in\mathcal H_t^\vartheta[\beta_1,\dots ,\beta_k]$,
\item if $s\in\mathcal H_t^\vartheta\cap\Omega$ then $s<\vartheta(t+1)$,
\item if $\{s,t\}\subseteq\mathcal H_t^\vartheta$ and $s<s'$ then $\vartheta(t+\Omega^s)<\vartheta(t+\Omega^{s'})$.
\end{enumerate}
\end{proposition}
\begin{proof}
(i) This is a special case of Lemma~\ref{lem:basic-c-sets}.\\
(ii) Immediate by the closure properties of $\mathcal H_t^\vartheta[\beta_1,\dots ,\beta_k]=C^\vartheta(t+1,\beta_1,\dots ,\beta_k)$.\\
(iii) Using Lemma~\ref{lem:basic-c-sets}(ii) we get
\begin{equation*}
\mathcal H_t^\vartheta=C^\vartheta(t+1,\emptyset)\subseteq C^\vartheta(t+1,\vartheta(t+1)).
\end{equation*}
Then the claim follows from Lemma~\ref{lem:recover-c-sets}.\\
(iv) From $s<s'$ we get $t+\Omega^s<t+\Omega^{s'}$. To infer $\vartheta(t+\Omega^s)<\vartheta(t+\Omega^{s'})$ it remains to establish $(t+\Omega^s)^*<\vartheta(t+\Omega^{s'})$. Now $\{s,t\}\subseteq\mathcal H_t^\vartheta$ implies $t+\Omega^s\in\mathcal H_t^\vartheta=C^\vartheta(t+1,\emptyset)$. Using Lemma~\ref{lem:basic-c-sets} we get
\begin{equation*}
 (t+\Omega^s)^*\in C^\vartheta(t+1,\emptyset)\subseteq C^\vartheta(t+\Omega^{s'},\emptyset)\subseteq C^\vartheta(t+\Omega^{s'},\vartheta(t+\Omega^{s'})).
\end{equation*}
Together with $(t+\Omega^s)^*<\Omega$ Lemma~\ref{lem:recover-c-sets} yields $(t+\Omega^s)^*<\vartheta(t+\Omega^{s'})$, as needed.
\end{proof}

\section{From Search Tree to Proof Tree}\label{sect:proof-trees}

In Section \ref{sect:search-trees} we have built a ``search tree'' $S_\alpha^u$ for each countable transitive set $u=\{u_i\,|\,i\in\omega\}$ and each ordinal $\alpha$. As stated there, $S_\alpha^u$ can be seen as an attempted proof of a contradiction in $\mathbb L_\alpha^u$-logic, with the axioms of Kripke-Platek set theory as open assumptions. The goal of this section is to remove these assumptions, by adding infinitary proofs of the Kripke-Platek axioms. To begin, we give a reasonably general definition of infinitary proof trees, which we call $\mathbf L_{\omega^\alpha}^u$-preproofs. 

Recall that the search tree $S_{\omega^\alpha}^u$ is a subtree of $(\mathbf L_{\omega^\alpha}^u)^{<\omega}$, where each node is labelled by an $\mathbf L_{\omega^\alpha}^u$-sequent. The order of formulas in a sequent was crucial for the definition of the search trees, but it is inessential in the context of proof trees. We will thus identify a sequent with the set of its entries (e.g.\ $\Gamma\subseteq\Delta$ expresses that each entry of $\Gamma$ is also an entry of $\Delta$). In addition to sequents, $\mathbf L_{\omega^\alpha}^u$-preproofs will carry labels for ``rules'': By an $\mathbf L_{\omega^\alpha}^u$-rule we shall mean a symbol from the list
\begin{multline*}
 \ax,\quad (\land,\psi_0,\psi_1),\quad (\lor_i,\psi_0,\psi_1),\quad (\forall_x,\psi),\quad (\exists_x,a,\psi),\\
(\cut, \psi),\quad (\rref,\exists_z\forall_{x\in a}\exists_{y\in z}\theta),\quad (\rep, a), 
\end{multline*}
where $\psi_0$, $\psi_1$ and $\psi$ are $\mathbf L_{\omega^\alpha}^u$-formulas; we have $i\in\{0,1\}$; $a$ is an element of $\mathbf L_{\omega^\alpha}^u$; and $\theta$ is a bounded disjunction which does not contain the variable $z$. In addition to the rules, each node of an $\mathbf L_{\omega^\alpha}^u$-preproof will be labelled by an element of the term system $\varepsilon(S_{\omega^\alpha}^u)$, defined in the previous section. We do not assume that $\varepsilon(S_{\omega^\alpha}^u)$ is well-founded. For this reason the term ``preproof'' is better than ``proof'', even though we will occasionally use the latter for the sake of brevity; also, we will sometimes refer to the elements of $\varepsilon(S_{\omega^\alpha}^u)$ as ``ordinal labels''. As in the previous section we will write $<$ rather than $<_{\varepsilon(S_{\omega^\alpha}^u)}$ for the order relation on $\varepsilon(S_{\omega^\alpha}^u)$ (recall from Lemma~\ref{lem:embed-alpha-s-alpha} that $\omega^\alpha$ can be identified with an initial segment of this ordering).

\begin{definition}\label{def:L-proof}
 Consider a countable transitive set $u=\{u_i\,|\,i\in\omega\}\supseteq\{0,1\}$ and an ordinal $\alpha>1$. An $\mathbf L_{\omega^\alpha}^u$-preproof consists of a non-empty tree $P\subseteq(\mathbf L_{\omega^\alpha}^u)^{<\omega}$ and labelling functions \mbox{$l:P\rightarrow\text{``$\mathbf L_{\omega^\alpha}^u$-sequents''}$}, $r:P\rightarrow\text{``$\mathbf L_{\omega^\alpha}^u$-rules''}$ and \mbox{$o:P\rightarrow\varepsilon(S_{\omega^\alpha}^u)$} such that the following ``local correctness conditions'' hold at every node $\sigma\in P$:
{\def\arraystretch{1.75}\tabcolsep=10pt
\setlength{\LTpre}{\baselineskip}\setlength{\LTpost}{\baselineskip}
\begin{longtable}{lp{0.65\textwidth}}\hline
If $r(\sigma)$ is \dots & \dots\ then \dots\\ \hline
$\ax$ & $\sigma$ is a leaf of $P$ and $l(\sigma)$ contains a true $\Delta_0$-formula,\\
$(\land,\psi_0,\psi_1)$ & we have $\sigma^\frown a\in P$ iff $a\in\{0,1\}$, and $o(\sigma^\frown a)<o(\sigma)$; \newline also $\psi_0\land\psi_1\in l(\sigma)$ and $l(\sigma^\frown i)\subseteq l(\sigma),\psi_i$ for $i=0,1$,\\
$(\lor_i,\psi_0,\psi_1)$ & we have $\sigma^\frown a\in P$ iff $a=0$, and $o(\sigma^\frown 0)<o(\sigma)$;\newline also $\psi_0\lor\psi_1\in l(\sigma)$ and $l(\sigma^\frown 0)\subseteq l(\sigma),\psi_i$,\\
$(\forall_x,\psi)$ & we have $\sigma^\frown a\in P$ for all $a\in\mathbf L_{\omega^\alpha}^u$,\newline and $o(\sigma^\frown a)+\omega\leq o(\sigma)$;\newline also $\forall_x\psi\in l(\sigma)$ and $l(\sigma^\frown a)\subseteq l(\sigma),\psi(a)$ for each $a$,\\
$(\exists_x,b,\psi)$ & we have $\sigma^\frown a\in P$ iff $a=0$,\newline and $o(\sigma^\frown 0)+\omega\leq o(\sigma)$ and $|b|<o(\sigma)$;\newline also $\exists_x\psi\in l(\sigma)$ and $l(\sigma^\frown 0)\subseteq l(\sigma),\psi(b)$,\\
$(\cut, \psi)$ & we have $\sigma^\frown a\in P$ iff $a\in\{0,1\}$, and $o(\sigma^\frown a)<o(\sigma)$; \newline also $l(\sigma^\frown 0)\subseteq l(\sigma),\psi$ and $l(\sigma^\frown 1)\subseteq l(\sigma),\neg\psi$,\\
$(\rref,\exists_z\forall_{x\in a}\exists_{y\in z}\theta)$ & we have $\sigma^\frown a\in P$ iff $a=0$,\newline and $o(\sigma^\frown 0)<o(\sigma)$ and $\Omega\leq o(\sigma)$;\newline also $\exists_z\forall_{x\in a}\exists_{y\in z}\theta\in l(\sigma)$ and $l(s^\frown 0)\subseteq l(\sigma),\forall_{x\in a}\exists_y\theta$,\\
$(\rep, b)$ & we have $\sigma^\frown a\in P$ iff $a=b$, and $o(\sigma^\frown b)<o(\sigma)$;\newline also $l(\sigma^\frown b)\subseteq l(\sigma)$.\\ \hline
\end{longtable}}
We call $l(\langle\rangle)$, $r(\langle\rangle)$ and $o(\langle\rangle)$ the end-sequent, the last rule, and the height of the preproof $P$, respectively.
\end{definition}

Let us give $\mathbf L_{\omega^\alpha}^u$-preproofs of the Kripke-Platek axioms:

\begin{lemma}\label{lem:proofs-kp-axioms}
Consider a countable transitive $u$ and an ordinal $\alpha>1$. Except for foundation, each Kripke-Platek axiom has an \mbox{$\mathbf L_{\omega^\alpha}^u$-preproof} with height below $\Omega^1\cdot 2$.
\end{lemma}

\begin{proof}
 Let us start with the case of $\Delta_0$-separation, i.e.\ an axiom of the form
\begin{equation*}
 \forall_{v_1}\cdots\forall_{v_k}\forall_x\exists_y(\forall_{z\in y}(z\in x\land\theta(x,z,\vec v))\land\forall_{z\in x}(\theta(x,z,\vec v)\rightarrow z\in y))
\end{equation*}
with a $\Delta_0$-formula $\theta$. The nodes in the desired $\mathbf L_{\omega^\alpha}^u$-preproof will be precisely those of the form $\langle c_1,\dots,c_k,a,0\rangle$, where $\vec c,a\in\mathbf L_{\omega^\alpha}^u$ are arbitrary. Given such parameters, set $\gamma:=\max\{|c_1|,\dots,|c_k|,|a|\}+1$ and observe that $\vec c,a\in\mathbf L_\gamma^u$ holds by definition of the rank. We can primitive recursively compute the set 
\begin{equation*}
 b:=\{z\in a\,|\,\theta(a,z,\vec c)\}=\{z\in\mathbb L_\gamma\,|\,\mathbb L_\gamma\vDash\theta(a,z,\vec c)\}\in\mathbb L_{\gamma+\omega}^u
\end{equation*}
and the ordinal
\begin{equation*}
 |b|:=\min\{\beta<\gamma+\omega\,|\,b\in\mathbb L_{\beta+1}^u\}<\omega^\alpha,
\end{equation*}
which allows us to view $b$ as an element of $\mathbf L_{\omega^\alpha}^u$. By construction of $b$ the bounded $\mathbf L_{\omega^\alpha}^u$-formula
\begin{equation*}
 \forall_{z\in b}(z\in a\land\theta(a,z,\vec c))\land\forall_{z\in a}(\theta(a,z,\vec c)\rightarrow z\in b)
\end{equation*}
is true. Thus the leaf $\langle c_1,\dots,c_k,a,0\rangle$ can be labelled by
\begin{align*}
 l(\langle c_1,\dots,c_k,a,0\rangle)&=\langle\forall_{z\in b}(z\in a\land\theta(a,z,\vec c))\land\forall_{z\in a}(\theta(a,z,\vec c)\rightarrow z\in b)\rangle,\\
 r(\langle c_1,\dots,c_k,a,0\rangle)&=\ax,\\
 o(\langle c_1,\dots,c_k,a,0\rangle)&=0,
\end{align*}
in a locally correct way. Next, we can take $b$ as a witness for an existential quantifier. This amounts to setting
\begin{align*}
 l(\langle c_1,\dots,c_k,a\rangle)&=\langle\exists_y(\forall_{z\in y}(z\in a\land\theta(a,z,\vec c))\land\forall_{z\in a}(\theta(a,z,\vec c)\rightarrow z\in y))\rangle,\\
 r(\langle c_1,\dots,c_k,a\rangle)&=(\exists_y,b,\forall_{z\in y}(z\in a\land\theta(a,z,\vec c))\land\forall_{z\in a}(\theta(a,z,\vec c)\rightarrow z\in y)),\\
 o(\langle c_1,\dots,c_k,a\rangle)&=\Omega.
\end{align*}
Concerning local correctness, note that we have $|b|<\Omega$ in the term system~$\varepsilon(S_{\omega^\alpha}^u)$, as $|b|$ is an ordinal below $\omega^\alpha$ (cf.~Lemma~\ref{lem:embed-alpha-s-alpha}). The above construction was performed for all values $a\in\mathbf L_{\omega^\alpha}$. Thus we may introduce a universal quantifier, by setting
\begin{align*}
 l(\langle c_1,\dots,c_k\rangle)&=\langle\forall_x\exists_y(\forall_{z\in y}(z\in x\land\theta(x,z,\vec c))\land\forall_{z\in x}(\theta(x,z,\vec c)\rightarrow z\in y))\rangle,\\
 r(\langle c_1,\dots,c_k\rangle)&=(\forall_x,\exists_y(\forall_{z\in y}(z\in x\land\theta(x,z,\vec c))\land\forall_{z\in x}(\theta(x,z,\vec c)\rightarrow z\in y))),\\
 o(\langle c_1,\dots,c_k\rangle)&=\Omega+\omega.
\end{align*}
The universal quantifiers over the variables $v_i$ are introduced in the same way, increasing the height by $\omega$ for each quantifier. The root will then receive labels
\begin{align*}
 l(\langle\rangle)&=\langle\forall_{\vec v}\forall_x\exists_y(\forall_{z\in y}(z\in x\land\theta(x,z,\vec v))\land\forall_{z\in x}(\theta(x,z,\vec v)\rightarrow z\in y))\rangle,\\
 o(\langle\rangle)&=\Omega+\omega\cdot(k+1),
\end{align*}
as required.

The axioms of equality, extensionality, pairing, union and infinity are proved in a similar way. For infinity one uses the witness $\omega\in\mathbf L_{\omega^\alpha}^u$ with $|\omega|\leq\omega$ (inequality is possible, e.g.~if $\omega\in u$).

Let us now look at an instance of $\Delta_0$-collection, i.e.\ at an axiom
\begin{equation*}
 \forall_{\vec v}\forall_w(\forall_{x\in w}\exists_y\theta(x,y,\vec v,w)\rightarrow\exists_z\forall_{x\in w}\exists_{y\in z}\theta(x,y,\vec v,w)),
\end{equation*}
where $\theta$ is a $\Delta_0$-formula. For arbitrary parameters $a,b,\vec c,d\in\mathbf L_{\omega^\alpha}^u$ the sequent
\begin{equation*}
 \neg\theta(a,b,\vec c,d),\theta(a,b,\vec c,d)
\end{equation*}
contains a true bounded formula, i.e.\ it is an axiom. Introducing an existential quantifier we get
\begin{equation*}
 \neg\theta(a,b,\vec c,d),\exists_y\theta(a,y,\vec c,d)
\end{equation*}
with height $\Omega$ (as above, the rank of the witness $b$ is always bounded by $\Omega$). Since $b\in\mathbf L_{\omega^\alpha}^u$ was arbitrary the constructed preproofs can be combined, to obtain
\begin{equation*}
 \forall_y\neg\theta(a,y,\vec c,d),\exists_y\theta(a,y,\vec c,d)
\end{equation*}
with height $\Omega+\omega$. Together with the axiom $a\notin d,a\in d$ we get
\begin{equation*}
 a\in d\land\forall_y\neg\theta(a,y,\vec c,d),a\notin d,\exists_y\theta(a,y,\vec c,d)
\end{equation*}
with height $\Omega+\omega+1$, and then
\begin{equation*}
 \exists_{x\in d}\forall_y\neg\theta(x,y,\vec c,d),a\notin d,\exists_y\theta(a,y,\vec c,d)
\end{equation*}
with height $\Omega+\omega\cdot 2$. Introducing two disjunctions yields
\begin{equation*}
 \exists_{x\in d}\forall_y\neg\theta(x,y,\vec c,d),a\notin d\lor\exists_y\theta(a,y,\vec c,d),
\end{equation*}
now with height $\Omega+\omega\cdot 2+2$. Since $a$ was arbitrary this gives
\begin{equation*}
 \exists_{x\in d}\forall_y\neg\theta(x,y,\vec c,d),\forall_{x\in d}\exists_y\theta(x,y,\vec c,d)
\end{equation*}
with height $\Omega+\omega\cdot 3$. Now we are in a position to use the reflection rule
\begin{equation*}
 (\rref,\exists_z\forall_{x\in d}\exists_{y\in z}\theta(x,y,\vec c,d)),
\end{equation*}
which gives a proof of
\begin{equation*}
 \exists_{x\in d}\forall_y\neg\theta(x,y,\vec c,d),\exists_z\forall_{x\in d}\exists_{y\in z}\theta(x,y,\vec c,d)
\end{equation*}
with height $\Omega+\omega\cdot 3+1$. Observe that we have $\Omega<\Omega+\omega\cdot 3+1$, as required for the local correctness of the reflection rule. Introducing two disjunctions we obtain
\begin{equation*}
 \forall_{x\in d}\exists_y\theta(x,y,\vec c,d)\rightarrow\exists_z\forall_{x\in d}\exists_{y\in z}\theta(x,y,\vec c,d)
\end{equation*}
with height $\Omega+\omega\cdot 3+3$. It only remains to introduce the universal quantifiers over $w$ and $\vec v$, as in the case of $\Delta_0$-separation.
\end{proof}

Now that we have proofs of the Kripke-Platek axioms, let us extend the search trees $S_{\omega^\alpha}^u$ from Section~\ref{sect:search-trees} to $\mathbf L_{\omega^\alpha}^u$-preproofs of the empty sequent:

\begin{proposition}\label{prop:proof-from-search-tree}
 For each countable transitive set $u=\{u_i\,|\,i\in\omega\}$ and each ordinal~$\alpha>1$ there is an $\mathbf L_{\omega^\alpha}^u$-preproof $(P_\alpha^u,l_\alpha^u,r_\alpha^u,o_\alpha^u)$ with end-sequent $\langle\rangle$ and height $\varepsilon_{\langle\rangle}\in\varepsilon(S_{\omega^\alpha}^u)$.
\end{proposition}
\begin{proof}
We go through the construction of the seach tree $S_{\omega^\alpha}^u$ with labelling function $l_\alpha^u=l:S_{\omega^\alpha}^u\rightarrow\text{``$\mathbf L_{\omega^\alpha}^u$-sequents''}$ from Definition \ref{def:construct-search-tree}. In doing so we construct additional labellings $r_\alpha^u:S_\alpha^u\rightarrow\text{``$\mathbf L_{\omega^\alpha}^u$-rules''}$ and $o_\alpha^u:S_{\omega^\alpha}^u\rightarrow\varepsilon(S_{\omega^\alpha}^u)$. We will also add certain subtrees to fulfil the local correctness conditions. Labelling the nodes in $S_{\omega^\alpha}^u$ by ``ordinals'' is easy: To the node $\sigma\in S_{\omega^\alpha}^u$ we attach the label $o_\alpha^u(\sigma):=\varepsilon_\sigma\in\varepsilon(S_{\omega^\alpha}^u)$. For $\sigma^\frown a\in S_{\omega^\alpha}^u$ we have $\sigma^\frown a <_{S_{\omega^\alpha}^u}\sigma$ in the Kleene-Brouwer ordering, and thus
\begin{equation*}
 \varepsilon_{\sigma^\frown a}<\varepsilon_{\sigma^\frown a}+\omega<\varepsilon_\sigma
\end{equation*}
by definition of the term system $\varepsilon(S_{\omega^\alpha}^u)$. This means that the labels $o_\alpha^u(\sigma)$ descend as required by the local correctness conditions. Next, for $\sigma\in S_{\omega^\alpha}^u$ we define the $\mathbf L_{\omega^\alpha}^u$-rule $r_\alpha^u(\sigma)$ as follows: If $\sigma$ has even length $2k$ then we put
\begin{equation*}
 r_\alpha^u(\sigma)=(\cut,\neg\theta_k).
\end{equation*}
Recall that $\langle\theta_k\rangle_{k\in\omega}$ is an enumeration of the Kripke-Platek axioms (excluding foundation) that we have used in the construction of the search trees. By definition of the search tree we have
\begin{equation*}
 l_\alpha^u(\sigma^\frown 0)\subseteq l_\alpha^u(\sigma),\neg\theta_k,
\end{equation*}
i.e.\ this part of the local correctness condition is satisfied. On the other hand, the search tree $S_{\omega^\alpha}^u$ does not contain the node $\sigma^\frown 1$. Here we must add a subtree to restore local correctness: In the previous lemma we have constructed an $\mathbf L_{\omega^\alpha}^u$-preproof of $\theta_k$, with ordinal height at most $\Omega^1\cdot 2$. In view of $\Omega^1\cdot 2<\varepsilon_{\sigma}$ the monotonicity of the ``ordinal'' labelling is preserved when we attach this preproof to the node $\sigma^\frown 1$. Now consider a node $\sigma\in S_{\omega^\alpha}^u$ of odd length. We will only write out one case, leaving the remaining ones to the reader: Assume that we have
\begin{equation*}
 l_\alpha^u(\sigma)=\Gamma,\exists_x\psi(x),\Gamma'
\end{equation*}
where $\Gamma$ consists of (negated) prime formulas. Recalling the definition of the search tree we see that $\sigma^\frown a\in S_{\omega^\alpha}^u$ holds precisely for $a=0$. Also, neglecting the order of the formulas, we have
\begin{equation*}
 l_\alpha^u(\sigma^\frown 0)=l_\alpha^u(\sigma),\psi(b)
\end{equation*}
for some particular $b\in\mathbf L_{\omega^\alpha}^u$. As $|b|$ is an ordinal below $\omega^\alpha$ we see
\begin{equation*}
 |b|<\Omega<\varepsilon_\sigma.
\end{equation*}
Setting
\begin{equation*}
 r_\alpha^u(\sigma)=(\exists_x,b,\psi)
\end{equation*}
we thus have local correctness at the node $\sigma$.
\end{proof}

To see where we would like to get, consider the following:

\begin{proposition}\label{prop:countable-proof-sound}
 If an $\mathbf L_{\omega^\alpha}^u$-preproof has ordinal height below $\Omega$ then some formula in its end-sequent holds in the structure $(\mathbb L_{\omega^\alpha}^u,\in)$. In particular the end-sequent cannot be empty.
\end{proposition}
\begin{proof}
 Writing $(P,l,r,o)$ for the given $\mathbf L_{\omega^\alpha}^u$-preproof, we have $o(\langle\rangle)<\Omega$ by assumption. By Lemma~\ref{lem:embed-alpha-s-alpha} this means that $o(\langle\rangle)$ is an actual ordinal (below $\omega^\alpha$). Thus it suffices to show
\begin{equation*}
 \forall_{\beta<\omega^\alpha}\forall_{\sigma\in P}(o(\sigma)=\beta\rightarrow\text{``some formula in $l(\sigma)$ holds in $\mathbb L_{\omega^\alpha}^u$''}).
\end{equation*}
This can be established by induction on $\beta$ (note that the induction statement is primitive recursive): Consider some ordinal $\beta$ and a node $\sigma\in P$ with $o(\sigma)=\beta$. We distinguish cases according to the rule $r(\sigma)$. First, observe that $r(\sigma)$ cannot be a reflection rule $(\rref,\dots)$, for this would require $\Omega\leq o(\sigma)$. The other cases follow from the local correctness conditions. As an example, consider $r(\sigma)=(\forall_x,\psi)$. Aiming at a contradiction, assume that no formula in $l(\sigma)$ holds in $\mathbb L_{\omega^\alpha}^u$. By local correctness the formula $\forall_x\psi(x)$ occurs in $l(\sigma)$. We want to establish that this formula holds in $\mathbb L_{\omega^\alpha}^u$. So consider an arbitrary element $a\in\mathbb L_{\omega^\alpha}^u$. Computing its rank
\begin{equation*}
 |a|=\min\{\beta<\omega^\alpha\,|\,a\in\mathbb L_{\beta+1}^u\}
\end{equation*}
we may view $a$ as an element of $\mathbf L_{\omega^\alpha}^u$. By local correctness we have $\sigma^\frown a\in P$, as well as $o(\sigma^\frown a)<o(\sigma)$ and $l(\sigma^\frown a)\subseteq l(\sigma),\psi(a)$. The induction hypothesis tells us that some formula in $l(\sigma^\frown a)$ holds in $\mathbb L_{\omega^\alpha}^u$. We have assumed that all formulas in the sequent $l(\sigma)$ fail, which means that $\mathbb L_{\omega^\alpha}^u$ must satisfy $\psi(a)$. As $a\in\mathbb L_{\omega^\alpha}^u$ was arbitrary we can conclude that $\forall_x\psi$ holds in $\mathbb L_{\omega^\alpha}^u$, which completes the induction step in this case. It is also worth looking at a cut rule $r(\sigma)=(\cut,\psi)$. Let us assume that $\psi$ holds in $\mathbb L_{\omega^\alpha}^u$, the converse case being symmetric. By local correctness we have $\sigma^\frown 1\in P$, as well as $o(\sigma^\frown 1)<o(\sigma)$ and $l(\sigma^\frown 1)\subseteq l(\sigma),\neg\psi$. The induction hypothesis tells us that some formula in $l(\sigma^\frown 1)$ holds in $\mathbb L_{\omega^\alpha}^u$. By assumption the formula $\neg\psi$ fails in $\mathbb L_{\omega^\alpha}^u$. Thus $\mathbb L_{\omega^\alpha}^u$ must satisfy some formula in $l(\sigma)$, as required for the induction step.
\end{proof}

 In Proposition \ref{prop:proof-from-search-tree} we have constructed $\mathbf L_{\omega^\alpha}^u$-preproofs $P_\alpha^u$ with empty end-sequent and ordinal height $\varepsilon_{\langle\rangle}>\Omega$. On the other hand we have just seen that no $\mathbf L_{\omega^\alpha}^u$-preproof with empty end-sequent can have ordinal height below $\Omega$. Now the plan is as follows: Aiming at a contradiction, assume that there is no admissible set that contains $u$. By Proposition~\ref{prop:no-admissible-gives-wop} it follows that $\alpha\mapsto\varepsilon(S_{\omega^\alpha}^u)$ is a well-ordering principle. The higher Bachmann-Howard principle from Definition~\ref{def:higher-bh-principle} then gives a collapsing function $\vartheta:\varepsilon(S_{\omega^\alpha}^u)\bh\alpha$, for some ordinal $\alpha$. This will allow us to ``collapse'' the $\mathbf L_{\omega^\alpha}^u$-preproof $P_\alpha^u$ to ordinal height below $\Omega$, yielding the desired contradiction. The required transformations of $P_\alpha^u$ use two techniques from proof-theory --- cut elimination and collapsing --- which will be developed in the next two sections. We point out that these methods stem from the ordinal analysis of Kripke-Platek set theory via local predicativity: see in particular the work of J\"ager~\cite{jaeger-kripke-platek}, Pohlers~\cite{pohlers-local-predicativity} and Buchholz~\cite{buchholz-local-predicativity}.

\section{Cut Elimination}

In the last paragraph of the previous section we have outlined an argument that establishes the existence of admissible sets. The present section is devoted to one particular step in this argument, cut elimination. We begin with the main concept:

\begin{definition}
 The height $\hth(\varphi)\in\omega$ of an $\mathbf L_{\omega^\alpha}^u$-formula $\varphi$ is defined as follows:
\begin{enumerate}[label=(\roman*)]
 \item If $\varphi$ is a $\Delta_0$-formula then we have $\hth(\varphi)=0$.
 \item If $\varphi\equiv\psi_0\land\psi_1$ is not a $\Delta_0$-formula then $\hth(\varphi)=\max\{\hth(\psi_0),\hth(\psi_1)\}+1$. Similarly for $\varphi\equiv\psi_0\lor\psi_1$.
 \item If $\varphi\equiv\forall_x\,\psi$ is not a $\Delta_0$-formula then we have $\hth(\varphi)=\hth(\psi)+1$. Similarly for~$\varphi\equiv\exists_x\,\psi$.
\end{enumerate}
We say that an $\mathbf L_{\omega^\alpha}^u$-preproof $(P,l,r,o)$ has cut-rank $n$ if the following holds: For any node $\sigma\in P$, if $r(\sigma)$ is of the form $(\cut,\psi)$ then we have $\hth(\psi)<n$.
\end{definition}

Let us check the cut-rank of the proofs that we have constructed so far:

\begin{lemma}\label{lem:cut-rank-search-trees}
 There is some global bound $C\in\omega$ such that the $\mathbf L_{\omega^\alpha}^u$-preproofs $P_\alpha^u$ constructed in Proposition \ref{prop:proof-from-search-tree} all have cut-rank $C$.
\end{lemma}
\begin{proof}
 The $\mathbf L_{\omega^\alpha}^u$-preproofs from Lemma \ref{lem:proofs-kp-axioms}, establishing the Kripke-Platek axioms, contain no cuts. Thus all cut rules in the preproofs $P_\alpha^u$ are of the form $(\cut,\neg\theta_k)$, where $\langle\theta_k\rangle_{k\in\omega}$ is our enumeration of the Kripke-Platek axioms (not including instances of foundation). As stated in Section \ref{sect:search-trees} we only allow a fixed number $C_0$ of parameters in the axiom schemes of $\Delta_0$-separation and $\Delta_0$-collection (all other instances can be deduced via coding of tuples). Each such instance of $\Delta_0$-separation (resp.~$\Delta_0$-collection) has height at most $C_0+2$ (resp.~$C_0+5$). Thus the claim holds for~$C=C_0+6$.
\end{proof}

In the previous proof we have used the fact that the axioms in our search tree have bounded quantifier complexity. We should point out that this is a convenient simplification rather than a necessary prerequisite: It is indeed possible to transform a search tree with unbounded cut rank into a preproof with bounded cut rank. To do so one must interweave embedding and cut elimination, as in Rathjen's and Vizca\'ino's construction of $\omega$-models of bar induction (see~\cite[Theorem~5.26]{rathjen-model-bi}). Rather than adopting this approach, we have chosen to bound the number of parameters in the axioms: It makes the presentation easier and means no real restriction (contrary to the situation for bar induction, where unbounded quantifier complexity cannot be avoided).

The goal of the present section is to transform the preproofs $P_\alpha^u$ into preproofs with cut-rank $2$ (and unchanged end-sequent). The easiest way to describe the required operations would be by transfinite recursion over proof trees. However, this approach is not available to us: Firstly, we do not currently assume that the preproofs $P_\alpha^u$ are well-founded. More importantly, recursion over arbitrary well-orderings is not available in primitive recursive set theory. In order to describe the required operations in a primitive recursive way we adopt Buchholz' approach \cite{buchholz91} to ``continuous cut elimination'' (for systems of set theory this is worked out in \cite{buchholz-notations-set-theory}, but we will not follow the formalism from that paper). The idea is to define a set of $\mathbf L_{\omega^\alpha}^u$-codes and a primitive recursive interpretation that maps each code to an $\mathbf L_{\omega^\alpha}^u$-preproof. Primitive recursive transformations of proofs can then be described by simple operations on the codes. As a first step, let us define codes for the preproofs $P_\alpha^u$ that we have already constructed. To make the approach work we will also need codes $P_\alpha^u\sigma$ for subtrees of the preproofs $P_\alpha^u$, rooted at arbitrary nodes $\sigma\in P_\alpha^u$.

\begin{definition}\label{def:basic-codes}
 By a basic $\mathbf L_{\omega^\alpha}^u$-code we mean a term of the form $P_\alpha^u\sigma$, for a countable transitive set $u=\{u_i\,|\,i\in\omega\}$, an ordinal $\alpha>1$, and a finite sequence $\sigma$ with entries in $\mathbf L_{\omega^\alpha}^u$. We define functions
\begin{align*}
 l_{\langle\rangle}:\text{``basic $\mathbf L_{\omega^\alpha}^u$-codes''}&\rightarrow\text{``$\mathbf L_{\omega^\alpha}^u$-sequents''},\\
 r_{\langle\rangle}:\text{``basic $\mathbf L_{\omega^\alpha}^u$-codes''}&\rightarrow\text{``$\mathbf L_{\omega^\alpha}^u$-rules''},\\
 o_{\langle\rangle}:\text{``basic $\mathbf L_{\omega^\alpha}^u$-codes''}&\rightarrow\varepsilon(S_{\omega^\alpha}^u)
\end{align*}
as follows: Let $l_\alpha^u$, $r_\alpha^u$ and $o_\alpha^u$ be the labelling functions of the $\mathbf L_\alpha^u$-preproof $P_\alpha^u$ (see~Proposition \ref{prop:proof-from-search-tree}). If $\sigma$ is not a node in $P_\alpha^u$ then set $l_{\langle\rangle}(P_\alpha^u\sigma)=\langle 0=0\rangle$, $r_{\langle\rangle}(P_\alpha^u\sigma)=\ax$ and $o_{\langle\rangle}(P_\alpha^u\sigma)=0$. If $\sigma$ is a node in $P_\alpha^u$ we put
\begin{align*}
 l_{\langle\rangle}(P_\alpha^u\sigma)=l_\alpha^u(\sigma),\\
 r_{\langle\rangle}(P_\alpha^u\sigma)=r_\alpha^u(\sigma),\\
 o_{\langle\rangle}(P_\alpha^u\sigma)=o_\alpha^u(\sigma).
\end{align*} 
Also, we define a function
\begin{equation*}
 n:\text{``basic $\mathbf L_{\omega^\alpha}^u$-codes''}\times\mathbf L_{\omega^\alpha}^u\rightarrow\text{``basic $\mathbf L_{\omega^\alpha}^u$-codes''},\quad n(P_\alpha^u\sigma,a):=P_\alpha^u(\sigma^\frown a).
\end{equation*}
\end{definition}

Intuitively, $P_\alpha^u\sigma$ represents the subtree of $P_\alpha^u$ rooted at $\sigma$. The functions $l_{\langle\rangle}$, $r_{\langle\rangle}$ and $o_{\langle\rangle}$ give the labels at the root of this subtree. Using $n$ we can also access the children of a given node. Combining these functions allows us to recover the full tree $P_\alpha^u$:

\begin{definition}\label{def:interpretation-codes}
 We iterate the function $n$ from Definition \ref{def:basic-codes} along finite sequences, to get a function
\begin{equation*}
 \bar n:\text{``basic $\mathbf L_{\omega^\alpha}^u$-codes''}\times(\mathbf L_{\omega^\alpha}^u)^{<\omega}\rightarrow\text{``basic $\mathbf L_{\omega^\alpha}^u$-codes''}
\end{equation*}
with
\begin{align*}
 \bar n(P,\langle\rangle)&:=P,\\
 \bar n(P,\sigma^\frown a)&:=n(\bar n(P,\sigma),a).
\end{align*}
Using this, define functions
 \begin{align*}
 l:\text{``basic $\mathbf L_{\omega^\alpha}^u$-codes''}\times(\mathbf L_{\omega^\alpha}^u)^{<\omega}&\rightarrow\text{``$\mathbf L_{\omega^\alpha}^u$-sequents''}, &l(P,\sigma)&\mathrel{:=}l_{\langle\rangle}(\bar n(P,\sigma)),\\
 r:\text{``basic $\mathbf L_{\omega^\alpha}^u$-codes''}\times(\mathbf L_{\omega^\alpha}^u)^{<\omega}&\rightarrow\text{``$\mathbf L_{\omega^\alpha}^u$-rules''}, &r(P,\sigma)&\mathrel{:=}r_{\langle\rangle}(\bar n(P,\sigma)),\\
 o:\text{``basic $\mathbf L_{\omega^\alpha}^u$-codes''}\times(\mathbf L_{\omega^\alpha}^u)^{<\omega}&\rightarrow\varepsilon(S_{\omega^\alpha}^u), &o(P,\sigma)&\mathrel{:=}o_{\langle\rangle}(\bar n(P,\sigma)).
\end{align*}
In order to define a subtree $[P]\subseteq(\mathbf L_{\omega^\alpha}^u)^{<\omega}$ for each basic $\mathbf L_{\omega^\alpha}^u$-code $P$ we need an auxiliary notion: The relevant premises of an $\mathbf L_{\omega^\alpha}^u$-rule are given by
\begin{gather*}
 \iota(\ax)=\emptyset,\\
 \iota((\land,\psi_0,\psi_1))=\iota((\cut,\psi))=\{0,1\},\\
 \iota((\lor_i,\psi_0,\psi_1))=\iota((\exists_x,b,\psi))=\iota((\rref,\exists_z\forall_{x\in a}\exists_{y\in z}\theta))=\{0\},\\
 \iota((\forall_x,\psi))=\mathbf L_{\omega^\alpha}^u,\\
 \iota((\rep,a))=\{a\}.
\end{gather*}
 Given an $\mathbf L_{\omega^\alpha}^u$-code $P$ we now define $[P]\subseteq(\mathbf L_{\omega^\alpha}^u)^{<\omega}$ by recursion on the length of sequences: We always have $\langle\rangle\in[P]$. Given $\sigma\in [P]$ we stipulate $\sigma^\frown a\in [P]$ if and only if $a\in\iota(r(P,\sigma))$. Finally, write $l_{[P]},r_{[P]},o_{[P]}$ for the restrictions of the functions $l(P,\cdot),r(P,\cdot),o(P,\cdot)$ to $[P]$. The tree $[P]$, together with the functions $l_{[P]},r_{[P]},o_{[P]}$, is called the interpretation of the basic $\mathbf L_{\omega^\alpha}^u$-code $P$.
\end{definition}

The following result is reassuring, even though we will never use it:

\begin{lemma}\label{lem:int-basic-codes-expected}
 The labelled trees $[P_\alpha^u\langle\rangle]$ and $P_\alpha^u$ are equal. In particular $[P_\alpha^u\langle\rangle]$ is a locally correct $\mathbf L_{\omega^\alpha}^u$-preproof.
\end{lemma}
\begin{proof}
 By induction on (the length of) a sequence $\sigma\in(\mathbf L_{\omega^\alpha}^u)^{<\omega}$ we verify
\begin{equation*}
 \sigma\in[P_\alpha^u\langle\rangle]\Leftrightarrow \sigma\in P_\alpha^u.
\end{equation*}
Note that, once we have seen $\sigma\in [P_\alpha^u\langle\rangle]\cap P_\alpha^u$, equality of the labels is immediate: Unravelling the definitions we get
\begin{equation*}
 l_{[P_\alpha^u\langle\rangle]}(\sigma)=l(P_\alpha^u\langle\rangle,\sigma)=l_{\langle\rangle}(\bar n(P_\alpha^u\langle\rangle,\sigma))=l_{\langle\rangle}(P_\alpha^u\sigma)=l_\alpha^u(\sigma),
\end{equation*}
where $l_\alpha^u$ is the labelling function of the proof $P_\alpha^u$ (the same holds for rules and ordinal labels). As for the base case of the induction, both $[P_\alpha^u\langle\rangle]$ and $P_\alpha^u$ contain the empty sequent. Now assume that the equivalence holds for $\sigma$. By definition $[P_\alpha^u\langle\rangle]$ and $P_\alpha^u$ are trees, so it suffices to consider the case where we have $\sigma\in[P_\alpha^u\langle\rangle]$ and $\sigma\in P_\alpha^u$. We distinguish cases according to the rule $r_{[P_\alpha^u\langle\rangle]}(\sigma)=r_\alpha^u(\sigma)$. Assume for example that $r_{[P_\alpha^u\langle\rangle]}(\sigma)$ is of the form $(\land,\psi_0,\psi_1)$. By the definition of $[P_\alpha^u\langle\rangle]$ we have $\sigma^\frown a\in[P_\alpha^u\langle\rangle]$ if and only if $a\in\{0,1\}$. However, the latter is also equivalent to $\sigma^\frown a\in P_\alpha^u$, by the local correctness of $P_\alpha^u$. The other cases are checked in the same way.
\end{proof}

We extend this result by showing that $[P]$ is an $\mathbf L_{\omega^\alpha}^u$-preproof for any basic $\mathbf L_{\omega^\alpha}^u$-code P. As we shall see, the proof of this is at least as important as the result.

\begin{lemma}\label{lem:local-correctness-codes}
The system of basic $\mathbf L_{\omega^\alpha}^u$-codes is locally correct, in the sense that the following --- which we will call condition (L) --- holds for any basic $\mathbf L_{\omega^\alpha}^u$-code $P$:
{\def\arraystretch{1.75}\tabcolsep=10pt
\setlength{\LTpre}{\baselineskip}\setlength{\LTpost}{\baselineskip}
\begin{longtable}{lp{0.68\textwidth}}\hline
If $r_{\langle\rangle}(P)$ is \dots & \dots\ then \dots\\ \hline
$\ax$ & $l_{\langle\rangle}(P)$ contains a true $\Delta_0$-formula,\\
$(\land,\psi_0,\psi_1)$ & we have $o_{\langle\rangle}(n(P,i))<o_{\langle\rangle}(P)$ for $i=0,1$;\newline also $\psi_0\land\psi_1\in l_{\langle\rangle}(P)$ and $l_{\langle\rangle}(n(P,i))\subseteq l_{\langle\rangle}(P),\psi_i$,\\
$(\lor_i,\psi_0,\psi_1)$ & we have $o_{\langle\rangle}(n(P,0))<o_{\langle\rangle}(P)$;\newline also $\psi_0\lor\psi_1\in l_{\langle\rangle}(P)$ and $l_{\langle\rangle}(n(P,0))\subseteq l_{\langle\rangle}(P),\psi_i$,\\
$(\forall_x,\psi)$ & we have $o_{\langle\rangle}(n(P,a))+\omega\leq o_{\langle\rangle}(P)$ for all $a\in\mathbf L_{\omega^\alpha}^u$;\newline also $\forall_x\psi\in l_{\langle\rangle}(P)$ and $l_{\langle\rangle}(n(P,a))\subseteq l_{\langle\rangle}(P),\psi(a)$,\\
$(\exists_x,b,\psi)$ & we have $o_{\langle\rangle}(n(P,0))+\omega\leq o_{\langle\rangle}(P)$ and $|b|<o_{\langle\rangle}(P)$;\newline also $\exists_x\psi\in l_{\langle\rangle}(P)$ and $l_{\langle\rangle}(n(P,0))\subseteq l_{\langle\rangle}(P),\psi(b)$,\\
$(\cut, \psi)$ & we have $o_{\langle\rangle}(n(P,i))<o_{\langle\rangle}(P)$ for $i=0,1$; also\newline $l_{\langle\rangle}(n(P,0))\subseteq l_{\langle\rangle}(P),\psi$ and $l_{\langle\rangle}(n(P,1))\subseteq l_{\langle\rangle}(P),\neg\psi$,\\
$(\rref,\exists_z\forall_{x\in a}\exists_{y\in z}\theta)$ & we have $o_{\langle\rangle}(n(P,0))<o_{\langle\rangle}(P)$ and $\Omega\leq o_{\langle\rangle}(P)$; also\newline $\exists_z\forall_{x\in a}\exists_{y\in z}\theta\in l_{\langle\rangle}(P)$ and $l_{\langle\rangle}(n(P,0))\subseteq l_{\langle\rangle}(P),\forall_{x\in a}\exists_y\theta$,\\
$(\rep, b)$ & we have $o_{\langle\rangle}(n(P,b))<o(s)$; also $l_{\langle\rangle}(n(P,b))\subseteq l_{\langle\rangle}(P)$.\\ \hline
\end{longtable}}
\end{lemma}
\begin{proof}
Any basic $\mathbf L_{\omega^\alpha}^u$-code is of the form $P=P_\alpha^u\sigma$. We must consider two cases: For $\sigma\notin P_\alpha^u$ we have defined $r_{\langle\rangle}(P_\alpha^u\sigma)=\ax$ and $l_{\langle\rangle}(P_\alpha^u\sigma)=\langle 0=0\rangle$. Local correctness is given because $0=0$ is a true $\Delta_0$-formula. Now consider the case $\sigma\in P_\alpha^u$. Then we have $r_{\langle\rangle}(P_\alpha^u\sigma)=r_\alpha^u(\sigma)$, where $r_\alpha^u$ is the labelling function of the preproof $P_\alpha^u$. As an example, assume $r_\alpha^u(\sigma)=(\exists_x,b,\psi)$. By local correctness of $P_\alpha^u$ we get $\sigma^\frown 0\in P_\alpha^u$. It follows that we have
\begin{equation*}
 o_{\langle\rangle}(n(P_\alpha^u\sigma,0))=o_{\langle\rangle}(P_\alpha^u\sigma^\frown 0)=o_\alpha^u(\sigma^\frown 0),
\end{equation*}
as well as  $o_{\langle\rangle}(P_\alpha^u\sigma)=o_\alpha^u(\sigma)$. Thus
\begin{equation*}
 o_{\langle\rangle}(n(P_\alpha^u\sigma,0))+\omega\leq o_{\langle\rangle}(P_\alpha^u\sigma)
\end{equation*}
follows from $o_\alpha^u(\sigma^\frown 0)+\omega\leq o_\alpha^u(\sigma)$, as guaranteed by the local correctness of $P_\alpha^u$. The other conditions and cases are checked in the same way.
\end{proof}

\begin{corollary}\label{cor:correct-proofs-from-codes}
 For any basic $\mathbf L_{\omega^\alpha}^u$-code $P$ the tree $[P]$ is an $\mathbf L_{\omega^\alpha}^u$-preproof.
\end{corollary}
\begin{proof}
 To check local correctness at a node $\sigma\in[P]$, we distinguish cases according to the rule $r_{[P]}(\sigma)=r(P,\sigma)$. Assume for example that we have $r_{[P]}(\sigma)=(\cut,\psi)$. By the definition of $[P]$ we have $\sigma^\frown a\in [P]$ if and only if $a\in\{0,1\}$, so this part of the local correctness condition is satisfied. Concerning the remaining conditions, observe that we have
\begin{equation*}
 r_{\langle\rangle}(\bar n(P,\sigma))=r(P,\sigma)=(\cut,\psi).
\end{equation*}
Thus condition (L) for $\bar n(P,\sigma)$ gives
\begin{equation*}
 l_{\langle\rangle}(n(\bar n(P,\sigma),0))\subseteq l_{\langle\rangle}(\bar n(P,\sigma)),\psi
\end{equation*}
We can deduce
\begin{multline*}
 l_{[P]}(\sigma^\frown 0)=l(P,\sigma^\frown 0)=l_{\langle\rangle}(\bar n(P,\sigma^\frown 0))=\\
=l_{\langle\rangle}(n(\bar n(P,\sigma),0))\subseteq l_{\langle\rangle}(\bar n(P,\sigma)),\psi=l_{[P]}(\sigma),\psi,
\end{multline*}
as required for the local correctness of $[P]$ at $\sigma$. The other conditions are deduced in the same way.
\end{proof}

Note that Lemma~\ref{lem:local-correctness-codes} only involves the functions $l_{\langle\rangle},r_{\langle\rangle},o_{\langle\rangle},n$ from Definition~\ref{def:basic-codes}. This gives us an easy way to extend the system of basic $\mathbf L_{\omega^\alpha}^u$-codes by new codes: All we need to do is extend the functions $l_{\langle\rangle},r_{\langle\rangle},o_{\langle\rangle},n$ to the new codes and show that condition (L) is still satisfied. Definition~\ref{def:interpretation-codes} will automatically provide an interpretation of the new codes (based on the extended functions $l_{\langle\rangle},r_{\langle\rangle},o_{\langle\rangle},n$). The proof of the corollary ensures that the interpretations of the new codes are locally correct $\mathbf L_{\omega^\alpha}^u$-preproofs. As a first application, let us show how a proof of a universal statement $\forall_x\psi(x)$ can be transformed into a proof of any instance $\psi(a)$, keeping the same ordinal height; similarly, a proof of a conjunction can be transformed into a proof of either conjunct:

\begin{lemma}\label{lem:inversion}
 We can extend the system of basic $\mathbf L_{\omega^\alpha}^u$-codes in the following way:
\begin{enumerate}[label=(\alph*)]
 \item For each universal formula $\forall_x\psi$ and each $b\in\mathbf L_{\omega^\alpha}^u$ we can add a unary function symbol $\mathcal I_{\forall_x\psi,b}$ such that we have
\begin{align*}
 l_{\langle\rangle}(\mathcal I_{\forall_x\psi,b}P)&=(l_{\langle\rangle}(P)\backslash\{\forall_x\psi\})\cup\{\psi(b)\},\\
 o_{\langle\rangle}(\mathcal I_{\forall_x\psi,b}P)&=o_{\langle\rangle}(P)
\end{align*}
for any $\mathbf L_{\omega^\alpha}^u$-code $P$ (of the extended system).
 \item For each conjunction $\psi_0\land\psi_1$ and each $i\in\{0,1\}$ we can add a unary function symbol $\mathcal I_{\psi_0\land\psi_1,i}$ such that we have
\begin{align*}
 l_{\langle\rangle}(\mathcal I_{\psi_0\land\psi_1,i}P)&=(l_{\langle\rangle}(P)\backslash\{\psi_0\land\psi_1\})\cup\{\psi_i\},\\
 o_{\langle\rangle}(\mathcal I_{\psi_0\land\psi_1,i}P)&=o_{\langle\rangle}(P)
\end{align*}
for any $\mathbf L_{\omega^\alpha}^u$-code $P$ (of the extended system).
\end{enumerate}
\end{lemma}
 \begin{proof}
  (a) Informally, the idea is to replace any rule $(\forall_x,\psi)$, which infers the formula $\forall_x\psi$ from the premises $\psi(a)$ for $a\in\mathbf L_{\omega^\alpha}^u$, by the rule $(\rep,b)$, which repeats the premise $\psi(b)$. It is instructive to phrase this as a recursion on $P$: First, apply the operator $\mathcal I_{\forall_x\psi,b}$ to the immediate subproofs $n(P,a)$ of $P$, to replace any occurrences of the rule $(\forall_x,\psi)$ in these subproofs. Additionally, if $(\forall_x,\psi)$ is the last rule of $P$ then replace it by the rule $(\rep,b)$. Astonishingly, $\mathbf L_{\omega^\alpha}^u$-codes allow us to make this idea formal, even when the preproof $P$ is not well-founded and recursion on $P$ is not available: Formally, we define $l_{\langle\rangle}(P),r_{\langle\rangle}(P),o_{\langle\rangle}(P)$ and $n(P,a)$ by recursion on the length of the term $P$. Definition~\ref{def:basic-codes} accounts for a basic code $P$. Any other code is of the form $\mathcal I_{\forall_x\psi,b}P$, for some formula $\psi$ and some parameter $b$. Set
\begin{align*}
 r_{\langle\rangle}(\mathcal I_{\forall_x\psi,b}P)&:=\begin{cases}
                                                     (\rep,b)\quad &\text{if $r_{\langle\rangle}(P)=(\forall_x,\psi)$},\\
                                                     r_{\langle\rangle}(P)\quad &\text{otherwise},
                                                    \end{cases}\\
 n(\mathcal I_{\forall_x\psi,b}P,a)&:=\mathcal I_{\forall_x\psi,b}n(P,a).
\end{align*}
The recursive clauses for $l_{\langle\rangle}(\mathcal I_{\forall_x\psi,b}P)$ and $o_{\langle\rangle}(\mathcal I_{\forall_x\psi,b}P)$ can be copied from the statement of the lemma. We have thus extended $l_{\langle\rangle}(\cdot),r_{\langle\rangle}(\cdot),o_{\langle\rangle}(\cdot)$ and $n(\cdot,\cdot)$ to primitive recursive functions on all $\mathbf L_{\omega^\alpha}^u$-codes of the extended system. It remains to show that condition (L) holds for all new $\mathbf L_{\omega^\alpha}^u$-codes. Note that the statement ``condition (L) holds for a given code $P$'' is primitive recursive (even the quantification over all $a\in\mathbf L_{\omega^\alpha}^u$ in the case $r_{\langle\rangle}(P)=(\forall_y,\varphi)$ is harmless --- in contrast to the situation for first order number theory, where one has to be careful to avoid an unbounded quantification over $\omega$ at this place). We may thus establish condition (L) by induction on the length of $\mathbf L_{\omega^\alpha}^u$-codes. For basic codes condition (L) holds by Lemma~\ref{lem:local-correctness-codes}. In the induction step we must deduce condition (L) for the term $\mathcal I_{\forall_x\psi,b}P$ from condition (L) for $P$. We distinguish cases according to the rule $r_{\langle\rangle}(P)$:

\emph{Case $r_{\langle\rangle}(P)=\ax$:} Then we have $r_{\langle\rangle}(\mathcal I_{\forall_x\psi,b}P)=\ax$, so me must ensure that
\begin{equation*}
 l_{\langle\rangle}(\mathcal I_{\forall_x\psi,b}P)=(l_{\langle\rangle}(P)\backslash\{\forall_x\psi\})\cup\{\psi(b)\}
\end{equation*}
contains a true $\Delta_0$-formula. By induction hypothesis $l_{\langle\rangle}(P)$ contains a true $\Delta_0$-for\-mu\-la. This formula is still contained in $l_{\langle\rangle}(\mathcal I_{\forall_x\psi,b}P)$, unless it is the formula $\forall_x\psi$ itself. In the latter case $\psi(b)$ is also a true $\Delta_0$-formula, contained in $l_{\langle\rangle}(\mathcal I_{\forall_x\psi,b}P)$.

\emph{Case $r_{\langle\rangle}(P)=(\land,\psi_0,\psi_1)$:} Then we have $r_{\langle\rangle}(\mathcal I_{\forall_x\psi,b}P)=(\land,\psi_0,\psi_1)$. Using the induction hypothesis we get
\begin{equation*}
 o_{\langle\rangle}(n(\mathcal I_{\forall_x\psi,b}P,i))=o_{\langle\rangle}(\mathcal I_{\forall_x\psi,b}n(P,i))=o_{\langle\rangle}(n(P,i))<o_{\langle\rangle}(P)=o_{\langle\rangle}(\mathcal I_{\forall_x\psi,b}P).
\end{equation*}
The local correctness of $P$ also gives $\psi_0\land\psi_1\in l_{\langle\rangle}(P)$. As $\psi_0\land\psi_1$ and $\forall_x\psi$ are different formulas we still have $\psi_0\land\psi_1\in l_{\langle\rangle}(\mathcal I_{\forall_x\psi,b}P)$. Finally, we also have
\begin{multline*}
 l_{\langle\rangle}(n(\mathcal I_{\forall_x\psi,b}P,i))=l_{\langle\rangle}(\mathcal I_{\forall_x\psi,b}n(P,i))=\\
=(l_{\langle\rangle}(n(P,i))\backslash\{\forall_x\psi\})\cup\{\psi(b)\}\subseteq ((l_{\langle\rangle}(P)\cup\{\psi_i\})\backslash\{\forall_x\psi\})\cup\{\psi(b)\}\subseteq\\
\subseteq (l_{\langle\rangle}(P)\backslash\{\forall_x\psi\})\cup\{\psi(b),\psi_i\}=l_{\langle\rangle}(\mathcal I_{\forall_x\psi,b}P)\cup\{\psi_i\},
\end{multline*}
as demanded by condition (L) for $\mathcal I_{\forall_x\psi,b}P$.

\emph{Case $r_{\langle\rangle}(P)=(\lor_i,\psi_0,\psi_1)$:} Similar to the previous case.

\emph{Case $r_{\langle\rangle}(P)=(\forall_y,\varphi)\neq(\forall_x,\psi)$:} Then we have $r_{\langle\rangle}(\mathcal I_{\forall_x\psi,b}P)=(\forall_y,\varphi)$. By the induction hypothesis we get
\begin{multline*}
 o_{\langle\rangle}(n(\mathcal I_{\forall_x\psi,b}P,a))+\omega=o_{\langle\rangle}(\mathcal I_{\forall_x\psi,b}n(P,a))+\omega=\\
=o_{\langle\rangle}(n(P,a))+\omega\leq o_{\langle\rangle}(P)=o_{\langle\rangle}(\mathcal I_{\forall_x\psi,b}P)
\end{multline*}
for all $a\in\mathbf L_{\omega^\alpha}^u$. Also, the formula $\forall_y\varphi$ occurs in $l_{\langle\rangle}(P)$; as this formula is different from the formula $\forall_x\psi$ it still occurs in $l_{\langle\rangle}(\mathcal I_{\forall_x\psi,b}P)$. Finally we have
\begin{multline*}
 l_{\langle\rangle}(n(\mathcal I_{\forall_x\psi,b}P,a))=l_{\langle\rangle}(\mathcal I_{\forall_x\psi,b}n(P,a))=\\
=(l_{\langle\rangle}(n(P,a))\backslash\{\forall_x\psi\})\cup\{\psi(b)\}\subseteq ((l_{\langle\rangle}(P)\cup\{\varphi(a)\})\backslash\{\forall_x\psi\})\cup\{\psi(b)\}\subseteq\\
\subseteq (l_{\langle\rangle}(P)\backslash\{\forall_x\psi\})\cup\{\psi(b),\varphi(a)\}=l_{\langle\rangle}(\mathcal I_{\forall_x\psi,b}P)\cup\{\varphi(a)\},
\end{multline*}
as required for condition (L).

\emph{Case $r_{\langle\rangle}(P)=(\forall_x,\psi)$:} Here we have $r_{\langle\rangle}(\mathcal I_{\forall_x\psi,b}P)=(\rep,b)$. We deduce
\begin{multline*}
 o_{\langle\rangle}(n(\mathcal I_{\forall_x\psi,b}P,b))=o_{\langle\rangle}(\mathcal I_{\forall_x\psi,b}n(P,b))=\\
=o_{\langle\rangle}(n(P,b))<o_{\langle\rangle}(n(P,b))+\omega\leq o_{\langle\rangle}(P)=o_{\langle\rangle}(\mathcal I_{\forall_x\psi,b}P)
\end{multline*}
from the induction hypothesis. Furthermore we have
\begin{multline*}
 l_{\langle\rangle}(n(\mathcal I_{\forall_x\psi,b}P,b))=l_{\langle\rangle}(\mathcal I_{\forall_x\psi,b}n(P,b))=\\
(l_{\langle\rangle}(n(P,b))\backslash\{\forall_x\psi\})\cup\{\psi(b)\}\subseteq ((l_{\langle\rangle}(P)\cup\{\psi(b)\})\backslash\{\forall_x\psi\})\cup\{\psi(b)\}=\\
= (l_{\langle\rangle}(P)\backslash\{\forall_x\psi\})\cup\{\psi(b)\}=l_{\langle\rangle}(\mathcal I_{\forall_x\psi,b}P),
\end{multline*}
as condition (L) requires in case of the rule $(\rep,b)$.

\emph{Case $r_{\langle\rangle}(P)=(\exists_x,c,\varphi)$:} Then we have $r_{\langle\rangle}(\mathcal I_{\forall_x\psi,b}P)=(\exists_x,c,\varphi)$. From the induction hypothesis we get $|c|<o_{\langle\rangle}(P)=o_{\langle\rangle}(\mathcal I_{\forall_x\psi,b}P)$. The remaining conditions are checked as in the previous cases.

\emph{Case $r_{\langle\rangle}(P)=(\cut,\varphi)$:} Similar to the previous cases.

\emph{Case $r_{\langle\rangle}(P)=(\rref,\exists_z\forall_{v\in a}\exists_{w\in z}\theta)$:} We have $r_{\langle\rangle}(\mathcal I_{\forall_x\psi,b}P)=(\rref,\exists_z\forall_{v\in a}\exists_{w\in z}\theta)$. Using the induction hypothesis we get $\Omega\leq o_{\langle\rangle}(P)=o_{\langle\rangle}(\mathcal I_{\forall_x\psi,b}P)$. Also by the induction hypothesis we know that $\exists_z\forall_{v\in a}\exists_{w\in z}\theta$ occurs in $l_{\langle\rangle}(P)$. Crucially, the formulas $\exists_z\forall_{v\in a}\exists_{w\in z}\theta$ and $\forall_x\psi$ cannot be the same, so that $\exists_z\forall_{v\in a}\exists_{w\in z}\theta$ still occurs in $l_{\langle\rangle}(\mathcal I_{\forall_x\psi,b}P)$. The remaining conditions are checked as in the previous cases.

\emph{Case $r_{\langle\rangle}(P)=(\rep,c)$:} Similar to the previous cases.

We have established condition (L) for all $\mathbf L_{\omega^\alpha}^u$-codes of the extended system. By (the proof of) Corollary~\ref{cor:correct-proofs-from-codes} it follows that the interpretations of these codes are locally correct $\mathbf L_{\omega^\alpha}^u$-preproofs.

(b) Similar to (a) we set
\begin{align*}
 r_{\langle\rangle}(\mathcal I_{\psi_0\land\psi_1,i}P)&:=\begin{cases}
                                                     (\rep,i)\quad &\text{if $r_{\langle\rangle}(P)=(\land,\psi_0,\psi_1)$},\\
                                                     r_{\langle\rangle}(P)\quad &\text{otherwise},
                                                    \end{cases}\\
 n(\mathcal I_{\psi_0\land\psi_1,i}P,a)&:=\mathcal I_{\psi_0\land\psi_1,i}n(P,a).
\end{align*}
The verification of condition (L) is similar to part (a), and left to the reader. Note that we also want to allow codes of the form $\mathcal I_{\forall_x\psi,b}\mathcal I_{\psi_0\land\psi_1,i}P$. Thus we must repeat the recursive clauses from (a) after adding the symbols $\mathcal I_{\psi_0\land\psi_1,i}$. Similarly, we must repeat the arguments from (a) in the inductive verification of condition (L). From a formal viewpoint it might in fact be preferable to write down the recursive clauses and inductive verifications for all $\mathbf L_{\omega^\alpha}^u$-codes that we ever want to introduce at a single blow. However, a modular presentation seems to be more readable, and can clearly be converted into a formal proof in principle.
\end{proof}

We have developed a system of codes $P$ that allow us to construct various $\mathbf L_{\omega^\alpha}^u$-preproofs $[P]$ in a primitive recursive way. Now let us investigate how we can prove properties of these preproofs, at the example of cut-rank.

\begin{definition}\label{def:cut-ranks-codes}
 We define an assignment of cut ranks
\begin{equation*}
 d:\text{``$\mathbf L_{\omega^\alpha}^u$-codes''}\rightarrow\omega
\end{equation*}
by recursion over $\mathbf L_{\omega^\alpha}^u$-codes: For basic codes $P_\alpha^u\sigma$ we set $d(P_\alpha^u\sigma)=C$, where $C$ is the constant from Lemma~\ref{lem:cut-rank-search-trees}. We extend $d$ to the $\mathbf L_{\omega^\alpha}^u$-codes introduced in Lemma~\ref{lem:inversion} by setting
\begin{equation*}
 d(\mathcal I_{\forall_x\psi,b}P):=d(\mathcal I_{\psi_0\land\psi_1,i}P):=d(P).
\end{equation*}
\end{definition}

So far, the assigment of cut ranks to $\mathbf L_{\omega^\alpha}^u$-codes is constant; more interesting cases will appear shortly. The following is parallel to the development in Lemma~\ref{lem:local-correctness-codes} and Corollary~\ref{cor:correct-proofs-from-codes}:

\begin{lemma}\label{lem:cut-codes-correct}
 The assignment of cut ranks to $\mathbf L_{\omega^\alpha}^u$-codes is locally correct, in the sense that the following conditions hold for any $\mathbf L_{\omega^\alpha}^u$-code $P$:
\begin{enumerate}[label=(C\arabic*)]
 \item If $r_{\langle\rangle}(P)$ is of the form $(\cut,\varphi)$ then we have $\hth(\varphi)<d(P)$.
 \item We have $d(n(P,a))\leq d(P)$ for all $a\in\iota(r_{\langle\rangle}(P))$. 
\end{enumerate}
\end{lemma}
\begin{proof}
 We argue by induction on (the length of) $\mathbf L_{\omega^\alpha}^u$-codes. First consider the case of a basic code $P=P_\alpha^u\sigma$. Concerning (C1) we need to distinguish two cases: If $\sigma$ is a node in $P_\alpha^u$ then we have $r_{\langle\rangle}(P_\alpha^u\sigma)=r_\alpha^u(\sigma)$, where $r_\alpha^u$ is the labelling function of the preproof $P_\alpha^u$. In this case (C1) holds by Lemma~\ref{lem:cut-rank-search-trees}. If $\sigma$ lies outside of $P_\alpha^u$ then we have $r_{\langle\rangle}(P_\alpha^u\sigma)=\ax$ by definition, so (C1) does not apply. Concerning (C2) for the basic code $P=P_\alpha^u\sigma$, observe that we have
\begin{equation*}
 d(n(P_\alpha^u\sigma,a))=d(P_\alpha^u\sigma^\frown a)=C
\end{equation*}
because $P_\alpha^u\sigma^\frown a$ is also a basic code. In the induction step, consider $P=\mathcal I_{\forall_x\psi,b}P'$. Note that $r_{\langle\rangle}(\mathcal I_{\forall_x\psi,b}P')=(\cut,\varphi)$ can only occur if we have $r_{\langle\rangle}(P')=(\cut,\varphi)$. Using the induction hypothesis we thus get
\begin{equation*}
 \hth(\varphi)<d(P')=d(\mathcal I_{\forall_x\psi,b}P'),
\end{equation*}
as needed for (C1). Concerning (C2), observe that $\iota(r_{\langle\rangle}(\mathcal I_{\forall_x\psi,b}P'))\subseteq\iota(r_{\langle\rangle}(P'))$ holds in all possible cases. Thus the induction hypothesis implies
\begin{equation*}
 d(n(\mathcal I_{\forall_x\psi,b}P',a))=d(\mathcal I_{\forall_x\psi,b}n(P',a))=d(n(P',a))\leq d(P')=d(\mathcal I_{\forall_x\psi,b}P')
\end{equation*}
for all $a\in\iota(r_{\langle\rangle}(\mathcal I_{\forall_x\psi,b}P'))$. For a code $P=\mathcal I_{\psi_o\land\psi_1,u}P'$ one argues similarly.
\end{proof}

\begin{corollary}\label{cor:cut-rank-codes}
 For any $\mathbf L_{\omega^\alpha}^u$-code $P$ the $\mathbf L_{\omega^\alpha}^u$-preproof $[P]$ has cut-rank $d(P)$.
\end{corollary}
\begin{proof}
Recall the function $\bar n$ defined by the recursion
\begin{align*}
 \bar n(P,\langle\rangle)&=P,\\
 \bar n(P,\sigma^\frown a)&=n(\bar n(P,\sigma),a).
\end{align*}
By induction on the sequence $\sigma\in[P]$ one establishes $d(\bar n(P,\sigma))\leq d(P)$, using condition (C2) from the previous lemma as the induction step. Also recall that $r_{[P]}(\sigma)$ was defined to be the rule $r_{\langle\rangle}(\bar n(P,\sigma))$. Thus if $r_{[P]}(\sigma)$ is of the form $(\cut,\varphi)$ then we have
\begin{equation*}
 \hth(\varphi)<d(\bar n(P,\sigma))\leq d(P),
\end{equation*}
by condition (C1) from the previous lemma.
\end{proof}

It turns out that the corollary that we have just established is never used. This is because cut-rank is an auxiliary notion, which we could limit to the realm of codes rather than actual proofs. To handle cut-ranks for codes we will need Lemma~\ref{lem:cut-codes-correct} but not, strictly speaking, its corollary. In contrast to this, Corollary~\ref{cor:correct-proofs-from-codes} above will be used, namely to obtain an actual $\mathbf L_{\omega^\alpha}^u$-preproof to which Proposition~\ref{prop:countable-proof-sound} can be applied (of course we could renounce the notion of $\mathbf L_{\omega^\alpha}^u$-preproof altogether, and formulate everything in terms of codes --- but this seems somewhat forced in a set-theoretic context, where infinite proof trees exist as first-rate objects). Even though they are not strictly required we will continue to state results such as Corollary~\ref{cor:cut-rank-codes}, as they clarify the intended semantics of $\mathbf L_{\omega^\alpha}^u$-codes.

In the following we will extend our system of $\mathbf L_{\omega^\alpha}^u$-codes by a unary function symbol~$\mathcal E$ such that the $\mathbf L_{\omega^\alpha}^u$-preproof $[\mathcal EP]$ has lower cut-rank than $[P]$. To show that this is the case we will only need to extend the assignment $d$ from Definition~\ref{def:cut-ranks-codes} to the new codes. We will prove that this extension of $d$ satisfies the conditions (C1) and (C2). By the (proof of the) corollary it will immediately follow that the preproof~$[\mathcal EP]$ has cut-rank $<d(\mathcal EP)$. The following is a preparation:

\begin{lemma}\label{lem:cut-reduction}
 We can extend the system of $\mathbf L_{\omega^\alpha}^u$-codes in the following way:
\begin{enumerate}[label=(\alph*)]
 \item For each formula $\exists_x\psi$ with $\hth(\exists_x\psi)>1$ we add a binary function symbol $\mathcal R_{\exists_x\psi}$ such that we have
\begin{align*}
 l_{\langle\rangle}(\mathcal R_{\exists_x\psi}P_0P_1)&=(l_{\langle\rangle}(P_0)\backslash\{\exists_x\psi\})\cup(l_{\langle\rangle}(P_1)\backslash\{\forall_x\neg\psi\}),\\
 o_{\langle\rangle}(\mathcal R_{\exists_x\psi}P_0P_1)&=o_{\langle\rangle}(P_1)+o_{\langle\rangle}(P_0),\\
 d(\mathcal R_{\exists_x\psi}P_0P_1)&=\max\{d(P_0),d(P_1),\hth(\exists_x\psi)\}
\end{align*}
for any $\mathbf L_{\omega^\alpha}^u$-codes $P_0,P_1$.
 \item For each formula $\psi_0\lor\psi_1$ with $\hth(\psi_0\lor\psi_1)>1$ we add a binary function symbol $\mathcal R_{\psi_0\lor\psi_1}$ such that we have
\begin{align*}
 l_{\langle\rangle}(\mathcal R_{\psi_0\lor\psi_1}P_0P_1)&=(l_{\langle\rangle}(P_0)\backslash\{\psi_0\lor\psi_1\})\cup(l_{\langle\rangle}(P_1)\backslash\{\neg\psi_0\land\neg\psi_1\}),\\
 o_{\langle\rangle}(\mathcal R_{\psi_0\lor\psi_1}P_0P_1)&=o_{\langle\rangle}(P_1)+o_{\langle\rangle}(P_0),\\
 d(\mathcal R_{\exists_x\psi}P_0P_1)&=\max\{d(P_0),d(P_1),\hth(\psi_0\lor\psi_1)\}
\end{align*}
for any $\mathbf L_{\omega^\alpha}^u$-codes $P_0,P_1$.
\end{enumerate}
\end{lemma}
\begin{proof}
 (a) Let us first describe the proof idea in informal terms: Assume that the formula $\exists_x\psi$ is deduced from a premise $\psi(b)$ at some node of $P$. We want to avoid the introduction of $\exists_x\psi$, to the keep it out of the end-sequent of $\mathcal R_{\exists_x\psi}P_0P_1$. To achieve this we transform $P_1$ according to Lemma~\ref{lem:inversion}, deleting the formula $\forall_x\neg\psi$ in favour of $\neg\psi(b)$. Now we can apply a cut over $\psi(b)$, to remove the premise $\psi(b)$ in $P_0$ and the formula $\neg\psi(b)$ that we have added to $P_1$. Concerning the cut-rank, note that we have $\hth(\psi(b))<\psi(\exists_x\psi)$. It is helpful to phrase this as a recursion over the preproof $P_0$: First, form the preproofs $\mathcal R_{\exists_x\psi}n(P_0,a)P_1$, to remove the formula $\exists_x\psi$ from the immediate subtrees $n(P_0,a)$ of $P_0$. If $\exists_x\psi$ was introduces by the last rule of $P_0$, then remove it by a cut over $\psi(b)$, as described above. Formally, we extend the functions $l_{\langle\rangle},r_{\langle\rangle},o_{\langle\rangle},n,d$ to the new codes by the recursive clauses
\begin{align*}
 r_{\langle\rangle}(\mathcal R_{\exists_x\psi}P_0P_1)&:=\begin{cases}
                                                     (\cut,\psi(b))\quad &\text{if $r_{\langle\rangle}(P_0)=(\exists_x,b,\psi)$ for some $b$},\\
                                                     r_{\langle\rangle}(P_0)\quad &\text{otherwise},
                                                    \end{cases}\\
 n(\mathcal R_{\exists_x\psi}P_0P_1,a)&:=\begin{cases}
                                                     \mathcal I_{\forall_x\neg\psi,b}P_1\quad &\text{if $r_{\langle\rangle}(P_0)=(\exists_x,b,\psi)$ and $a=1$},\\
                                                     \mathcal R_{\exists_x\psi}n(P_0,a)P_1\quad &\text{otherwise}.
                                                    \end{cases}
\end{align*}
Conditions (L), (C1) and (C2) are verified by induction on the length of the new $\mathbf L_{\omega^\alpha}^u$-codes: Concerning (C1), assume that $r_{\langle\rangle}(\mathcal R_{\exists_x\psi}P_0P_1)$ is of the form $(\cut,\varphi)$. There are two possibilities: If $\varphi$ is a formula $\psi(b)$ then we have
\begin{equation*}
 \hth(\varphi)<\hth(\exists_x\psi)\leq d(\mathcal R_{\exists_x\psi}P_0P_1),
\end{equation*}
as required. Otherwise we must have $r_{\langle\rangle}(\mathcal R_{\exists_x\psi}P_0P_1)=r_{\langle\rangle}(P_0)$. Using the induction hypothesis for $P_0$ we get
\begin{equation*}
 \hth(\varphi)<d(P_0)\leq d(\mathcal R_{\exists_x\psi}P_0P_1).
\end{equation*}
As for condition (C2), we either have
\begin{equation*}
 d(n(\mathcal R_{\exists_x\psi}P_0P_1,a))=d(\mathcal I_{\forall_x\neg\psi,b}P_1)=d(P_1)\leq d(\mathcal R_{\exists_x\psi}P_0P_1)
\end{equation*}
or, using the induction hypothesis for $P_0$,
\begin{multline*}
 d(n(\mathcal R_{\exists_x\psi}P_0P_1,a))=d(\mathcal R_{\exists_x\psi}s(P_0,a)P_1)=\max\{d(s(P_0,a)),d(P_1),\hth(\exists_x\psi)\}\leq\\
\leq\max\{d(P_0),d(P_1),\hth(\exists_x\psi)\}=d(\mathcal R_{\exists_x\psi}P_0P_1).
\end{multline*}
To verify condition (L) we distinguish cases according to the rule $r_{\langle\rangle}(P_0)$:

\emph{Case $r_{\langle\rangle}(P_0)=\ax$:} Then we have $r_{\langle\rangle}(\mathcal R_{\exists_x\psi}P_0P_1)=\ax$. Condition (L) for $P_0$ tells us that $l_{\langle\rangle}(P_0)$ contains a true $\Delta_0$-formula. Our assumption $\hth(\exists_x\psi)>1$ implies that $\exists_x\psi$ is not a $\Delta_0$-formula. Thus the true $\Delta_0$-formula is still contained in $l_{\langle\rangle}(\mathcal R_{\exists_x\psi}P_0P_1)$, as required by condition (L) for $\mathcal R_{\exists_x\psi}P_0P_1$.

\emph{Case $r_{\langle\rangle}(P_0)=(\land,\psi_0,\psi_1)$:} Then we have $r_{\langle\rangle}(\mathcal R_{\exists_x\psi}P_0P_1)=(\land,\psi_0,\psi_1)$. By condition (L) for $P_0$ we get
\begin{multline*}
 o_{\langle\rangle}(n(\mathcal R_{\exists_x\psi}P_0P_1,i))=o_{\langle\rangle}(\mathcal R_{\exists_x\psi}s(P_0,i)P_1)=\\
=o_{\langle\rangle}(P_1)+o_{\langle\rangle}(s(P_0,i))<o_{\langle\rangle}(P_1)+o_{\langle\rangle}(P_0)=o_{\langle\rangle}(\mathcal R_{\exists_x\psi}P_0P_1)
\end{multline*}
for $i=0,1$. Condition (L) also tells us that $\psi_0\land\psi_1$ occurs in $l_{\langle\rangle}(P_0)$. As the formula $\psi_0\land\psi_1$ is different from $\exists_x\psi$ it still occurs in $l_{\langle\rangle}(\mathcal R_{\exists_x\psi}P_0P_1)$. Finally, again using condition (L) for $P_0$, we have
\begin{multline*}
 l_{\langle\rangle}(n(\mathcal R_{\exists_x\psi}P_0P_1,i))=l_{\langle\rangle}(\mathcal R_{\exists_x\psi}s(P_0,i)P_1)=\\
=(l_{\langle\rangle}(s(P_0,i))\backslash\{\exists_x\psi\})\cup(l_{\langle\rangle}(P_1)\backslash\{\forall_x\neg\psi\})\subseteq\\
\subseteq ((l_{\langle\rangle}(P_0)\cup\{\psi_i\})\backslash\{\exists_x\psi\})\cup(l_{\langle\rangle}(P_1)\backslash\{\forall_x\neg\psi\})\subseteq\\
\subseteq (l_{\langle\rangle}(P_0)\backslash\{\exists_x\psi\})\cup\{\psi_i\}\cup(l_{\langle\rangle}(P_1)\backslash\{\forall_x\neg\psi\})=l_{\langle\rangle}(\mathcal R_{\exists_x\psi}P_0P_1)\cup\{\psi_i\},
\end{multline*}
as required by condition (L) for $\mathcal R_{\exists_x\psi}P_0P_1$.

\emph{Case $r_{\langle\rangle}(P_0)=(\lor,\psi_0,\psi_1)$:} Similar to the previous case.

\emph{Case $r_{\langle\rangle}(P_0)=(\forall_y,\varphi)$:} Similar to the previous cases.

\emph{Case $r_{\langle\rangle}(P_0)=(\exists_y,b,\varphi)$ with $\exists_y\varphi\neq\exists_x\psi$:} Then $r_{\langle\rangle}(\mathcal R_{\exists_x\psi}P_0P_1)=(\exists_y,b,\varphi)$. By condition (L) for $P_0$ the formula $\exists_y\varphi$ occurs in $l_{\langle\rangle}(P_0)$. In view of $\exists_y\varphi\neq\exists_x\psi$ this formula is still contained in $l_{\langle\rangle}(\mathcal R_{\exists_x\psi}P_0P_1)$. The remaining conditions are checked as in the previous cases.

\emph{Case $r_{\langle\rangle}(P_0)=(\exists_x,b,\psi)$:} Then we have $r_{\langle\rangle}(\mathcal R_{\exists_x\psi}P_0P_1)=(\cut,\psi(b))$. We have
\begin{multline*}
 o_{\langle\rangle}(n(\mathcal R_{\exists_x\psi}P_0P_1,0))=o_{\langle\rangle}(\mathcal R_{\exists_x\psi}s(P_0,0)P_1)=\\
=o_{\langle\rangle}(P_1)+o_{\langle\rangle}(s(P_0,0))<o_{\langle\rangle}(P_1)+o_{\langle\rangle}(P_0)=o_{\langle\rangle}(\mathcal R_{\exists_x\psi}P_0P_1)
\end{multline*}
and
\begin{multline*}
 o_{\langle\rangle}(n(\mathcal R_{\exists_x\psi}P_0P_1,1))=o_{\langle\rangle}(\mathcal I_{\forall_x\neg\psi,b}P_1)=o_{\langle\rangle}(P_1)\leq\\ \leq o_{\langle\rangle}(P_1)+o_{\langle\rangle}(s(P_0,0))<o_{\langle\rangle}(P_1)+o_{\langle\rangle}(P_0)=o_{\langle\rangle}(\mathcal R_{\exists_x\psi}P_0P_1).
\end{multline*}
Condition (L) for $P_0$ yields
\begin{multline*}
 l_{\langle\rangle}(n(\mathcal R_{\exists_x\psi}P_0P_1,0))=l_{\langle\rangle}(\mathcal R_{\exists_x\psi}n(P_0,0)P_1)=\\
=(l_{\langle\rangle}(n(P_0,0))\backslash\{\exists_x\psi\})\cup(l_{\langle\rangle}(P_1)\backslash\{\forall_x\neg\psi\}\subseteq\\
\subseteq ((l_{\langle\rangle}(P_0)\cup\{\psi(b)\})\backslash\{\exists_x\psi\})\cup(l_{\langle\rangle}(P_1)\backslash\{\forall_x\neg\psi\}\subseteq\\
\subseteq (l_{\langle\rangle}(P_0)\backslash\{\exists_x\psi\})\cup(l_{\langle\rangle}(P_1)\backslash\{\forall_x\neg\psi\})\cup\{\psi(b)\}=l_{\langle\rangle}(\mathcal R_{\exists_x\psi}P_0P_1)\cup\{\psi(b)\}.
\end{multline*}
By Lemma~\ref{lem:inversion} we have
\begin{multline*}
 l_{\langle\rangle}(n(\mathcal R_{\exists_x\psi}P_0P_1,1))=l_{\langle\rangle}(\mathcal I_{\forall_x\neg\psi,b}P_1)=(l_{\langle\rangle}(P_1)\backslash\{\forall_x\neg\psi\})\cup\{\neg\psi(b)\}\subseteq\\
\subseteq (l_{\langle\rangle}(P_0)\backslash\{\exists_x\psi\})\cup(l_{\langle\rangle}(P_1)\backslash\{\forall_x\neg\psi\})\cup\{\neg\psi(b)\}=l_{\langle\rangle}(\mathcal R_{\exists_x\psi}P_0P_1)\cup\{\neg\psi(b)\},
\end{multline*}
as condition (L) requires in the case of the rule $(\cut,\psi(b))$.

\emph{Case $r_{\langle\rangle}(P_0)=(\rref,\exists_z\forall_{v\in a}\exists_{w\in z}\theta)$:} Then $r_{\langle\rangle}(\mathcal R_{\exists_x\psi}P_0P_1)$ is the same rule. We have
\begin{equation*}
 \Omega\leq o_{\langle\rangle}(P_0)\leq o_{\langle\rangle}(P_1)+o_{\langle\rangle}(P_0)=o_{\langle\rangle}(\mathcal R_{\exists_x\psi}P_0P_1).
\end{equation*}
Condition (L) for $P_0$ implies that $\exists_z\forall_{v\in a}\exists_{w\in z}\theta$ occurs in $l_{\langle\rangle}(P_0)$. Since $\theta$ is a $\Delta_0$-formula we have $\hth(\exists_z\forall_{v\in a}\exists_{w\in z}\theta)=1$. Thus our assumption $\hth(\exists_x\psi)>1$ implies that $\exists_z\forall_{v\in a}\exists_{w\in z}\theta$ and $\exists_x\psi$ are different formulas. It follows that $\exists_z\forall_{v\in a}\exists_{w\in z}\theta$ is still contained in $l_{\langle\rangle}(\mathcal R_{\exists_x\psi}P_0P_1)$, as demanded by condition (L) in the case of the rule $(\rref,\exists_z\forall_{v\in a}\exists_{w\in z}\theta)$. The other conditions are verified as in the previous cases.

\emph{Case $r_{\langle\rangle}(P_0)=(\rep,b)$:} Similar to the previous cases.

(b) Similar to part (a) we set
\begin{align*}
 r_{\langle\rangle}(\mathcal R_{\psi_0\lor\psi_1}P_0P_1)&:=\begin{cases}
                                                     (\cut,\psi_i)\quad &\text{if $r_{\langle\rangle}(P_0)=(\lor_i,\psi_0,\psi_1)$ for some $i\in\{0,1\}$},\\
                                                     r_{\langle\rangle}(P_0)\quad &\text{otherwise},
                                                    \end{cases}\\
 n(\mathcal R_{\psi_0\lor\psi_1}P_0P_1,a)&:=\begin{cases}
                                                     \mathcal I_{\neg\psi_0\land\neg\psi_1,i}P_1\quad &\text{if $r_{\langle\rangle}(P_0)=(\lor_i,\psi_0,\psi_1)$ and $a=1$},\\
                                                     \mathcal R_{\psi_0\lor\psi_1}s(P_0,a)P_1\quad &\text{otherwise}.
                                                    \end{cases}
\end{align*}
The verification of (L), (C1) and (C2) is similar to (a), and left to the reader.
\end{proof}

Finally we have all ingredients for the desired cut elimination operator:

\begin{proposition}\label{prop:cut-elimination-codes}
 We can extend the system of $\mathbf L_{\omega^\alpha}^u$-codes by a unary function symbol $\mathcal E$ such that we have
\begin{align*}
 l_{\langle\rangle}(\mathcal EP)&=l_{\langle\rangle}(P),\\
 o_{\langle\rangle}(\mathcal EP)&=\Omega^{o_{\langle\rangle}(P)},\\
 d(\mathcal EP)&=\max\{2,d(P)-1\}
\end{align*}
for each $\mathbf L_{\omega^\alpha}^u$-code $P$.
\end{proposition}
\begin{proof}
The intuitive idea is straightforward: We replace any cut over a formula $\exists_x\psi$ by an application of the operator $\mathcal R_{\exists_x\psi}$, which only involves cuts of lower rank. Formally, the functions $l_{\langle\rangle}(P),r_{\langle\rangle}(P),o_{\langle\rangle}(P),n(P,a),d(P)$ are defined by recursion on the length of the code $P$. We distinguish cases according to the last rule of $P$, and verify conditions (L), (C1) and (C2) as we go along:

\emph{Case $r_{\langle\rangle}(P)=\ax$:} We set $r_{\langle\rangle}(\mathcal EP):=\ax$ and $n(\mathcal EP,a):=\mathcal E n(P,a)$ (the latter is irrelevant because of $\iota(\ax)=\emptyset$). By the induction hypothesis $l_{\langle\rangle}(\mathcal EP)=l_{\langle\rangle}(P)$ contains a true $\Delta_0$-formula, as required for condition (L). Conditions (C1) and (C2) do not apply.

\emph{Case $r_{\langle\rangle}(P)=(\land,\psi_0,\psi_1)$:} Set $r_{\langle\rangle}(\mathcal EP):=(\land,\psi_0,\psi_1)$ and $n(\mathcal EP,a):=\mathcal E n(P,a)$. Let us verify condition (L): As exponentiation to the base $\Omega$ is monotone we get
\begin{equation*}
 o(n(\mathcal EP,i))=o(\mathcal En(P,i))=\Omega^{o(n(P,i))}<\Omega^{o(P)}=o(\mathcal EP).
\end{equation*}
By condition (L) for $P$ the formula $\psi_0\land\psi_1$ is contained in $l_{\langle\rangle}(P)=l_{\langle\rangle}(\mathcal EP)$. Also, we have
\begin{equation*}
 l_{\langle\rangle}(n(\mathcal EP,i))=l_{\langle\rangle}(\mathcal En(P,i))=l_{\langle\rangle}(n(P,i))\subseteq l_{\langle\rangle}(P)\cup\{\psi_i\}=l_{\langle\rangle}(\mathcal EP)\cup\{\psi_i\}
\end{equation*}
for $i=0,1$, as required by condition (L) for $\mathcal EP$. Condition (C1) does not apply. Using (C2) for $P$ we get
\begin{equation*}
 d(n(\mathcal EP,i))=d(\mathcal En(P,i))=\max\{2,d(n(P,i))-1\}\leq\max\{2,d(P)-1\}=d(\mathcal EP),
\end{equation*}
as required by condition (C2) for $\mathcal EP$.

\emph{Case $r_{\langle\rangle}(P)=(\lor_i,\psi_0,\psi_1)$:} Set $r_{\langle\rangle}(\mathcal EP):=(\lor_i,\psi_0,\psi_1)$ and $n(\mathcal EP,a):=\mathcal E n(P,a)$. The verification of (L), (C1) and (C2) is similar to the previous case.

\emph{Case $r_{\langle\rangle}(P)=(\forall_x,\psi)$:} Set $r_{\langle\rangle}(\mathcal EP):=(\forall_x,\psi)$ and $n(\mathcal EP,a):=\mathcal E n(P,a)$. Concerning the ordinal height, we have
\begin{equation*}
 o(n(\mathcal EP,a))+\omega=\Omega^{o(n(P,a))}+\omega<\Omega^{o(n(P,a))+1}\leq\Omega^{o(P)}=o(\mathcal EP).
\end{equation*}
The other conditions are checked as in the previous cases.

\emph{Case $r_{\langle\rangle}(P)=(\exists_x,b,\psi)$:} Set $r_{\langle\rangle}(\mathcal EP):=(\exists_x,b,\psi)$ and $n(\mathcal EP,a):=\mathcal E n(P,a)$. Observe that we have
\begin{equation*}
 |b|<o(P)\leq\Omega^{o(P)}=o(\mathcal EP),
\end{equation*}
as demanded by condition (L). The other conditions are verified as in above.

\emph{Case $r_{\langle\rangle}(P)=(\rref,\exists_z\forall_{x\in a}\exists_{y\in z}\theta)$:} We set $r_{\langle\rangle}(\mathcal EP):=(\rref,\exists_z\forall_{x\in a}\exists_{y\in z}\theta)$ and $n(\mathcal EP,a):=\mathcal E n(P,a)$. Observe that we have
\begin{equation*}
 \Omega\leq o(P)\leq\Omega^{o(P)}=o(\mathcal EP).
\end{equation*}
The other conditions are checked as in the previous cases.

\emph{Case $r_{\langle\rangle}(P)=(\rep,b)$:} Set $r_{\langle\rangle}(\mathcal EP):=(\rep,b)$ and $n(\mathcal EP,a):=\mathcal E n(P,a)$. The verification of (L), (C1) and (C2) is similar to the previous cases.

\emph{Case $r_{\langle\rangle}(P)=(\cut,\psi)$ with $\hth(\psi)\leq 1$:} In this case we set $r_{\langle\rangle}(\mathcal EP):=(\cut,\psi)$ and $n(\mathcal EP,a)=\mathcal En(P,a)$. Condition (C1) follows from $\hth(\psi)\leq 1<2\leq d(\mathcal EP)$. Conditions (L) and (C2) are checked as in the previous cases.

\emph{Case $r_{\langle\rangle}(P)=(\cut,\exists_x\psi)$ with $\hth(\exists_x\psi)>1$:} Here we set
\begin{align*}
 r_{\langle\rangle}(\mathcal EP)&:=(\rep,0),\\
 n(\mathcal EP,a)&:=\mathcal R_{\exists_x\psi}(\mathcal En(P,0))(\mathcal En(P,1)).
\end{align*}
Let us verify condition (L): Using Lemma~\ref{lem:cut-reduction} and condition (L) for $P$ we get
\begin{multline*}
 o_{\langle\rangle}(n(\mathcal EP,0))=o_{\langle\rangle}(\mathcal En(P,1))+o_{\langle\rangle}(\mathcal En(P,0))=\\
=\Omega^{o_{\langle\rangle}(n(P,1))}+\Omega^{o_{\langle\rangle}(n(P,0))}<\Omega^{o_{\langle\rangle}(P)}=o_{\langle\rangle}(\mathcal EP).
\end{multline*}
Also, we have
\begin{multline*}
 l_{\langle\rangle}(n(\mathcal EP,0))=l_{\langle\rangle}(\mathcal R_{\exists_x\psi}(\mathcal En(P,0))(\mathcal En(P,1)))=\\
=(l_{\langle\rangle}(\mathcal En(P,0))\backslash\{\exists_x\psi\})\cup(l_{\langle\rangle}(\mathcal En(P,1))\backslash\{\forall_x\neg\psi\})=\\
=(l_{\langle\rangle}(n(P,0))\backslash\{\exists_x\psi\})\cup(l_{\langle\rangle}(n(P,1))\backslash\{\forall_x\neg\psi\})\subseteq l_{\langle\rangle}(P)=l_{\langle\rangle}(\mathcal EP).
\end{multline*}
Here, the last line relies on condition (L) for $P$, which provides the inclusions $l_{\langle\rangle}(n(P,0))\subseteq l_{\langle\rangle}(P)\cup\{\exists_x\psi\}$ and $l_{\langle\rangle}(n(P,1))\subseteq l_{\langle\rangle}(P)\cup\{\forall_x\neg\psi\}$. Condition (C1) does not apply. Condition (C2) holds by
\begin{multline*}
 d(n(\mathcal EP,0))=\max\{d(\mathcal En(P,0)),d(\mathcal En(P,1)),\hth(\exists_x\psi)\}=\\
=\max\{2,d(n(P,0))-1,d(n(P,1))-1,\hth(\exists_x\psi)\}=\\
=\max\{2,d(P)-1,\hth(\exists_x\psi)\}=\max\{2,d(P)-1\}=d(\mathcal EP).
\end{multline*}
The last line uses $\hth(\exists_x\psi)<d(P)$, as provided by condition (C1) for $P$.

\emph{Case $r_{\langle\rangle}(P)=(\cut,\forall_x\psi)$ with $\hth(\forall_x\psi)>1$:} We set
\begin{align*}
 r_{\langle\rangle}(\mathcal EP)&:=(\rep,0),\\
 n(\mathcal EP,a)&:=\mathcal R_{\exists_x\neg\psi}(\mathcal En(P,1))(\mathcal En(P,0)).
\end{align*}
The verification of conditions (L), (C1) and (C2) is similar to the previous case.

\emph{Case $r_{\langle\rangle}(P)=(\cut,\psi_0\lor\psi_1)$ with $\hth(\psi_0\lor\psi_1)>1$:} Set
\begin{align*}
 r_{\langle\rangle}(\mathcal EP)&:=(\rep,0),\\
 n(\mathcal EP,a)&:=\mathcal R_{\psi_0\lor\psi_1}(\mathcal En(P,0))(\mathcal En(P,1)).
\end{align*}
The verification of conditions (L), (C1) and (C2) is similar to the previous cases.

\emph{Case $r_{\langle\rangle}(P)=(\cut,\psi_0\land\psi_1)$ with $\hth(\psi_0\land\psi_1)>1$:} Set
\begin{align*}
 r_{\langle\rangle}(\mathcal EP)&:=(\rep,0),\\
 n(\mathcal EP,a)&:=\mathcal R_{\neg\psi_0\lor\neg\psi_1}(\mathcal En(P,1))(\mathcal En(P,0)).
\end{align*}
The verification of conditions (L), (C1) and (C2) is similar to the previous cases.
\end{proof}

In the following we write $\mathcal E^CP$ for the $\mathbf L_{\omega^\alpha}^u$-code $\mathcal E\cdots\mathcal EP$ with $C$ occurrences of the function symbol $\mathcal E$.

\begin{proposition}\label{prop:without-cuts}
 Let $C\in\omega$ be as in Lemma~\ref{lem:cut-rank-search-trees}. The $\mathbf L_{\omega^\alpha}^u$-preproof~$[\mathcal E^CP_\alpha^u\langle\rangle]$ has empty end-sequent, height $\varepsilon_{\langle\rangle}\in\varepsilon(S_{\omega^\alpha}^u)$, and cut-rank $2$.
\end{proposition}
\begin{proof}
 Using Proposition~\ref{prop:cut-elimination-codes} and Proposition~\ref{prop:proof-from-search-tree} we get
\begin{equation*}
 l_{[\mathcal E^CP_\alpha^u\langle\rangle]}(\langle\rangle)=l_{\langle\rangle}(\mathcal E^CP_\alpha^u\langle\rangle)=l_{\langle\rangle}(P_\alpha^u\langle\rangle)=l_\alpha^u(\langle\rangle)=\langle\rangle,
\end{equation*}
which means that $[\mathcal E^CP_\alpha^u\langle\rangle]$ has empty end-sequent. In view of $\Omega^{\varepsilon_{\langle\rangle}}=\varepsilon_{\langle\rangle}$ a similar argument shows $o_{[\mathcal E^CP_\alpha^u\langle\rangle]}(\langle\rangle)=\varepsilon_{\langle\rangle}$. Proposition~\ref{prop:cut-elimination-codes} and Definition~\ref{def:cut-ranks-codes} give
\begin{equation*}
 d(\mathcal E^CP_\alpha^u\langle\rangle)=\max\{2,d(P_\alpha^u\langle\rangle)-C\}=2.
\end{equation*}
By Corollary~\ref{cor:cut-rank-codes} this implies that $[\mathcal E^CP_\alpha^u\langle\rangle]$ has cut-rank $2$.
\end{proof}

Let us stress once more that the functions $l_{\langle\rangle},r_{\langle\rangle},o_{\langle\rangle},n$, defined by recursion on the length of $\mathbf L_{\omega^\alpha}^u$-codes, are primitive recursive. Thus a function such as
\begin{equation*}
 (\sigma,\alpha,u,C)\mapsto l_{[\mathcal E^CP_\alpha^u\langle\rangle]}(\sigma)=l_{\langle\rangle}(\bar n(\mathcal E^CP_\alpha^u\langle\rangle,\sigma))
\end{equation*}
is primitive recursive as well. Let us also point out that no Bachmann-Howard collapse $\vartheta:\varepsilon(S_{\omega^\alpha}^u)\bh\alpha$ was needed for the continuous cut-elimination procedure of the present section. This will be different in the next section, where we collapse $\mathbf L_{\omega^\alpha}^u$-preproofs to height below $\Omega$.

\section{Collapsing Proofs}

Let us recall our overall goal: We want to establish that any countable transitive set $u$ is contained in some admissible set. In the last paragraph of Section~\ref{sect:proof-trees} we have described the following plan to achieve this: By Proposition \ref{prop:proof-from-search-tree} we have $\mathbf L_{\omega^\alpha}^u$-preproofs $P_\alpha^u$ with empty end-sequent and ordinal height $\varepsilon_{\langle\rangle}>\Omega$. On the other hand, Proposition~\ref{prop:countable-proof-sound} tells us that no $\mathbf L_{\omega^\alpha}^u$-preproof with empty end-sequent can have ordinal height below $\Omega$. Now assume, aiming at a contradiction, that there is no admissible set $\mathbb A$ with $u\subseteq\mathbb A$. Then Proposition~\ref{prop:no-admissible-gives-wop} implies that $\alpha\mapsto\varepsilon(S_{\omega^\alpha}^u)$ is a well-ordering principle. So the higher Bachmann-Howard principle from Definition~\ref{def:higher-bh-principle} yields a collapsing function $\vartheta:\varepsilon(S_{\omega^\alpha}^u)\bh\alpha$, for some ordinal~$\alpha$. Using this function we want to collapse $P_\alpha^u$ to ordinal height below $\Omega$. This is the desired contradiction, showing that $u$ is contained in an admissible set. The previous section was devoted to a preliminary transformation, turning $P_\alpha^u$ it into a preproof of cut-rank $2$. In this section we will present the collapsing procedure itself.

A particularly elegant description of collapsing relies on the notion of operator control, due to Buchholz \cite{buchholz-local-predicativity}. As controlling operators we use the functions
\begin{equation*}
 \mathcal H_t^\vartheta[\beta]:\gamma\mapsto\mathcal H_t^\vartheta[\beta,\gamma]
\end{equation*}
introduced in Definition~\ref{def:operators}. Here we assume that $t\in\varepsilon(S_{\omega^\alpha}^u)$ satisfies $\Omega\leq t$, that we have $\beta,\gamma<\omega^\alpha$, and that $\vartheta:\varepsilon(S_{\omega^\alpha}^u)\bh\alpha$ is a Bachmann-Howard collapse. In Section~\ref{sect:higher-wop} we have seen that $\mathcal H_t^\vartheta[\beta,\gamma]$ can be computed by a primitive recursive function in $t,\beta,\gamma,\vartheta$ (and $u,\alpha$). The functions $\mathcal H_t^\vartheta[\beta]$, which restrict this primitive recursive class function to a set, will thus exist as sets. Before we can give a definition of operator controlled proof, we must specify which parameters we wish to control:

\begin{definition}
For an $\mathbf L_{\omega^\alpha}^u$-formula $\psi$ we set
\begin{equation*}
 k(\psi):=\max(\{|a|\,;\,\text{the parameter $a$ occurs in $\psi$}\}\cup\{0\}).
\end{equation*}
Given a sequent $\Gamma=\langle\psi_1,\dots,\psi_k\rangle$ we set
\begin{equation*}
 k(\Gamma):=\max\{k(\psi_1),\dots ,k(\psi_k),0\}.
\end{equation*}
Concerning parameters in rules, we put
\begin{gather*}
 k((\exists_x,a,\psi)):=k((\rep,a)):=|a|,\\
 k((\cut,\psi)):=k(\psi),\\
\end{gather*}
and $k(r):=0$ for any rule of a different form. Finally, we put
\begin{equation*}
 k_P(\sigma):=\max\{o(\sigma)^*,k(l(\sigma)),k(r(\sigma))\}.
\end{equation*}
for an $\mathbf L_{\omega^\alpha}^u$-preproof $(P,l,r,o)$ and a node $\sigma\in P$.
\end{definition}

The reader may wish to recall Notation~\ref{notation:abs-and-star}: In particular $o(\sigma)^*$ refers to the rank function of the well-ordering principle $\alpha\mapsto\varepsilon(S_{\omega^\alpha}^u)$. As we consider $o(\sigma)^*$ rather than $o(\sigma)$ we have $k_P(\sigma)<\omega^\alpha$. We can now say when an operator controls a proof:

\begin{definition}
 Assume $\vartheta:\varepsilon(S_{\omega^\alpha}^u)\bh\alpha$. We say that the operator $\mathcal H_t^\vartheta[\beta]$ controls the $\mathbf L_{\omega^\alpha}^u$-preproof $P$ if we have
\begin{equation*}
 k_P(\sigma)\in\mathcal H_t^\vartheta[\beta,|\sigma|]
\end{equation*}
for all nodes $\sigma\in P$.
\end{definition}

Our first goal is operator control for the ``basic'' preproofs $P_\alpha^u$ from Proposition~\ref{prop:proof-from-search-tree}. The following is a preparation (observe $\mathcal H_t^\vartheta[0,\gamma]=\mathcal H_t^\vartheta[\gamma]$, by Lemma~\ref{lem:operators-max-counts}):

\begin{lemma}
 Assume $\vartheta:\varepsilon(S_{\omega^\alpha}^u)\bh\alpha$. The $\mathbf L_{\omega^\alpha}^u$-preproofs of the Kripke-Platek axioms, as constructed in Lemma~\ref{lem:proofs-kp-axioms}, are controlled by the operator $\mathcal H_\Omega^\vartheta$.
\end{lemma}
\begin{proof}
As an example, consider a node of the form $\sigma=\langle c_1,\dots ,c_k,a\rangle$ in a proof $P$ of $\Delta_0$-separation. There we have
\begin{equation*}
k_P(\sigma)=\max\{\Omega^*,|a|,|b|,|c_1|,\dots,|c_k|\}
\end{equation*}
with $b=\{z\in a\,|\,\theta(a,z,\vec c)\}$. We have observed $|b|<\max\{|a|,|c_1|,\dots,|c_k|\}+\omega$ in the proof of Lemma~\ref{lem:proofs-kp-axioms}. Together with $\Omega^*=0$ we get
\begin{equation*}
 k_P(\sigma)<\max\{|a|,|c_1|,\dots,|c_k|\}+\omega=|\sigma|+\omega.
\end{equation*}
From Lemma~\ref{lem:operators-closure} we know $|\sigma|\in\mathcal H_\Omega^\vartheta[|\sigma|]$. Lemma~\ref{lem:operators-nice} gives $\omega\in\mathcal H_\Omega^\vartheta[|\sigma|]$ and then $|\sigma|+\omega\in\mathcal H_\Omega^\vartheta[|\sigma|]$. In view of $k_P(\sigma)<|\sigma|+\omega<\Omega$ we get $k_P(\sigma)\in\mathcal H_\Omega^\vartheta[|\sigma|]$. The other cases are similar and left to the reader.
\end{proof}

Building on this, let us look at the preproofs from Proposition~\ref{prop:proof-from-search-tree}:

\begin{lemma}\label{lem:basic-proofs-controlled}
 Assume $\vartheta:\varepsilon(S_{\omega^\alpha}^u)\bh\alpha$. The $\mathbf L_{\omega^\alpha}^u$-preproof $P_\alpha^u$ is controlled by the operator $\mathcal H_\Omega^\vartheta$.
\end{lemma}
\begin{proof}
 First, consider a node $\sigma$ in the search tree $S_{\omega^\alpha}^u\subseteq P_\alpha^u$. Let us begin by looking at the ``ordinal'' label $o_\alpha^u(\sigma)=\varepsilon_\sigma$. In view of $|\sigma|\in\mathcal H_\Omega^\vartheta[|\sigma|]$ Lemma~\ref{lem:operators-nice} gives $\varepsilon_\sigma\in\mathcal H_\Omega^\vartheta[|\sigma|]$. By Lemma~\ref{lem:basic-c-sets} we obtain ${\varepsilon_\sigma}^*\in\mathcal H_\Omega^\vartheta[|\sigma|]$. Thus we have established $k_{P_\alpha^u}(\sigma)\in\mathcal H_\Omega^\vartheta[|\sigma|]$ in case $k_P(\sigma)=o_\alpha^u(\sigma)^*$. Now let us look at $k(l_\alpha^u(\sigma))$ and $k(r_\alpha^u(\sigma))$: In the proof of Lemma~\ref{lem:properties-search-tree} we have seen that all parameters in $l_\alpha^u(\sigma)$ lie in the set $\rng(\sigma)\cup u$; the same argument applies to parameters of the rule $r_\alpha^u(\sigma)$. As we have $|u_i|=|u_i|_{\mathbf L}^u=0$ for all $u_i\in u$ we get $\max\{k(l_\alpha^u(\sigma)),k(r_\alpha^u(\sigma))\}\leq|\sigma|$, which implies $\max\{k(l_\alpha^u(\sigma)),k(r_\alpha^u(\sigma))\}\in\mathcal H_\Omega^\vartheta[|\sigma|]$.

It remains to look at a node $\sigma\in P_\alpha^u$ which does not lie in the search tree $S_\alpha^u$. Then $\sigma$ is of the form ${\sigma_0}^\frown \sigma_1$ where $\sigma_1$ lies in one of the proofs of the Kripke-Platek axioms, constructed in Lemma~\ref{lem:proofs-kp-axioms}. From the previous lemma we get $k_{P_\alpha^u}(\sigma)\in\mathcal H_\Omega^\vartheta[|\sigma_1|]$. In view of $|\sigma_1|\leq|\sigma|$ Lemma~\ref{lem:operators-max-counts} gives $\mathcal H_\Omega^\vartheta[|\sigma_1|]\subseteq\mathcal H_\Omega^\vartheta[|\sigma|]$, and the claim follows.
\end{proof}

In the previous section we have developed $\mathbf L_{\omega^\alpha}^u$-codes as an important tool for the construction of $\mathbf L_{\omega^\alpha}^u$-preproofs. We need to recast the notion of operator control in terms of these codes. The bound on parameters is easily defined for codes:

\begin{definition}
For any $\mathbf L_{\omega^\alpha}^u$-code $P$ we set
\begin{equation*}
 k_{\langle\rangle}(P):=\max\{o_{\langle\rangle}(P)^*,k(l_{\langle\rangle}(P)),k(r_{\langle\rangle}(P))\}.
\end{equation*}
\end{definition}

As for the controlling operators, we can of course represent $\mathcal H_t^\vartheta[\beta]$ by the pair $\langle t,\beta\rangle$. Let us restrict attention to the basic $\mathbf L_{\omega^\alpha}^u$-codes from Definition~\ref{def:basic-codes} first:

\begin{definition}\label{def:operators-basic-proofs}
 We define an assignment
\begin{align*}
 h_0:\text{``basic $\mathbf L_{\omega^\alpha}^u$-codes''}&\rightarrow\{t\in\varepsilon(S_{\omega^\alpha}^u)\,|\,\Omega\leq t\}, && h_0(P_\alpha^u\sigma):=\Omega,\\
 h_1:\text{``basic $\mathbf L_{\omega^\alpha}^u$-codes''}&\rightarrow\omega^\alpha, && h_1(P_\alpha^u\sigma):=|\sigma|
\end{align*}
of operators to basic $\mathbf L_{\omega^\alpha}^u$-codes.
\end{definition}

We will abbreviate
\begin{equation*}
 \mathcal H_P^\vartheta[\beta_1,\dots,\beta_k]=\mathcal H_{h_0(P)}^\vartheta[h_1(P),\beta_1,\dots,\beta_k].
\end{equation*}
As in the previous section, the point of the following is that it can be extended beyond basic codes:

\begin{lemma}\label{lem:locally-correct-codes}
 Assume $\vartheta:\varepsilon(S_{\omega^\alpha}^u)\bh\alpha$. The assigment of operators to basic $\mathbf L_{\omega^\alpha}^u$-codes is locally correct in the sense that we have
\begin{enumerate}[label=(H\arabic*)]
 \item $k_{\langle\rangle}(P)\in\mathcal H_P^\vartheta$,
 \item $h_0(n(P,a))\leq h_0(P)$,
 \item $\{h_1(n(P,a)),o_{\langle\rangle}(n(P,a))\}\subseteq\mathcal H_P^\vartheta[|a|]$
\end{enumerate}
for any basic $\mathbf L_{\omega^\alpha}^u$-code $P$ and all $a\in\iota(r_{\langle\rangle}(P))$.
\end{lemma}
\begin{proof}
Write $P=P_\alpha^u\sigma$. To verify (H1) we have to distinguish two cases: Assume first that $\sigma$ is a node in the preproof $P_\alpha^u$. Then we have $o_{\langle\rangle}(P_\alpha^u\sigma)=o_\alpha^u(\sigma)$, where $o_\alpha^u$ is the labelling function of $P_\alpha^u$. The same holds for the other labelling functions, so that we see $k_{\langle\rangle}(P_\alpha^u\sigma)=k_{P_\alpha^u}(\sigma)$. Using Lemma~\ref{lem:basic-proofs-controlled} we get
\begin{equation*}
 k_{\langle\rangle}(P_\alpha^u\sigma)\in\mathcal H_\Omega^\vartheta[|\sigma|]=\mathcal H_P^\vartheta.
\end{equation*}
If the node $\sigma$ does not lie in $P_\alpha^u$ then we have $k_{\langle\rangle}(P_\alpha^u\sigma)=0$ by default, so (H1) holds in this case as well. Condition (H2) is trivial for basic codes. As for (H3), recall that $n(P_\alpha^u\sigma,a)$ was defined to be the basic code $P_\alpha^u\sigma^\frown a$. Thus we have
\begin{equation*}
 h_1(n(P_\alpha^u\sigma,a))=|\sigma^\frown a|=\max\{|\sigma|,|a|\}\in\mathcal H_\Omega^\vartheta[|\sigma|,|a|]=\mathcal H_{P_\alpha^u\sigma}^\vartheta[|a|].
\end{equation*}
Finally, condition (H1) for $n(P_\alpha^u\sigma,a)=P_\alpha^u\sigma^\frown a$ gives
\begin{equation*}
 o_{\langle\rangle}(n(P_\alpha^u\sigma,a))^*\in\mathcal H_{P_\alpha^u\sigma^\frown a}^\vartheta=\mathcal H_{P_\alpha^u\sigma}^\vartheta[|a|].
\end{equation*}
By Lemma~\ref{lem:basic-c-sets} we get $o_{\langle\rangle}(n(P_\alpha^u\sigma,a))\in\mathcal H_{P_\alpha^u\sigma}^\vartheta[|a|]$, as required for the second part of (H3).
\end{proof}
Having seen the last part of this proof, the reader may wonder whether the condition $o_{\langle\rangle}(n(P,a))\in\mathcal H_P^\vartheta[|a|]$ in (H3) is redundant. This is indeed the case once we have established conditions (H1-H3) for \emph{all} $\mathbf L_{\omega^\alpha}^u$-codes. However, in order to establish (H1-H3) by induction on the length of codes we will need condition (H3) as it stands. The following relies on the interpretation of codes as proofs (see~Definition~\ref{def:interpretation-codes}):

\begin{corollary}\label{cor:operator-code-to-proof}
 Assume $\vartheta:\varepsilon(S_{\omega^\alpha}^u)\bh\alpha$. For any basic $\mathbf L_{\omega^\alpha}^u$-code $P$, the $\mathbf L_{\omega^\alpha}^u$-preproof $[P]$ is controlled by the operator $\mathcal H_P^\vartheta$.
\end{corollary}
\begin{proof}
 Recall that we have $o_{[P]}(\sigma)=o_{\langle\rangle}(\bar n(P,\sigma))$ for any node $\sigma\in[P]$. The same holds for the other labels, so we get $k_{[P]}(\sigma)=k_{\langle\rangle}(\bar n(P,\sigma))$. By condition (H1) from the previous lemma we obtain
\begin{equation*}
 k_{[P]}(\sigma)\in\mathcal H_{\bar n(P,\sigma)}^\vartheta.
\end{equation*}
Thus it suffices to establish
\begin{equation*}
 \mathcal H_{\bar n(P,\sigma)}^\vartheta\subseteq\mathcal H_P^\vartheta[|\sigma|].
\end{equation*}
We prove the stronger claim
\begin{equation*}
 \mathcal H_{\bar n(P,\tau)}^\vartheta[|\sigma|]\subseteq\mathcal H_P^\vartheta[|\sigma|].
\end{equation*}
by induction on initial segments $\tau$ of $\sigma$: For $\tau=\langle\rangle$ we have $\bar n(P,\tau)=P$ and the claim is trivial. In the induction step, write $\tau=\rho^\frown a$. In view of $\bar n(P,\tau)=n(\bar n(P,\rho),a)$ condition (H3) yields
\begin{equation*}
 h_1(\bar n(P,\tau))\in\mathcal H_{\bar n(P,\rho)}^\vartheta[|a|].
\end{equation*}
Since $\tau=\rho^\frown a$ is an initial segment of $\sigma$ we have $|a|\leq|\sigma|$. Together with the induction hypothesis we obtain
\begin{equation*}
 h_1(\bar n(P,\tau))\in\mathcal H_{\bar n(P,\rho)}^\vartheta[|\sigma|]\subseteq\mathcal H_P^\vartheta[|\sigma|].
\end{equation*}
By Lemma~\ref{lem:operators-closure} this implies
\begin{equation*}
 \mathcal H_{h_0(P)}^\vartheta[h_1(\bar n(P,\tau)),|\sigma|]\subseteq\mathcal H_P^\vartheta[|\sigma|].
\end{equation*}
Iterative applications of (H2) yield $h_0(\bar n(P,\tau))\leq h_0(P)$. Using Proposition~\ref{prop:operators-special} we can thus conclude
\begin{equation*}
 \mathcal H_{\bar n(P,\tau)}^\vartheta[|\sigma|]=\mathcal H_{h_0(\bar n(P,\tau))}^\vartheta[h_1(\bar n(P,\tau)),|\sigma|]\subseteq\mathcal H_{h_0(P)}^\vartheta[h_1(\bar n(P,\tau)),|\sigma|]\subseteq\mathcal H_P^\vartheta[|\sigma|],
\end{equation*}
as required for the induction step.
\end{proof}

Similar to the treatment of cut-rank in the previous section, we can now extend the assignment of operators beyond basic codes: To do so it is enough to find an extension of $h_0$ and $h_1$ which remains locally correct in the sense of conditions (H1) to (H3). The (proof of the) corollary will guarantee that the assigned operators do indeed control the interpretations of the codes.

\begin{lemma}\label{lem:operators-cut-elimination}
 Assume $\vartheta:\varepsilon(S_{\omega^\alpha}^u)\bh\alpha$. The operator assigment $\langle h_0,h_1\rangle$ from Definition~\ref{def:operators-basic-proofs} can be extended to all $\mathbf L_{\omega^\alpha}^u$-codes constructed in the previous section. In particular this assigment gives $h_i(\mathcal EP)=h_i(P)$ for $i=0,1$ and any $\mathbf L_{\omega^\alpha}^u$-code $P$.
\end{lemma}
\begin{proof}
 We set $h_0(P)=\Omega$ for all codes $P$ from the previous section (different values of $h_0(P)$ will occur later in this section). So condition (H2) from Lemma~\ref{lem:locally-correct-codes} is immediate. The value $h_1(P)$ is defined by recursion over the code $P$: Definition~\ref{def:operators-basic-proofs} accounts for the base case of a basic $\mathbf L_{\omega^\alpha}^u$-code. As a first recursive case, consider a term of the form $\mathcal I_{\forall_x\psi,b}P$. We put
\begin{equation*}
 h_1(\mathcal I_{\forall_x\psi,b}P):=\max\{h_1(P),k(\forall_x\psi),|b|\}.
\end{equation*}
 This is designed to satisfy condition (H1): Observe
\begin{equation*}
 k(l_{\langle\rangle}(\mathcal I_{\forall_x\psi,b}P))=\max\{k(l_{\langle\rangle}(P)),k(\psi(b))\}\leq\max\{k_{\langle\rangle}(P),k(\forall_x\psi),|b|\}
\end{equation*}
and
\begin{equation*}
 k(r_{\langle\rangle}(\mathcal I_{\forall_x\psi,b}P))=\max\{k(r_{\langle\rangle}(P)),|b|\}\leq\max\{k_{\langle\rangle}(P),|b|\},
\end{equation*}
where $|b|$ accounts for the possibility of a new rule $r_{\langle\rangle}(\mathcal I_{\forall_x\psi,b}P)=(\rep,b)$. Together with $o_{\langle\rangle}(\mathcal I_{\forall_x\psi,b}P)=o_{\langle\rangle}(P)$ this implies
\begin{equation*}
 k_{\langle\rangle}(\mathcal I_{\forall_x\psi,b}P)\leq\max\{k_{\langle\rangle}(P),k(\forall_x\psi),|b|\}.
\end{equation*}
To establish (H1) for $\mathcal I_{\forall_x\psi,b}P$ it is thus enough to show
\begin{equation*}
 \{k_{\langle\rangle}(P),k(\forall_x\psi),|b|\}\subseteq\mathcal H_{\mathcal I_{\forall_x\psi,b}P}^\vartheta.
\end{equation*}
By the induction hypothesis (condition (H1) for $P$) and $h_1(P)\leq h_1(\mathcal I_{\forall_x\psi,b}P)$ we get
\begin{equation*}
 k_{\langle\rangle}(P)\in\mathcal H_P^\vartheta\subseteq\mathcal H_{\mathcal I_{\forall_x\psi,b}P}^\vartheta.
\end{equation*}
By definition of the operators we have $h_1(\mathcal I_{\forall_x\psi,b}P)\in\mathcal H_{\mathcal I_{\forall_x\psi,b}P}^\vartheta$. Together with $\max\{k(\forall_x\psi),|b|\}\leq h_1(\mathcal I_{\forall_x\psi,b}P)$ this implies
\begin{equation*}
 \{k(\forall_x\psi),|b|\}\subseteq\mathcal H_{\mathcal I_{\forall_x\psi,b}P}^\vartheta,
\end{equation*}
which completes the verification of (H1) for $\mathcal I_{\forall_x\psi,b}P$. As for condition (H3), in view of $n(\mathcal I_{\forall_x\psi,b}P,a)=\mathcal I_{\forall_x\psi,b}n(P,a)$ we have
\begin{equation*}
 h_1(n(\mathcal I_{\forall_x\psi,b}P,a))=\max\{h_1(n(P,a)),k(\forall_x\psi),|b|\}.
\end{equation*}
Observe $\iota(r_{\langle\rangle}(\mathcal I_{\forall_x\psi,b}P))\subseteq\iota(r_{\langle\rangle}(P))$. Thus condition (H3) for $P$ gives
\begin{equation*}
 h_1(n(P,a))\in\mathcal H_P^\vartheta[|a|]\subseteq\mathcal H_{\mathcal I_{\forall_x\psi,b}P}^\vartheta[|a|]
\end{equation*}
for all $a\in\iota(r_{\langle\rangle}(\mathcal I_{\forall_x\psi,b}P))$. Together with the above we get
\begin{equation*}
 h_1(n(\mathcal I_{\forall_x\psi,b}P,a))\in\mathcal H_{\mathcal I_{\forall_x\psi,b}P}^\vartheta[|a|].
\end{equation*}
Similarly, we can infer
\begin{equation*}
 o_{\langle\rangle}(n(\mathcal I_{\forall_x\psi,b}P,a))=o_{\langle\rangle}(n(P,a))\in\mathcal H_P^\vartheta[|a|]\subseteq\mathcal H_{\mathcal I_{\forall_x\psi,b}P}^\vartheta[|a|]
\end{equation*}
from condition (H3) for $P$.

For terms of the form $\mathcal I_{\psi_0\land\psi_1,i}P$ we set
\begin{equation*}
 h_1(\mathcal I_{\psi_0\land\psi_1,i}P):=\max\{h_1(P),k(\psi_i)\}.
\end{equation*}
Conditions (H1) to (H3) are verified as above (we now have to accomodate a new rule $(\rep,i)$, but $i\in\mathcal H_t[\beta]$ is automatic for $i=0,1$).

Let us come to the case of a term $\mathcal R_{\exists_x\psi}P_0P_1$: Here we put
\begin{equation*}
 h_1(\mathcal R_{\exists_x\psi}P_0P_1):=\max\{h_1(P_0),h_1(P_1)\}.
\end{equation*}
As a preparation for (H1), let us show
\begin{equation*}
 k_{\langle\rangle}(\mathcal R_{\exists_x\psi}P_0P_1)\leq\max\{k_{\langle\rangle}(P_0),k_{\langle\rangle}(P_1)\}.
\end{equation*}
Concerning the ``ordinal'' labels, we have
\begin{multline*}
 o_{\langle\rangle}(\mathcal R_{\exists_x\psi}P_0P_1)^*=(o_{\langle\rangle}(P_1)+o_{\langle\rangle}(P_0))^*\leq\\
\leq\max\{o_{\langle\rangle}(P_0)^*,o_{\langle\rangle}(P_1)^*\}\leq\max\{k_{\langle\rangle}(P_0),k_{\langle\rangle}(P_1)\}.
\end{multline*}
As for the end sequents, $l_{\langle\rangle}(\mathcal R_{\exists_x\psi}P_0P_1)\subseteq l_{\langle\rangle}(P_0)\cup l_{\langle\rangle}(P_1)$ implies
\begin{equation*}
 k(l_{\langle\rangle}(\mathcal R_{\exists_x\psi}P_0P_1))\leq\max\{k(l_{\langle\rangle}(P_0)),k(l_{\langle\rangle}(P_1))\}\leq\max\{k_{\langle\rangle}(P_0),k_{\langle\rangle}(P_1)\}.
\end{equation*}
Concerning the last rule of the preproof $\mathcal R_{\exists_x\psi}P_0P_1$, the only interesting case is $r_{\langle\rangle}(P_0)=(\exists,b,\psi)$ and $r_{\langle\rangle}(\mathcal R_{\exists_x\psi}P_0P_1)=(\cut,\psi(b))$. Here we have
\begin{equation*}
 |b|\leq k(r_{\langle\rangle}(P_0))\leq k_{\langle\rangle}(P_0).
\end{equation*}
Also, condition (L) for $P_0$ implies that the formula $\exists_x\psi$ occurs in $l_{\langle\rangle}(P_0)$. Thus we obtain
\begin{equation*}
 k(r_{\langle\rangle}(\mathcal R_{\exists_x\psi}P_0P_1))=k(\psi(b))\leq\max\{k(\exists_x\psi),|b|\}\leq k_{\langle\rangle}(P_0).
\end{equation*}
Now that we know $k_{\langle\rangle}(\mathcal R_{\exists_x\psi}P_0P_1)\leq\max\{k_{\langle\rangle}(P_0),k_{\langle\rangle}(P_1)\}$ condition (H1) is easily established: From the induction hypothesis
\begin{equation*}
 k_{\langle\rangle}(P_i)\in\mathcal H_{P_i}^\vartheta\subseteq\mathcal H_{\mathcal R_{\exists_x\psi}P_0P_1}^\vartheta
\end{equation*}
we can infer
\begin{equation*}
 k_{\langle\rangle}(\mathcal R_{\exists_x\psi}P_0P_1)\in\mathcal H_{\mathcal R_{\exists_x\psi}P_0P_1}^\vartheta.
\end{equation*}
To verify condition (H3) we need to distinguish two cases: Assume first that we have $r_{\langle\rangle}(P_0)=(\exists_x,b,\psi)$ and $a=1$. Then $n(\mathcal R_{\exists_x\psi}P_0P_1,a)$ was defined to be $\mathcal I_{\forall_x\neg\psi,b}P_1$. As observed above this implies $\max\{k(\forall_x\neg\psi),|b|\}\leq k_{\langle\rangle}(P_0)$. Thus we get
\begin{equation*}
 h_1(\mathcal I_{\forall_x\neg\psi,b}P_1)=\max\{h_1(P_1),k(\forall_x\neg\psi),|b|\}\leq\max\{h_1(P_1),k_{\langle\rangle}(P_0)\}.
\end{equation*}
Now $h_1(P_1)\in\mathcal H_{P_1}^\vartheta$ and $k_{\langle\rangle}(P_0)\in\mathcal H_{P_0}^\vartheta$ (condition (H1) for $P_0$) imply
\begin{equation*}
 h_1(n(\mathcal R_{\exists_x\psi}P_0P_1,a))=h_1(\mathcal I_{\forall_x\neg\psi,b}P_1)\in\mathcal H_{\mathcal R_{\exists_x\psi}P_0P_1}^\vartheta\subseteq\mathcal H_{\mathcal R_{\exists_x\psi}P_0P_1}^\vartheta[|a|],
\end{equation*}
as required for (H3). Using condition (H1) for $P_1$ we also get
\begin{equation*}
 o_{\langle\rangle}(n(\mathcal R_{\exists_x\psi}P_0P_1,a))=o_{\langle\rangle}(\mathcal I_{\forall_x\neg\psi,b}P_1)=o_{\langle\rangle}(P_1)\in\mathcal H_{P_1}^\vartheta\subseteq\mathcal H_{\mathcal R_{\exists_x\psi}P_0P_1}^\vartheta[|a|].
\end{equation*}
It remains to consider the case $n(\mathcal R_{\exists_x\psi}P_0P_1,a)=\mathcal R_{\exists_x\psi}n(P_0,a)P_1$. Here we have
\begin{equation*}
 h_1(n(\mathcal R_{\exists_x\psi}P_0P_1,a))=\max\{h_1(n(P_0,a)),h_1(P_1)\}.
\end{equation*}
We have $h_1(P_1)\in\mathcal H_{P_1}^\vartheta\subseteq\mathcal H_{\mathcal R_{\exists_x\psi}P_0P_1}^\vartheta[|a|]$ and condition (H3) for $P_0$ yields
\begin{equation*}
 h_1(n(P_0,a))\in\mathcal H_{P_0}^\vartheta[|a|]\subseteq\mathcal H_{\mathcal R_{\exists_x\psi}P_0P_1}^\vartheta[|a|].
\end{equation*}
Finally, we have $o_{\langle\rangle}(P_1)\in\mathcal H_{P_1}^\vartheta$ and $o_{\langle\rangle}(n(P_0,a))\in\mathcal H_{P_0}^\vartheta[|a|]$ from (H1) for $P_1$ resp.\ (H3) for $P_0$. This implies
\begin{multline*}
 o_{\langle\rangle}(n(\mathcal R_{\exists_x\psi}P_0P_1,a))=o_{\langle\rangle}(\mathcal R_{\exists_x\psi}n(P_0,a)P_1)=\\
=o_{\langle\rangle}(P_1)+o_{\langle\rangle}(n(P_0,a))\in\mathcal H_{\mathcal R_{\exists_x\psi}P_0P_1}^\vartheta[|a|],
\end{multline*}
as required by condition (H3) for $\mathcal R_{\exists_x\psi}P_0P_1$.

The last remaining case is that of a term $\mathcal EP$: As announced in the statement of the lemma we set
\begin{equation*}
 h_1(\mathcal EP):=h_1(P).
\end{equation*}
Observing $(\Omega^s)^*\leq s^*$ it is easy to see
\begin{equation*}
 k_{\langle\rangle}(\mathcal EP)\leq k_{\langle\rangle}(P).
\end{equation*}
Thus condition (H1) for $P$ implies
\begin{equation*}
 k_{\langle\rangle}(\mathcal EP)\in\mathcal H_P^\vartheta=\mathcal H_{\mathcal EP}^\vartheta,
\end{equation*}
as required by condition (H1) for $\mathcal EP$. Concerning (H3), let us first assume that $r_{\langle\rangle}(P)$ is not a cut, or a cut over a formula of height at most one. Then we have $n(\mathcal EP,a)=\mathcal En(P,a)$. Then condition (H3) for $P$ yields
\begin{equation*}
 h_1(n(\mathcal EP,a))=h_1(\mathcal En(P,a))=h_1(n(P,a))\in\mathcal H_P^\vartheta[|a|]=\mathcal H_{\mathcal EP}^\vartheta[|a|],
\end{equation*}
as required by condition (H3) for $\mathcal EP$. Also, (H3) for $P$ gives $o_{\langle\rangle}(n(P,a))\in\mathcal H_P^\vartheta[|a|]$. Thus we get
\begin{equation*}
 o_{\langle\rangle}(n(\mathcal EP,a))=\Omega^{o_{\langle\rangle}(n(P,a))}\in\mathcal H_{\mathcal EP}^\vartheta[|a|].
\end{equation*}
Next, consider the case $r_{\langle\rangle}(P)=(\cut,\exists_x\psi)$ with $\hth(\exists_x\psi)>1$. Here we have
\begin{equation*}
 n(\mathcal EP,a)=\mathcal R_{\exists_x\psi}(\mathcal En(P,0))(\mathcal En(P,1)).
\end{equation*}
Condition (H3) for $P$ gives
\begin{equation*}
 h_1(n(P,i))\in\mathcal H_P^\vartheta[|i|]=\mathcal H_{\mathcal EP}^\vartheta.
\end{equation*}
Thus we get
\begin{equation*}
 h_1(n(\mathcal EP,a))=\max\{h_1(n(P,0)),h_1(n(P,1))\}\in\mathcal H_{\mathcal EP}^\vartheta[|a|].
\end{equation*}
Also, $o_{\langle\rangle}(n(P,i))\in\mathcal H_P^\vartheta[|i|]=\mathcal H_{\mathcal EP}^\vartheta$ implies
\begin{equation*}
 o_{\langle\rangle}(n(\mathcal EP,a))=\Omega^{o_{\langle\rangle}(n(P,1))}+\Omega^{o_{\langle\rangle}(n(P,0))}\in\mathcal H_{\mathcal EP}^\vartheta[|a|],
\end{equation*}
as required by condition (H3) for $\mathcal EP$. The remaining cases (cut-formulas of the forms $\forall_x\psi$, $\psi_0\lor\psi_1$ and~$\psi_0\land\psi_1$) are similar and left to the reader.
\end{proof}

In particular we obtain operator control for the preproof $[\mathcal E^CP_\alpha^u\langle\rangle]$ from Proposition~\ref{prop:without-cuts}. Recall that this proof had empty end-sequent and cut rank $2$ (but height above $\Omega$).

\begin{corollary}
 Assume $\vartheta:\varepsilon(S_{\omega^\alpha}^u)\bh\alpha$. For any $C\in\omega$ the $\mathbf L_{\omega^\alpha}^u$-preproof $[\mathcal E^CP_\alpha^u\langle\rangle]$ is controlled by the operator $\mathcal H_\Omega^\vartheta$.
\end{corollary}
\begin{proof}
 The previous lemma and Definition~\ref{def:operators-basic-proofs} give $h_0(\mathcal E^CP_\alpha^u\langle\rangle)=h_0(P_\alpha^u\langle\rangle)=\Omega$ and $h_1(\mathcal E^CP_\alpha^u\langle\rangle)=h_1(P_\alpha^u\langle\rangle)=|\langle\rangle|=0$. By Corollary~\ref{cor:operator-code-to-proof} this means that $[\mathcal E^CP_\alpha^u\langle\rangle]$ is controlled by the operator $\mathcal H_\Omega^\vartheta[0]=\mathcal H_\Omega^\vartheta$.
\end{proof}

In the first part of this section we have defined the notion of operator controlled $\mathbf L_{\omega^\alpha}^u$-preproof. We have seen how this notion can be formulated in terms of \mbox{$\mathbf L_{\omega^\alpha}^u$-codes}. Also, we have established operator control for all $\mathbf L_{\omega^\alpha}^u$-codes constructed so far. Building on this, we can now move towards the collapsing procedure for proofs. Let us introduce some terminology: Consider an $\mathbf L_{\omega^\alpha}^u$-formula $\varphi$ and an ordinal $\beta<\omega^\alpha$. We write $\varphi^\beta$ for the formula that results from $\varphi$ when we replace each unbounded quantifier $\forall_x\dots$ resp.~$\exists_x\dots$ by the bounded quantifier $\forall_{x\in\mathbb L_\beta^u}\dots$ resp.~$\exists_{x\in\mathbb L_\beta^u}\dots$. In view of $|\mathbb L_\beta^u|=\beta$ we can conceive $\mathbb L_\beta^u$ as an element of $\mathbf L_{\omega^\alpha}^u$, and $\varphi^\beta$ as an $\mathbf L_{\omega^\alpha}^u$-formula. Note that we have
\begin{equation*}
 k(\varphi^\beta)\leq\max\{k(\varphi),|\mathbb L_\beta^u|\}=\max\{k(\varphi),\beta\}.
\end{equation*}
A $\Sigma(\mathbf L_{\omega^\alpha}^u)$-formula is an $\mathbf L_{\omega^\alpha}^u$-formula which contains no unbounded universal quantifier. A $\Pi_1(\mathbf L_{\omega^\alpha}^u)$-formula is an $\mathbf L_{\omega^\alpha}^u$-formula of the form $\forall_x\theta$ where $\theta$ is bounded. The following boundedness result is a final preparation for collapsing:

\begin{lemma}\label{lem:boundedness}
 Assume $\vartheta:\varepsilon(S_{\omega^\alpha}^u)\bh\alpha$. We can extend the system of $\mathbf L_{\omega^\alpha}^u$-codes in the following way:
\begin{enumerate}[label=(\alph*)]
 \item For each $\Sigma(\mathbf L_{\omega^\alpha}^u)$-formula $\varphi$ and each ordinal $\beta<\omega^\alpha$ we add a unary function symbol $\mathcal B_{\exists,\varphi}^\beta$ such that we have
\begin{align*}
 l_{\langle\rangle}(\mathcal B_{\exists,\varphi}^\beta P)&=(l_{\langle\rangle}(P)\backslash\{\varphi\})\cup\{\varphi^\beta\},\\
 o_{\langle\rangle}(\mathcal B_{\exists,\varphi}^\beta P)&=o_{\langle\rangle}(P),\\
 d(\mathcal B_{\exists,\varphi}^\beta P)&=d(P),\\
 h_0(\mathcal B_{\exists,\varphi}^\beta P)&=h_0(P),\\
 h_1(\mathcal B_{\exists,\varphi}^\beta P)&=\max\{h_1(P),k(\varphi^\beta)\}
\end{align*}
for any $\mathbf L_{\omega^\alpha}^u$-code $P$ with $o_{\langle\rangle}(P)\leq\beta$.
 \item For each unbounded $\Pi_1(\mathbf L_{\omega^\alpha}^u)$-formula $\psi$ and each ordinal $\beta<\omega^\alpha$ we add a unary function symbol $\mathcal B_{\forall,\psi}^\beta$ such that we have
\begin{align*}
 l_{\langle\rangle}(\mathcal B_{\forall,\psi}^\beta P)&=(l_{\langle\rangle}(P)\backslash\{\psi\})\cup\{\psi^\beta\},\\
 o_{\langle\rangle}(\mathcal B_{\forall,\psi}^\beta P)&=o_{\langle\rangle}(P),\\
 d(\mathcal B_{\forall,\psi}^\beta P)&=d(P),\\
 h_0(\mathcal B_{\forall,\psi}^\beta P)&=h_0(P),\\
 h_1(\mathcal B_{\forall,\psi}^\beta P)&=\max\{h_1(P),k(\psi^\beta)\}
\end{align*}
for any $\mathbf L_{\omega^\alpha}^u$-code $P$ (without any restriction on $o_{\langle\rangle}(P)$).
\end{enumerate}
\end{lemma}
\begin{proof}
 (a) Let us first say what happens in the unintended case~$\beta<o_{\langle\rangle}(P)$: There we stipulate that $\mathcal B_{\exists_y\theta}^\beta P$ behaves as $P$, i.e.~we set
\begin{align*}
 l_{\langle\rangle}(\mathcal B_{\exists_y\theta}^\beta P)&=l_{\langle\rangle}(P), & o_{\langle\rangle}(\mathcal B_{\exists_y\theta}^\beta P)&=o_{\langle\rangle}(P), & r_{\langle\rangle}(\mathcal B_{\exists_y\theta}^\beta P)&=r_{\langle\rangle}(P),\\
 n(\mathcal B_{\exists_y\theta}^\beta P,a)&=n(P,a), & d(\mathcal B_{\exists_y\theta}^\beta P)&=d(P), & h_i(\mathcal B_{\exists_y\theta}^\beta P)&=h_i(P).
\end{align*}
Then conditions (L), (C1,C2) and (H1-H3) for $\mathcal B_{\exists_y\theta}^\beta P$ immediately follow from the same conditions for $P$. In the intended case $o_{\langle\rangle}(P)\leq\beta$ it remains to extend the functions $r_{\langle\rangle}$ and $n$ to terms of the form $\mathcal B_{\exists_y\theta}^\beta P$. We do this by case distinction on the rule $r_{\langle\rangle}(P)$, checking local correctness as we go along:

\emph{Case $r_{\langle\rangle}(P)=(\exists_x,b,\theta)$ with $\exists_x\theta\equiv\varphi$ and the outer quantifier unbounded:} Intuitively, $P$ deduces $\exists_x\theta$ from $\theta(b)$. Using the induction hypothesis we will be able to obtain $\theta(b)^\beta$. Also, condition (L) for $P$ gives $|b|<o_{\langle\rangle}(P)\leq\beta$. Thus the bounded formula $b\in\mathbb L_\beta^u$ is true, and hence an axiom. We can introduce the conjunction $b\in\mathbb L_\beta^u\land\theta(b)^\beta$ and then the existential statement $\exists_{x\in\mathbb L_\beta^u}\theta^\beta\equiv\varphi^\beta$. To cast this in terms of codes we define a new constant $\mathbf L_{\omega^\alpha}^u$-code $\ax_\theta$ for each true bounded $\mathbf L_{\omega^\alpha}^u$-formula $\theta$, setting
\begin{align*}
 l_{\langle\rangle}(\ax_\theta)&=\langle\theta\rangle, & o_{\langle\rangle}(\ax_\theta)&=0, & r_{\langle\rangle}(\ax_\theta)&=\ax,\\
 n(\ax_\theta,a)&=\ax_\theta, & d(\ax_\theta)&=0, & \langle h_0(\ax_\theta),h_1(\ax_\theta)\rangle&=\langle\Omega,k(\theta)\rangle.
\end{align*}
Conditions (L), (C1,C2) and (H1-H3) are easily checked. Also, we introduce a binary function symbol $\bigwedge_{\psi_0,\psi_1}$ on $\mathbf L_{\omega^\alpha}^u$-codes, for all $\mathbf L_{\omega^\alpha}^u$-formulas~$\psi_0,\psi_1$. This will serve to introduce conjunctions:
\begingroup\allowdisplaybreaks
\begin{align*}
 l_{\langle\rangle}(\textstyle\bigwedge_{\psi_0,\psi_1}P_0P_1)&=l_{\langle\rangle}(P_0)\backslash\{\psi_0\}\cup l_{\langle\rangle}(P_1)\backslash\{\psi_1\}\cup\{\psi_0\land\psi_1\},\\
 o_{\langle\rangle}(\textstyle\bigwedge_{\psi_0,\psi_1}P_0P_1)&=\max\{o_{\langle\rangle}(P_0),o_{\langle\rangle}(P_1)\}+1,\\
 r_{\langle\rangle}(\textstyle\bigwedge_{\psi_0,\psi_1}P_0P_1)&=(\land,\psi_0,\psi_1),\\
 n(\textstyle\bigwedge_{\psi_0,\psi_1}P_0P_1,a)&=\begin{cases}
                                          P_0\quad&\text{if $a=0$},\\
                                          P_1\quad&\text{otherwise},
                                       \end{cases}\\
 d(\textstyle\bigwedge_{\psi_0,\psi_1}P_0P_1)&=\max\{d(P_0),d(P_1)\},\\
 h_0(\textstyle\bigwedge_{\psi_0,\psi_1}P_0P_1)&=\max\{h_0(P_0),h_0(P_1)\},\\
 h_1(\textstyle\bigwedge_{\psi_0,\psi_1}P_0P_1)&=\max\{h_1(P_0),h_1(P_1),k(\psi_0\land\psi_1))\}.
\end{align*}\endgroup%
Again, it is straightforward to check conditions (L), (C1,C2) and (H1-H3). After these preparations we can address the code $\mathcal B_{\exists,\varphi}^\beta P$ itself: Set
\begin{align*}
 r_{\langle\rangle}(\mathcal B_{\exists,\varphi}^\beta P)&=(\exists_x,b,x\in\mathbb L_\beta^u\land\theta^\beta),\\
 n(\mathcal B_{\exists_x\theta}^\beta P,a)&=\textstyle\bigwedge_{b\in\mathbb L_\beta^u,\theta(b)^\beta}\ax_{b\in\mathbb L_\beta^u} (\mathcal B_{\exists,\theta(b)}^\beta\mathcal B_{\exists,\varphi}^\beta n(P,0)).
\end{align*}
We have already observed $b\in\mathbb L_\beta^u$, which ensures that $\ax_{b\in\mathbb L_\beta^u}$ is an $\mathbf L_{\omega^\alpha}^u$-code. Also note that $\theta(b)$ is a $\Sigma(\mathbf L_{\omega^\alpha}^u)$-formula, so that $\mathcal B_{\exists,\theta(b)}^\beta$ is one of our new function symbols. As a preparation for local correctness, observe that condition (L) for $P$ implies $o_{\langle\rangle}(\mathcal B_{\exists,\varphi}^\beta n(P,0))=o_{\langle\rangle}(n(P,0))\leq o_{\langle\rangle}(P)\leq\beta$. Thus the codes $\mathcal B_{\exists,\theta(b)}^\beta\mathcal B_{\exists,\varphi}^\beta n(P,0)$ and $\mathcal B_{\exists,\varphi}^\beta n(P,0)$ fall under the ``intended case''. Now condition (L) for $\mathcal B_{\exists,\varphi}^\beta P$ holds by
\begin{multline*}
 o_{\langle\rangle}(n(\mathcal B_{\exists,\varphi}^\beta P,0))+\omega=o_{\langle\rangle}(\mathcal B_{\exists,\theta(b)}^\beta\mathcal B_{\exists,\varphi}^\beta n(P,0))+1+\omega=\\
=o_{\langle\rangle}(n(P,0))+1+\omega=o_{\langle\rangle}(n(P,0))+\omega\leq o_{\langle\rangle}(P)=o(\mathcal B_{\exists,\varphi}^\beta P)
\end{multline*}
and
\begin{equation*}
 |b|<o_{\langle\rangle}(P)=o_{\langle\rangle}(\mathcal B_{\exists,\varphi}^\beta P),
\end{equation*}
as well as $\exists_{x\in\mathbb L_\beta^u}\theta^\beta\equiv\varphi^\beta\in l_{\langle\rangle}(\mathcal B_{\exists,\varphi}^\beta P)$ (by definition) and
\begingroup\allowdisplaybreaks
\begin{multline*}
 l_{\langle\rangle}(n(\mathcal B_{\exists,\varphi}^\beta P,0))=l_{\langle\rangle}(\textstyle\bigwedge_{b\in\mathbb L_\beta^u,\theta(b)^\beta}\ax_{b\in\mathbb L_\beta^u} (\mathcal B_{\exists,\theta(b)}^\beta\mathcal B_{\exists,\varphi}^\beta n(P,0)))=\\
= l_{\langle\rangle}(\ax_{b\in\mathbb L_\beta^u})\backslash\{b\in\mathbb L_\beta^u\}\cup l_{\langle\rangle}(\mathcal B_{\exists,\theta(b)}^\beta\mathcal B_{\exists,\varphi}^\beta n(P,0))\backslash\{\theta(b)^\beta\}\cup\{b\in\mathbb L_\beta^u\land\theta(b)^\beta\}\subseteq\\
\subseteq \left[l_{\langle\rangle}(n(P,0))\backslash\{\varphi,\theta(b)\}\cup\{\varphi^\beta,\theta(b)^\beta\}\right]\backslash\{\theta(b)^\beta\}\cup\{b\in\mathbb L_\beta^u\land\theta(b)^\beta\}\subseteq\\
\subseteq l_{\langle\rangle}(P)\backslash\{\varphi\}\cup\{\varphi^\beta,b\in\mathbb L_\beta^u\land\theta(b)^\beta\}=l_{\langle\rangle}(\mathcal B_{\exists,\varphi}^\beta P)\cup\{b\in\mathbb L_\beta^u\land\theta(b)^\beta\}.
\end{multline*}\endgroup%
Condition (C1) is void and condition (C2) follows from same condition for $P$. Concerning condition (H1), we have
\begin{equation*}
 k_{\langle\rangle}(\mathcal B_{\exists,\varphi}^\beta P)\leq\max\{k_{\langle\rangle}(P),k(\varphi^\beta)\}\leq\max\{k_{\langle\rangle}(P),h_1(\mathcal B_{\exists,\varphi}^\beta P)\}.
\end{equation*}
By (H1) for $P$ we have $k_{\langle\rangle}(P)\in\mathcal H_P^\varphi\subseteq\mathcal H_{\mathcal B_{\exists,\varphi}^\beta P}^\vartheta$. Also, $h_1(\mathcal B_{\exists,\varphi}^\beta P)\in\mathcal H_{\mathcal B_{\exists,\varphi}^\beta P}^\vartheta$ follows from the definition of our operators. Together we obtain
\begin{equation*}
 k_{\langle\rangle}(\mathcal B_{\exists,\varphi}^\beta P)\in\mathcal H_{\mathcal B_{\exists,\varphi}^\beta P}^\vartheta,
\end{equation*}
as required by condition (H1) for $\mathcal B_{\exists,\varphi}^\beta P$. Condition (H2) is easily reduced to the same condition for $P$. Finally, let us verify condition (H3): From condition (L) for $P$ we know that $\varphi$ occurs in $l_{\langle\rangle}(P)$, so that we have $k(\theta)=k(\varphi)\leq k_{\langle\rangle}(P)$. Together with $|b|\leq k(r_{\langle\rangle}(P))\leq k_{\langle\rangle}(P)$ we get
\begin{multline*}
 h_1(n(\mathcal B_{\exists,\varphi}^\beta P,0))=h_1(\textstyle\bigwedge_{b\in\mathbb L_\beta^u,\theta(b)^\beta}\ax_{b\in\mathbb L_\beta^u} (\mathcal B_{\exists,\theta(b)}^\beta\mathcal B_{\exists,\varphi}^\beta n(P,0)))\leq\\
\leq\max\{h_1(\mathcal B_{\exists,\theta(b)}^\beta\mathcal B_{\exists,\varphi}^\beta n(P,0)),k(b\in\mathbb L_\beta^u\land\theta(b)^\beta)\}\leq\\
\leq\max\{h_1(n(P,0)),k(\varphi),|b|\}\leq\max\{h_1(n(P,0)),h_1(\mathcal B_{\exists,\varphi}^\beta P),k_{\langle\rangle}(P)\}.
\end{multline*}
Now $h_1(n(P,0))\in\mathcal H_P^\vartheta[|0|]\subseteq\mathcal H_{\mathcal B_{\exists,\varphi}^\beta P}^\vartheta[|0|]$ holds by (H3) for $P$; by definition of our operators we have $h_1(\mathcal B_{\exists,\varphi}^\beta P)\in\mathcal H_{\mathcal B_{\exists,\varphi}^\beta P}^\vartheta[|0|]$; and 
$k_{\langle\rangle}(P)\in\mathcal H_P^\vartheta\subseteq\mathcal H_{\mathcal B_{\exists,\varphi}^\beta P}^\vartheta[|0|]$ is due to (H1) for $P$. Together this gives
\begin{equation*}
 h_1(n(\mathcal B_{\exists,\varphi}^\beta P,0))\in\mathcal H_{\mathcal B_{\exists,\varphi}^\beta P}^\vartheta[|0|],
\end{equation*}
as required as required by condition (H3) for $\mathcal B_{\exists,\varphi}^\beta P$. Still concerning condition~(H3), we have
\begin{equation*}
 o_{\langle\rangle}(n(\mathcal B_{\exists,\varphi}^\beta P,0))=o_{\langle\rangle}(\mathcal B_{\exists,\theta(b)}^\beta\mathcal B_{\exists,\varphi}^\beta n(P,0))+1=o_{\langle\rangle}(n(P,0))+1.
\end{equation*}
Condition (H3) for $P$ provides $o_{\langle\rangle}(n(P,0))\in\mathcal H_P^\vartheta[|0|]\subseteq\mathcal H_{\mathcal B_{\exists,\varphi}^\beta P}^\vartheta[|0|]$, which implies
\begin{equation*}
 o_{\langle\rangle}(n(\mathcal B_{\exists,\varphi}^\beta P,0))\in\mathcal H_{\mathcal B_{\exists,\varphi}^\beta P}^\vartheta[|0|],
\end{equation*}
as required by condition (H3) for $\mathcal B_{\exists,\varphi}^\beta P$.

\emph{Case $r_{\langle\rangle}(P)=(\exists_x,b,\theta)$ with $\exists_x\theta\equiv\varphi$ and the outer quantifier bounded:} Note that we now have $\exists_x\theta^\beta\equiv\varphi^\beta$. We set
\begin{align*}
 r_{\langle\rangle}(\mathcal B_{\exists,\varphi}^\beta P)&=(\exists_x,b,\theta^\beta),\\
 n(\mathcal B_{\exists,\varphi}^\beta P,a)&=\mathcal B_{\exists,\theta(b)}^\beta\mathcal B_{\exists,\varphi}^\beta n(P,0).
\end{align*}
The verification of local correctness is easier than in the previous case.

\emph{Case $r_{\langle\rangle}(P)=(\exists_x,b,\theta)$ with $\exists_x\theta\not\equiv\varphi$:} Here we put
\begin{align*}
 r_{\langle\rangle}(\mathcal B_{\exists,\varphi}^\beta P)&=r_{\langle\rangle}(P),\\
 n(\mathcal B_{\exists,\varphi}^\beta P,a)&=\mathcal B_{\exists,\varphi}^\beta n(P,a).
\end{align*}
The verification of local correctness is easier than in the previous cases.

The rule $r_{\langle\rangle}(P)=(\land,\psi_0,\psi_1)$ is treated similarly, distinguishing the two cases $\varphi\equiv\psi_0\land\psi_1$ and $\varphi\not\equiv\psi_0\land\psi_1$. The same applies to the rules $r_{\langle\rangle}(P)=(\lor_i,\psi_0,\psi_1)$ and $r_{\langle\rangle}(P)=(\forall_x,\psi)$ (if $\forall_x\psi\equiv\varphi$ then the outer quantifier must be bounded, which implies $\forall_x(\psi^\beta)\equiv\varphi^\beta$). Let us look at the remaining cases:

\emph{Case $r_{\langle\rangle}(P)=\ax$:} We set
\begin{align*}
 r_{\langle\rangle}(\mathcal B_{\exists,\varphi}^\beta P)&=\ax,\\
 n(\mathcal B_{\exists,\varphi}^\beta P,a)&=\mathcal B_{\exists,\varphi}^\beta n(P,a).
\end{align*}
By condition (L) for $P$ the sequent $l_{\langle\rangle}(P)$ contains a true bounded formula. The same formula is still contained in $l_{\langle\rangle}(\mathcal B_{\exists,\varphi}^\beta P)=(l_{\langle\rangle}(P)\backslash\{\varphi\})\cup\{\varphi^\beta\}$, as we have $\varphi^\beta\equiv\varphi$ if $\varphi$ is bounded. The remaining conditions are verified as above.

\emph{Case $r_{\langle\rangle}(P)=(\cut,\psi)$:} Set
\begin{align*}
 r_{\langle\rangle}(\mathcal B_{\exists,\varphi}^\beta P)&=(\cut,\psi),\\
 n(\mathcal B_{\exists,\varphi}^\beta P,a)&=\mathcal B_{\exists,\varphi}^\beta n(P,a).
\end{align*}
The verification of local correctness is straightforward.

\emph{Case $r_{\langle\rangle}(P)=(\rref,\exists_z\forall_{x\in a}\exists_{y\in z}\theta)$:} By condition (L) for $P$ this would require $\Omega\leq o_{\langle\rangle}(P)\leq\beta$, which is incompatible with the assumption $\beta<\omega^\alpha$.

\emph{Case $r_{\langle\rangle}(P)=(\rep,b)$:} Set
\begin{align*}
 r_{\langle\rangle}(\mathcal B_{\exists,\varphi}^\beta P)&=(\rep,b),\\
 n(\mathcal B_{\exists,\varphi}^\beta P,a)&=\mathcal B_{\exists,\varphi}^\beta n(P,a).
\end{align*}
The verification of local correctness is straightforward.

(b) The proof is similar to that of part (a), but the assumption that $\psi$ is an unbounded $\Pi_1(\mathbf L_{\omega^\alpha}^u)$-formula saves us some cases: For example it implies that $\psi$ cannot be of the form $\psi_0\land\psi_1$. We write out the only interesting case, leaving the other cases to the reader:

\emph{Case $r_{\langle\rangle}(P)=(\forall_x,\theta)$ with $\forall_x\theta\equiv\psi$ and the outer quantifier unbounded:} Intuitively, $P$ deduces $\forall_x\theta$ from the assumptions $\theta(a)$ for all $a\in\mathbf L_{\omega^\alpha}^u$. Introducing disjunctions we get $a\notin\mathbb L_\beta^u\lor\theta(a)$, from which we obtain $\forall_{x\in\mathbb L_\beta^u}\theta\equiv\psi$. To formulate this in terms of codes we introduce a unary function symbol $\bigvee_{\psi_0,\psi_1}^i$ for $i=0,1$ and arbitrary $\mathbf L_{\omega^\alpha}^u$-formulas $\psi_0,\psi_1$. This serves to introduce disjunctions:
\begingroup\allowdisplaybreaks
\begin{align*}
 l_{\langle\rangle}(\textstyle\bigvee_{\psi_0,\psi_1}^iP)&=l_{\langle\rangle}(P)\backslash\{\psi_i\}\cup\{\psi_0\lor\psi_1\},\\
 o_{\langle\rangle}(\textstyle\bigvee^i_{\psi_0,\psi_1}P)&=o_{\langle\rangle}(P)+1,\\
 r_{\langle\rangle}(\textstyle\bigvee^i_{\psi_0,\psi_1}P)&=(\lor_i,\psi_0,\psi_1),\\
 n(\textstyle\bigvee^i_{\psi_0,\psi_1}P,a)&=P,\\
 d(\textstyle\bigvee^i_{\psi_0,\psi_1}P)&=d(P),\\
 h_0(\textstyle\bigvee^i_{\psi_0,\psi_1}P)&=h_0(P),\\
 h_1(\textstyle\bigvee^i_{\psi_0,\psi_1}P)&=\max\{h_1(P),k(\psi_0\lor\psi_1)\}.
\end{align*}\endgroup%
It is straightforward to check local correctness. Now put
\begin{align*}
 r_{\langle\rangle}(\mathcal B_{\forall,\psi}^\beta P)&=(\forall_x,x\notin\mathbb L_\beta^u\lor\theta),\\
 n(\mathcal B_{\forall,\psi}^\beta P,a)&=\textstyle\bigvee_{a\notin\mathbb L_\beta^u,\theta(a)}^1\mathcal B_{\forall,\psi}^\beta n(P,a).
\end{align*}
Let us verify condition (L) for $\mathcal B_{\forall,\psi}^\beta P$: Concerning the ``ordinal'' labels we have
\begin{multline*}
 o_{\langle\rangle}(n(\mathcal B_{\forall,\psi}^\beta P,a))+\omega=o_{\langle\rangle}(\mathcal B_{\forall,\psi}^\beta n(P,a))+1+\omega=o_{\langle\rangle}(n(P,a))+1+\omega=\\
=o_{\langle\rangle}(n(P,a))+\omega\leq o_{\langle\rangle}(P)=o_{\langle\rangle}(\mathcal B_{\forall,\psi}^\beta P).
\end{multline*}
The formula $\forall_{x\in\mathbb L_\beta^u}\theta\equiv\varphi^\beta$ occurs in $l_{\langle\rangle}(\mathcal B_{\forall,\psi}^\beta P)$ by definition. We also have
\begingroup\allowdisplaybreaks
\begin{multline*}
 l_{\langle\rangle}(n(\mathcal B_{\forall,\psi}^\beta P,a))=l_{\langle\rangle}(\mathcal B_{\forall,\psi}^\beta n(P,a))\backslash\{\theta(a)\}\cup\{a\notin\mathbb L_\beta^u\lor\theta(a)\}\subseteq\\
\subseteq l_{\langle\rangle}(n(P,a))\backslash\{\varphi,\theta(a)\}\cup\{\varphi^\beta,a\notin\mathbb L_\beta^u\lor\theta(a)\}\subseteq\\
\subseteq (l_{\langle\rangle}(P)\cup\{\theta(a)\})\backslash\{\varphi,\theta(a)\}\cup\{\varphi^\beta,a\notin\mathbb L_\beta^u\lor\theta(a)\}\subseteq\\
\subseteq l_{\langle\rangle}(P)\backslash\{\varphi\}\cup\{\varphi^\beta,a\notin\mathbb L_\beta^u\lor\theta(a)\}=l_{\langle\rangle}(n(\mathcal B_{\forall,\psi}^\beta P)\cup\{a\notin\mathbb L_\beta^u\lor\theta(a)\},
\end{multline*}\endgroup
as required by condition (L) at the rule $(\forall_x,x\notin\mathbb L_\beta^u\lor\theta)$. Condition (C1) does not apply, and (C2) is easily reduced to the same condition for $P$. As for (H1), similar to part (a) one sees
\begin{equation*}
 k_{\langle\rangle}(\mathcal B_{\forall,\psi}^\beta P)\leq\max\{k_{\langle\rangle}(P),k(\psi_0\lor\psi_1)\}\leq\max\{k_{\langle\rangle}(P),h_1(\mathcal B_{\forall,\psi}^\beta P)\}.
\end{equation*}
By (H1) for $P$ we have $k_{\langle\rangle}(P)\in\mathcal H_P^\vartheta\subseteq\mathcal H_{\mathcal B_{\forall,\psi}^\beta P}^\vartheta$. Also, $h_1(\mathcal B_{\forall,\psi}^\beta P)\in\mathcal H_{\mathcal B_{\forall,\psi}^\beta P}^\vartheta$ holds by the definition of our operators. Together we obtain
\begin{equation*}
 k_{\langle\rangle}(\mathcal B_{\forall,\psi}^\beta P)\in\mathcal H_{\mathcal B_{\forall,\psi}^\beta P}^\vartheta,
\end{equation*}
as required by condition (H1) for $\mathcal B_{\forall,\psi}^\beta P$. Condition (H2) is easily reduced to the same condition for $P$. As for (H3), it is straightforward to see
\begin{equation*}
 h_1(n(\mathcal B_{\forall,\psi}^\beta P,a))\leq\max\{h_1(n(P,a)),k(\psi^\beta),|a|\}\leq\max\{h_1(n(P,a)),h_1(\mathcal B_{\forall,\psi}^\beta P),|a|\}.
\end{equation*}
Condition (H3) for $P$ gives $h_1(n(P,a))\in\mathcal H_P^\vartheta[|a|]\subseteq\mathcal H_{\mathcal B_{\forall,\psi}^\beta P}^\vartheta[|a|]$. Also, the definition of our operators yields $\{h_1(\mathcal B_{\forall,\psi}^\beta P),|a|\}\subseteq\mathcal H_{\mathcal B_{\forall,\psi}^\beta P}^\vartheta[|a|]$. Together we get
\begin{equation*}
 h_1(n(\mathcal B_{\forall,\psi}^\beta P,a))\in\mathcal H_{\mathcal B_{\forall,\psi}^\beta P}^\vartheta[|a|],
\end{equation*}
as required by condition (H3) for $\mathcal B_{\forall,\psi}^\beta P$. Still concerning (H3), we have
\begin{equation*}
 o_{\langle\rangle}(n(\mathcal B_{\forall,\psi}^\beta P,a))=o_{\langle\rangle}(n(P,a))+1.
\end{equation*}
Condition (H3) for $P$ gives $o_{\langle\rangle}(n(P,a))\in\mathcal H_P^\vartheta[|a|]\subseteq\mathcal H_{\mathcal B_{\forall,\psi}^\beta P}^\vartheta[|a|]$. This implies
\begin{equation*}
 o_{\langle\rangle}(n(\mathcal B_{\forall,\psi}^\beta P,a))\in\mathcal H_{\mathcal B_{\forall,\psi}^\beta P}^\vartheta[|a|],
\end{equation*}
as required by condition (H3) for $\mathcal B_{\forall,\psi}^\beta P$. 
\end{proof}

The last two sections have been a preparation of the following collapsing result:

\begin{theorem}\label{thm:collapsing}
 Assume $\vartheta:\varepsilon(S_{\omega^\alpha}^u)\bh\alpha$. We can extend the system of $\mathbf L_{\omega^\alpha}^u$-codes by a unary function symbol~$\mathcal C_t$ for each $t\in\varepsilon(S_{\omega^\alpha}^u)$ with $\Omega\leq t$, in such a way that we have
\begin{align*}
 l_{\langle\rangle}(\mathcal C_tP)&=l_{\langle\rangle}(P), & o_{\langle\rangle}(\mathcal C_tP)&=\vartheta(t+\Omega^{o_{\langle\rangle}(P)}),\\
 d(\mathcal C_tP)&=1, & \langle h_0(\mathcal C_tP),h_1(\mathcal C_tP)\rangle&=\langle t+\Omega^{o_{\langle\rangle}(P)},0\rangle
\end{align*}
whenever the following conditions are satisfied:
\begin{enumerate}[label=(\roman*)]
 \item $l_{\langle\rangle}(P)$ contains only $\Sigma(\mathbf L_{\omega^\alpha}^u)$-formulas and $d(P)\leq 2$,
 \item $h_0(P)\leq t$ and $\{t,h_1(P)\}\subseteq\mathcal H_t^\vartheta$.
\end{enumerate}
\end{theorem}
\begin{proof}
 First, in the unintended case where one of the conditions (i,ii) fails we stipulate that $\mathcal C_tP$ behaves like $P$ (see~the previous proof). In this case, the local correctness of $\mathcal C_tP$ follows from the local correctness of $P$. Now assume that conditions (i) and (ii) hold for $P$. We define $r_{\langle\rangle}(\mathcal C_tP)$ and $n(\mathcal C_tP,a)$ by case distinction on $r_{\langle\rangle}(P)$, verifying local correctness as we go along:

\emph{Case $r_{\langle\rangle}(P)=\ax$:} We set $r_{\langle\rangle}(\mathcal C_tP)=\ax$ and $n(\mathcal C_tP,a)=\mathcal C_tn(P,a)$. Condition (L) for $P$ implies that $l_{\langle\rangle}(P)$ contains a true bounded formula. Thus the same holds for $l_{\langle\rangle}(\mathcal C_tP)=l_{\langle\rangle}(P)$, as required by condition (L) for $\mathcal C_t(P)$. Condition (C1) is trivial as $r_{\langle\rangle}(\mathcal C_tP)$ is not a cut rule. Condition (C2) does not apply as $\iota(\ax)=\emptyset$. Concerning condition (H1), it is easy to see
\begin{equation*}
 k_{\langle\rangle}(\mathcal C_tP)\leq\max\{k_{\langle\rangle}(P),\vartheta(t+\Omega^{o_{\langle\rangle}(P)})^*\}.
\end{equation*}
By condition (H1) for $P$ and assumption (ii) we get
\begin{equation*}
 k_{\langle\rangle}(P)\in\mathcal H_P^\vartheta\subseteq\mathcal H_t^\vartheta[h_1(P)]=\mathcal H_t^\vartheta\subseteq\mathcal H_{t+\Omega^{o_{\langle\rangle}(P)}}^\vartheta=\mathcal H_{\mathcal C_tP}^\vartheta.
\end{equation*}
In particular we have $o_{\langle\rangle}(P)\in\mathcal H_{t+\Omega^{o_{\langle\rangle}(P)}}^\vartheta$. Together with assumption (ii) we obtain
\begin{equation*}
 \vartheta(t+\Omega^{o_{\langle\rangle}(P)})\in\mathcal H_{t+\Omega^{o_{\langle\rangle}(P)}}^\vartheta=\mathcal H_{\mathcal C_tP}^\vartheta
\end{equation*}
and then $\vartheta(t+\Omega^{o_{\langle\rangle}(P)})^*\in\mathcal H_{\mathcal C_tP}^\vartheta$, completing the proof of condition (H1) for $\mathcal C_tP$. Conditions (H2) and (H3) do not apply, because of $\iota(\ax)=\emptyset$.

\emph{Case $r_{\langle\rangle}(P)=(\land,\psi_0,\psi_1)$:} Set $r_{\langle\rangle}(\mathcal C_tP)=(\land,\psi_0,\psi_1)$ and $n(\mathcal C_tP,a)=\mathcal C_tn(P,a)$. Before we can verify local correctness we must check that $n(P,i)$ satisfies assumptions (i,ii) for~$i=0,1$: By condition (L) for $P$ the formula $\psi_0\land\psi_1$ occurs in $l_{\langle\rangle}(P)$. Thus assumption (i) for $P$ implies that $\psi_0\land\psi_1$ is a $\Sigma(\mathbf L_{\omega^\alpha})$-formula, and so are $\psi_0$ and $\psi_1$. Also by condition (L) for $P$ we have
\begin{equation*}
 l_{\langle\rangle}(n(P,i))\subseteq l_{\langle\rangle}(P)\cup\{\psi_i\}.
\end{equation*}
Thus $l_{\langle\rangle}(n(P,i))$ consists of $\Sigma(\mathbf L_{\omega^\alpha})$-formulas. Next, by condition (C2) and assumption (i) for $P$ we have $d(n(P,i))\leq d(P)\leq 2$. So $n(P,i)$ satisfies assumption (i). Concerning assumption (ii), using (H2) for $P$ we obtain $h_0(n(P,i))\leq h_0(P)\leq t$. Also, condition (H3) and assumption (ii) for $P$ yield
\begin{equation*}
 h_1(n(P,i))\in\mathcal H_P^\vartheta[|i|]=\mathcal H_P^\vartheta\subseteq\mathcal H_t^\vartheta.
\end{equation*}
We have established that $n(P,i)$ satisfies assumptions (i,ii). So $n(\mathcal C_tP,a)=\mathcal C_tn(P,a)$ falls under the ``intended case''. Now we can verify condition (L) for $\mathcal C_tP$: By condition (H3) and assumption (ii) for $P$ we get
\begin{equation*}
 o_{\langle\rangle}(n(P,i))\in\mathcal H_P^\vartheta[|i|]=\mathcal H_P^\vartheta\subseteq\mathcal H_t^\vartheta.
\end{equation*}
Assumption (ii) provides $t\in\mathcal H_t^\vartheta$ and condition (L) for $P$ gives $o_{\langle\rangle}(n(P,i))<o_{\langle\rangle}(P)$. In this situation Proposition~\ref{prop:operators-special} yields
\begin{equation*}
 o_{\langle\rangle}(n(\mathcal C_tP,i))=\vartheta(t+\Omega^{o_{\langle\rangle}(n(P,i))})<\vartheta(t+\Omega^{o_{\langle\rangle}(P)})=o_{\langle\rangle}(\mathcal C_tP),
\end{equation*}
as required by condition (L) for $\mathcal C_tP$. Still concerning condition (L) for $\mathcal C_tP$, by the local correctness of $P$ the formula $\psi_0\land\psi_1$ occurs in $l_{\langle\rangle}(P)=l_{\langle\rangle}(\mathcal C_tP)$. Also we have
\begin{equation*}
 l_{\langle\rangle}(n(\mathcal C_tP,i))=l_{\langle\rangle}(n(P,i))\subseteq l_{\langle\rangle}(P)\cup\{\psi_i\}=l_{\langle\rangle}(\mathcal C_tP)\cup\{\psi_i\},
\end{equation*}
completing the verification of (L) for $\mathcal C_tP$. Condition (H1) ist established as in the case of an axiom above. In view of $o_{\langle\rangle}(n(P,i))<o_{\langle\rangle}(P)$ we get
\begin{equation*}
 h_0(n(\mathcal C_tP,i))=t+\Omega^{o_{\langle\rangle}(n(P,i))}<t+\Omega^{o_{\langle\rangle}(P)}=h_0(\mathcal C_tP),
\end{equation*}
which settles condition (H2) for $\mathcal C_tP$. The first part of (H3) is automatic by
\begin{equation*}
 h_1(n(\mathcal C_tP,i))=h_1(\mathcal C_tn(P,i))=0\in\mathcal H_{\mathcal C_tP}^\vartheta[|i|].
\end{equation*}
We have already seen $o_{\langle\rangle}(n(P,i))\in\mathcal H_t^\vartheta$ above. Together with $o_{\langle\rangle}(n(P,i))<o_{\langle\rangle}(P)$ and $t\in\mathcal H_t^\vartheta$ this implies
\begin{equation*}
 o_{\langle\rangle}(n(\mathcal C_tP,i))=\vartheta(t+\Omega^{o_{\langle\rangle}(n(P,i))})\in\mathcal H_{t+\Omega^{o_{\langle\rangle}(P)}}^\vartheta\subseteq\mathcal H_{\mathcal C_tP}^\vartheta[|i|].
\end{equation*}
This completes the verification of (H3) for $\mathcal C_tP$.

\emph{Case $r_{\langle\rangle}(P)=(\lor_i,\psi_0,\psi_1)$:} Set $r_{\langle\rangle}(\mathcal C_tP)=(\lor_i,\psi_0,\psi_1)$ and $n(\mathcal C_tP,a)=\mathcal C_tn(P,a)$. Local correctness is verified as in the previous case.

\emph{Case $r_{\langle\rangle}=(\forall_x,\psi)$:} By assumption (i) the outer quantifier of $\forall_x\psi$ must be bounded, i.e.~we must have $\psi\equiv x\notin b\lor\theta$ for some $\Sigma(\mathbf L_{\omega^\alpha}^u)$-formula $\theta$. We set
\begin{align*}
 r_{\langle\rangle}(\mathcal C_tP)&=(\forall_x,\psi),\\
 n(\mathcal C_tP,a)&=\begin{cases}
                     \mathcal C_t n(P,a)\quad &\text{if $|a|\leq |b|$},\\
                     \textstyle\bigvee^0_{a\notin b,\theta(a)}\ax_{a\notin b}\quad &\text{if $|a|>|b|$}.
                    \end{cases}
\end{align*}
The code $\bigvee^0_{a\notin b,\theta(a)}\ax_{a\notin b}$ has been introduced in the proof of Lemma~\ref{lem:boundedness}. To see that $\ax_{a\notin b}$ is indeed an $\mathbf L_{\omega^\alpha}^u$-code we must check that $|a|>|b|$ implies $a\notin b$: Assume $a\in b$. By definition of the rank we have $b\in\mathbb L_{|b|+1}^u$. As the stages of the constructible hierarchy are transitive we obtain $a\in\mathbb L_{|b|+1}^u$. Now $|a|\leq |b|$ follows by the minimality of the rank. Next, let us verify that assumptions (i) and (ii) hold for $n(P,a)$ whenever we have $|a|\leq |b|$: Indeed, assumption (i) and the first part of assumption (ii) hold for any $a\in\mathbf L_{\omega^\alpha}^u$, as in the previous cases. As for the last part of (ii), observe that $b$ occurs in $\forall_x\psi\in l_{\langle\rangle}(P)$, which implies $|b|\leq k_{\langle\rangle}(P)$. Condition (H1) for $P$ provides $k_{\langle\rangle}(P)\in\mathcal H_P^\vartheta$. Thus we see that $\mathcal H_P^\vartheta[|a|]=\mathcal H_P^\vartheta$ holds whenever we have $|a|\leq|b|$. Using condition (H3) and assumption (ii) for $P$ we can conclude
\begin{equation*}
 h_1(n(P,a))\in\mathcal H_P^\vartheta[|a|]=\mathcal H_P^\vartheta\subseteq\mathcal H_t^\vartheta,
\end{equation*}
as required by assumption (ii) for $n(P,a)$. Now let us verify condition (L) for the code $\mathcal C_tP$: For $|a|\leq |b|$ we have $o_{\langle\rangle}(n(\mathcal C_tP,a))<o_{\langle\rangle}(\mathcal C_tP)$ as in the previous cases. By Proposition~\ref{prop:bh-epsilon-numbers} the ordinal $o_{\langle\rangle}(\mathcal C_tP)=\vartheta(t+\Omega^{o_{\langle\rangle}(P)})$ is additively principal and bigger than $\omega$. Thus we even have $o_{\langle\rangle}(n(\mathcal C_tP,a))+\omega\leq o_{\langle\rangle}(\mathcal C_tP)$, as required by condition (L) at the rule $(\forall_x,\psi)$. For $|a|>|b|$ we compute
\begin{equation*}
 o_{\langle\rangle}(n(\mathcal C_tP,a))+\omega=o_{\langle\rangle}(\textstyle\bigvee^0_{a\notin b,\theta(a)}\ax_{a\notin b})+\omega=1+\omega=\omega<o_{\langle\rangle}(\mathcal C_tP).
\end{equation*}
Still concerning (L), the formula $\forall_x\psi$ occurs in $l_{\langle\rangle}(P)=l_{\langle\rangle}(\mathcal C_tP)$. For $|a|\leq|b|$ we get
\begin{equation*}
 l_{\langle\rangle}(n(\mathcal C_tP,a))=l_{\langle\rangle}(n(P,a))\subseteq l_{\langle\rangle}(P)\cup\{\psi(a)\}=l_{\langle\rangle}(\mathcal C_tP)\cup\{\psi(a)\}
\end{equation*}
from condition (L) for $P$, and for $|a|>|b|$ we compute
\begin{multline*}
 l_{\langle\rangle}(n(\mathcal C_tP,a))=l_{\langle\rangle}(\textstyle\bigvee^0_{a\notin b,\theta(a)}\ax_{a\notin b})=l_{\langle\rangle}(\ax_{a\notin b})\backslash\{a\notin b\}\cup\{a\notin b\lor\theta(a)\}=\\
=\{a\notin b\lor\theta(a)\}=\{\psi(a)\}\subseteq l_{\langle\rangle}(\mathcal C_tP)\cup\{\psi(a)\}.
\end{multline*}
Condition (H1) is verified as in the previous cases. The same holds for condition (H2) in case $|a|\leq |b|$. For $|a|>|b|$ we compute
\begin{equation*}
 h_0(n(\mathcal C_tP,a))=h_0(\textstyle\bigvee^0_{a\notin b,\theta(a)}\ax_{a\notin b})=\Omega\leq t\leq h_1(\mathcal C_tP).
\end{equation*}
Condition (H3) for $|a|\leq|b|$ holds as in the previous cases. For $|a|>|b|$ we have
\begin{equation*}
 h_1(n(\mathcal C_tP,a))\leq k(a\notin b\lor\theta(a))\leq\max\{k(\forall_x\psi),|a|\}\leq\max\{k_{\langle\rangle}(P),|a|\}.
\end{equation*}
By condition (H1) and assumption (ii) for $P$ we get $k_{\langle\rangle}(P)\in\mathcal H_P^\vartheta\subseteq\mathcal H_{\mathcal C_tP}^\vartheta[|a|]$. Also, $|a|\in\mathcal H_{\mathcal C_tP}^\vartheta[|a|]$ holds by the definition of our operators. Finally, we have seen above that $|a|>|b|$ implies $o_{\langle\rangle}(n(\mathcal C_tP,a))=1$, and $1\in\mathcal H_{\mathcal C_tP}^\vartheta[|a|]$ is automatic. This completes the verification of condition (H3) for $\mathcal C_tP$.

\emph{Case $r_{\langle\rangle}(P)=(\exists_x,b,\varphi)$:} We set $r_{\langle\rangle}(\mathcal C_tP)=(\exists_x,b,\varphi)$ and $n(\mathcal C_tP,a)=\mathcal C_tn(P,a)$. Observe $|b|\leq k_{\langle\rangle}(P)$. By condition (H1) and assumption (ii) for $P$ we get
\begin{equation*}
 |b|\in\mathcal H_P^\vartheta\subseteq\mathcal H_t^\vartheta.
\end{equation*}
Using Proposition~\ref{prop:operators-special} we can deduce
\begin{equation*}
 |b|<\vartheta(t+\Omega^0)\leq\vartheta(t+\Omega^{o_{\langle\rangle}(P)})=o_{\langle\rangle}(\mathcal C_tP),
\end{equation*}
as condition (L) requires at the rule $(\exists_x,b,\varphi)$. The remaining conditions are verified as in the previous cases.

\emph{Case $r_{\langle\rangle}(P)=(\cut,\psi)$:} Note that assumption (i) and condition (C1) for $P$ guarantee $\hth(\psi)< d(P)\leq 2$. In case $\hth(\psi)=0$ we set $r_{\langle\rangle}(\mathcal C_tP)=(\cut,\psi)$ and $n(\mathcal C_tP,a)=\mathcal C_t n(P,a)$. To see that $n(P,i)$ satisfies assumptions (i) and (ii), observe $l_{\langle\rangle}(n(P,i))\subseteq l_{\langle\rangle}(P)\cup\{\psi,\neg\psi\}$ and recall that formulas of height zero are bounded. Condition (C1) for $\mathcal C_tP$ holds by
\begin{equation*}
 \hth(\psi)=0<1=d(\mathcal C_tP).
\end{equation*}
The other conditions are verified as above. Now assume that we have $\hth(\psi)=1$. It is easy to see that $\psi$ must be of the form $\exists_x\theta$ or $\forall_x\theta$ with the outer quantifier unbounded and $\theta$ a bounded formula. We only consider the case $\psi\equiv\exists_x\theta$, because the case $\psi\equiv\forall_x\theta$ is symmetric. Let us give an informal description of the proof idea first: The premise $n(P,1)$ of the cut rule may contain  the formula $\neg\psi\equiv\forall_x\neg\theta$. As this is not a $\Sigma(\mathbf L_{\omega^\alpha}^u)$-formula condition (i) fails for $n(P,1)$, and we cannot apply the operation $\mathcal C_t$ to $n(P,1)$. Instead, apply $\mathcal C_t$ to the premise $n(P,0)$, which contains the $\Sigma(\mathbf L_{\omega^\alpha}^u)$-formula $\psi$. This yields a deduction of $\psi$ with height $\vartheta(t+\Omega^{o_{\langle\rangle}(n(P,0))})<\Omega$. By the boundedness lemma we obtain a deduction of $\psi^{\vartheta(t+\Omega^{o_{\langle\rangle}(n(P,0))})}$. On the other hand we can apply boundedness to the premise $n(P,1)$, to get a deduction of $\neg\psi^{\vartheta(t+\Omega^{o_{\langle\rangle}(n(P,0))})}$ (it makes no difference whether we negate or relativize first). As $\neg\psi^{\vartheta(t+\Omega^{o_{\langle\rangle}(n(P,0))})}$ is a $\Sigma(\mathbf L_{\omega^\alpha}^u)$-formula (in fact a bounded formula) we may now collapse the ``bounded premise'' $n(P,1)$. Finally, we apply a cut over the bounded formula $\psi^{\vartheta(t+\Omega^{o_{\langle\rangle}(n(P,0))})}$. Formally, if $\psi\equiv\exists_x\theta$ with $\theta$ bounded then we set
\begin{align*}
 r_{\langle\rangle}(\mathcal C_tP)&=(\cut,\psi^{\vartheta(t+\Omega^{o_{\langle\rangle}(n(P,0))})}),\\
 n(\mathcal C_tP,a)&=\begin{cases}
                     \mathcal B_{\exists,\psi}^{\vartheta(t+\Omega^{o_{\langle\rangle}(n(P,0))})}\mathcal C_t n(P,0)\quad &\text{if $a=0$},\\
                     \mathcal C_{t+\Omega^{o_{\langle\rangle}(n(P,0))}}\mathcal B_{\forall,\neg\psi}^{\vartheta(t+\Omega^{o_{\langle\rangle}(n(P,0))})} n(P,1)\quad &\text{if $a\neq 0$}.
                    \end{cases}
\end{align*}
Let us show that this does not lead out of the ``intended cases'': Assumptions (i) and (ii) hold for $n(P,0)$ because of $l_{\langle\rangle}(n(P,0))\subseteq l_{\langle\rangle}(P)\cup\{\psi\}$ and because $\psi\equiv\exists_x\theta$ is a $\Sigma(\mathbf L_{\omega^\alpha}^u)$-formula. It follows that we have $o_{\langle\rangle}(\mathcal C_t n(P,0))=\vartheta(t+\Omega^{o_{\langle\rangle}(n(P,0))})$, which fits with the superscript of the function symbol $\mathcal B_{\exists,\psi}^{\vartheta(t+\Omega^{o_{\langle\rangle}(n(P,0))})}$. Concerning the case $a\neq 0$, we must verify that assumptions (i) and (ii) hold for the code $\mathcal B_{\forall,\neg\psi}^{\vartheta(t+\Omega^{o_{\langle\rangle}(n(P,0))})} n(P,1)$, with $t+\Omega^{o_{\langle\rangle}(n(P,0))}$ at the place of $t$. Now (i) holds by
\begin{multline*}
 l_{\langle\rangle}(\mathcal B_{\forall,\neg\psi}^{\vartheta(t+\Omega^{o_{\langle\rangle}(n(P,0))})} n(P,1))=l_{\langle\rangle}(n(P,1))\backslash\{\neg\psi\}\cup\{\neg\psi^{\vartheta(t+\Omega^{o_{\langle\rangle}(n(P,0))})}\}\subseteq\\
\subseteq (l_{\langle\rangle}(P)\cup\{\neg\psi\})\backslash\{\neg\psi\}\cup\{\neg\psi^{\vartheta(t+\Omega^{o_{\langle\rangle}(n(P,0))})}\}\subseteq l_{\langle\rangle}(P)\cup\{\neg\psi^{\vartheta(t+\Omega^{o_{\langle\rangle}(n(P,0))})}\},
\end{multline*}
where $(\neg\psi)^{\vartheta(t+\Omega^{o_{\langle\rangle}(n(P,0))})}\equiv\neg(\psi^{\vartheta(t+\Omega^{o_{\langle\rangle}(n(P,0))})})$ is a bounded formula, and by
\begin{equation*}
 d(\mathcal B_{\forall,\neg\psi}^{\vartheta(t+\Omega^{o_{\langle\rangle}(n(P,0))})} n(P,1))=d(n(P,1))\leq d(P)\leq 2.
\end{equation*}
Concerning assumption (ii), we have
\begin{equation*}
 h_0(\mathcal B_{\forall,\neg\psi}^{\vartheta(t+\Omega^{o_{\langle\rangle}(n(P,0))})} n(P,1))=h_0(n(P,1))\leq h_0(P)\leq t.
\end{equation*}
Also, condition (H3) and assumption (ii) for $P$ yield
\begin{equation*}
 o_{\langle\rangle}(n(P,0))\in\mathcal H_P^\vartheta\subseteq\mathcal H_t^\vartheta.
\end{equation*}
Together with $t\in\mathcal H_t^\vartheta$ this implies
\begin{equation*}
 t+\Omega^{o_{\langle\rangle}(n(P,0))}\in\mathcal H_t^\vartheta\subseteq\mathcal H_{t+\Omega^{o_{\langle\rangle}(n(P,0))}}^\vartheta,
\end{equation*}
as required by assumption (ii) with $t+\Omega^{o_{\langle\rangle}(n(P,0))}$ at the place of $t$. Finally, in view of $k(\psi)\leq k((\cut,\psi))\leq k_{\langle\rangle}(P)$ we get
\begin{multline*}
 h_1(\mathcal B_{\forall,\neg\psi}^{\vartheta(t+\Omega^{o_{\langle\rangle}(n(P,0))})} n(P,1))=\max\{h_1(n(P,1)),k(\neg\psi^{\vartheta(t+\Omega^{o_{\langle\rangle}(n(P,0))})})\}\leq\\
\leq\max\{h_1(n(P,1)),k_{\langle\rangle}(P),\vartheta(t+\Omega^{o_{\langle\rangle}(n(P,0))})\}.
\end{multline*}
By conditions (H1,H3) and assumption (ii) for $P$ we obtain
\begin{equation*}
 \{h_1(n(P,1)),k_{\langle\rangle}(P)\}\subseteq\mathcal H_P^\vartheta\subseteq\mathcal H_{t+\Omega^{o_{\langle\rangle}(n(P,0))}}^\vartheta.
\end{equation*}
We have already seen $t+\Omega^{o_{\langle\rangle}(n(P,0))}\in\mathcal H_{t+\Omega^{o_{\langle\rangle}(n(P,0))}}^\vartheta$, which implies
\begin{equation*}
 \vartheta(t+\Omega^{o_{\langle\rangle}(n(P,0))})\in\mathcal H_{t+\Omega^{o_{\langle\rangle}(n(P,0))}}^\vartheta.
\end{equation*}
This completes the verification of conditions (i) and (ii) for $\mathcal B_{\forall,\neg\psi}^{\vartheta(t+\Omega^{o_{\langle\rangle}(n(P,0))})} n(P,1)$ and $t+\Omega^{o_{\langle\rangle}(n(P,0))}$. Thus $n(\mathcal C_tP,1)=\mathcal C_{t+\Omega^{o_{\langle\rangle}(n(P,0))}}\mathcal B_{\forall,\neg\psi}^{\vartheta(t+\Omega^{o_{\langle\rangle}(n(P,0))})} n(P,1)$ falls under the ``intended case''. We can now verify condition (L) for $\mathcal C_tP$: Similar to the previous cases we have
\begin{multline*}
 o_{\langle\rangle}(n(\mathcal C_tP,0))=o_{\langle\rangle}(\mathcal B_{\exists,\psi}^{\vartheta(t+\Omega^{o_{\langle\rangle}(n(P,0))})}\mathcal C_t n(P,0))=o_{\langle\rangle}(\mathcal C_t n(P,0))=\\
=\vartheta(t+\Omega^{o_{\langle\rangle}(n(P,0))})<\vartheta(t+\Omega^{o_{\langle\rangle}(P)})=o_{\langle\rangle}(\mathcal C_tP).
\end{multline*}
In view of $o_{\langle\rangle}(n(P,1))\in\mathcal H_P^\vartheta\subseteq\mathcal H_{t+\Omega^{o_{\langle\rangle}(n(P,0))}}^\vartheta$ we also have
\begin{multline*}
 o_{\langle\rangle}(n(\mathcal C_tP,1))=o_{\langle\rangle}(\mathcal C_{t+\Omega^{o_{\langle\rangle}(n(P,0))}}\mathcal B_{\forall,\neg\psi}^{\vartheta(t+\Omega^{o_{\langle\rangle}(n(P,0))})} n(P,1))=\\
=\vartheta(t+\Omega^{o_{\langle\rangle}(n(P,0))}+\Omega^{o_{\langle\rangle}(n(P,1))})<\vartheta(t+\Omega^{o_{\langle\rangle}(n(P,0))}+\Omega^{o_{\langle\rangle}(P)})=\\
=\vartheta(t+\Omega^{o_{\langle\rangle}(P)})=o_{\langle\rangle}(\mathcal C_tP).
\end{multline*}
Concerning the end-sequents we have
\begin{multline*}
 l_{\langle\rangle}(n(\mathcal C_tP,0))=l_{\langle\rangle}(n(P,0))\backslash\{\psi\}\cup\{\psi^{\vartheta(t+\Omega^{o_{\langle\rangle}(n(P,0))})}\}\subseteq\\
\subseteq(l_{\langle\rangle}(P)\cup\{\psi\})\backslash\{\psi\}\cup\{\psi^{\vartheta(t+\Omega^{o_{\langle\rangle}(n(P,0))})}\}\subseteq l_{\langle\rangle}(\mathcal C_tP)\cup\{\psi^{\vartheta(t+\Omega^{o_{\langle\rangle}(n(P,0))})}\}
\end{multline*}
and
\begin{multline*}
 l_{\langle\rangle}(n(\mathcal C_tP,1))=l_{\langle\rangle}(n(P,1))\backslash\{\neg\psi\}\cup\{\neg\psi^{\vartheta(t+\Omega^{o_{\langle\rangle}(n(P,0))})}\}\subseteq\\
\subseteq(l_{\langle\rangle}(P)\cup\{\neg\psi\})\backslash\{\neg\psi\}\cup\{\neg\psi^{\vartheta(t+\Omega^{o_{\langle\rangle}(n(P,0))})}\}\subseteq l_{\langle\rangle}(\mathcal C_tP)\cup\{\neg\psi^{\vartheta(t+\Omega^{o_{\langle\rangle}(n(P,0))})}\},
\end{multline*}
as required by condition (L) at the rule $(\cut,\psi^{\vartheta(t+\Omega^{o_{\langle\rangle}(n(P,0))})})$. Coming to condition (C1), as $\psi^{\vartheta(t+\Omega^{o_{\langle\rangle}(n(P,0))})}$ is a bounded formula we have
\begin{equation*}
 \hth(\psi^{\vartheta(t+\Omega^{o_{\langle\rangle}(n(P,0))})})=0<1=d(\mathcal C_tP).
\end{equation*}
Conditions (C2) holds by
\begin{equation*}
 d(n(\mathcal C_tP,0))=d(\mathcal B_{\exists,\psi}^{\vartheta(t+\Omega^{o_{\langle\rangle}(n(P,0))})}\mathcal C_t n(P,0))=d(\mathcal C_t n(P,0))=1=d(\mathcal C_tP)
\end{equation*}
and
\begin{equation*}
 d(n(\mathcal C_tP,1))=d(\mathcal C_{t+\Omega^{o_{\langle\rangle}(n(P,0))}}\mathcal B_{\forall,\neg\psi}^{\vartheta(t+\Omega^{o_{\langle\rangle}(n(P,0))})} n(P,1))=1=d(\mathcal C_tP).
\end{equation*}
Concerning condition (H1), due to the new parameter in the cut formula we now have
\begin{equation*}
 k_{\langle\rangle}(\mathcal C_tP)\leq\max\{k_{\langle\rangle}(P),\vartheta(t+\Omega^{o_{\langle\rangle}(P)})^*,\vartheta(t+\Omega^{o_{\langle\rangle}(n(P,0))})\}.
\end{equation*}
As in the previous cases we have
\begin{equation*}
 \{k_{\langle\rangle}(P),\vartheta(t+\Omega^{o_{\langle\rangle}(P)})^*\}\subseteq\mathcal H_{\mathcal C_tP}^\vartheta.
\end{equation*}
Above we have seen $o_{\langle\rangle}(n(P,0))\in\mathcal H_t^\vartheta$. Together with $o_{\langle\rangle}(n(P,0))<o_{\langle\rangle}(P)$ and $t\in\mathcal H_t^\vartheta$ this implies
\begin{equation*}
 \vartheta(t+\Omega^{o_{\langle\rangle}(n(P,0))})\in\mathcal H_{t+\Omega^{o_{\langle\rangle}(P)}}^\vartheta=\mathcal H_{\mathcal C_tP}^\vartheta,
\end{equation*}
completing the verification of (H1). Condition (H2) holds by
\begin{multline*}
 h_0(n(\mathcal C_tP,0))=h_0(\mathcal B_{\exists,\psi}^{\vartheta(t+\Omega^{o_{\langle\rangle}(n(P,0))})}\mathcal C_t n(P,0))=h_0(\mathcal C_t n(P,0))=\\
=t+\Omega^{o_{\langle\rangle}(n(P,0))}<t+\Omega^{o_{\langle\rangle}(P)}=h_0(\mathcal C_tP)
\end{multline*}
and
\begin{multline*}
 h_0(n(\mathcal C_tP,1))=h_0(\mathcal C_{t+\Omega^{o_{\langle\rangle}(n(P,0))}}\mathcal B_{\forall,\neg\psi}^{\vartheta(t+\Omega^{o_{\langle\rangle}(n(P,0))})} n(P,1))=\\
=t+\Omega^{o_{\langle\rangle}(n(P,0))}+\Omega^{o_{\langle\rangle}(\mathcal B_{\forall,\neg\psi}^{\vartheta(t+\Omega^{o_{\langle\rangle}(n(P,0))})} n(P,1))}=t+\Omega^{o_{\langle\rangle}(n(P,0))}+\Omega^{o_{\langle\rangle}(n(P,1))}<\\
<t+\Omega^{o_{\langle\rangle}(n(P,0))}+\Omega^{o_{\langle\rangle}(P)}=t+\Omega^{o_{\langle\rangle}(P)}=h_0(\mathcal C_tP).
\end{multline*}
Concerning (H3), in view of $k(\psi)=k((\cut,\psi))\leq k_{\langle\rangle}(P)$ we have
\begin{multline*}
 h_1(n(\mathcal C_tP,0))=h_1(\mathcal B_{\exists,\psi}^{\vartheta(t+\Omega^{o_{\langle\rangle}(n(P,0))})}\mathcal C_t n(P,0))=k(\psi^{\vartheta(t+\Omega^{o_{\langle\rangle}(n(P,0))})})\leq\\
\leq\max\{k(\psi),\vartheta(t+\Omega^{o_{\langle\rangle}(n(P,0))})\}\leq\max\{k_{\langle\rangle}(P),\vartheta(t+\Omega^{o_{\langle\rangle}(n(P,0))})\}\in\mathcal H_{\mathcal C_tP}^\vartheta.
\end{multline*}
Also, we have
\begin{equation*}
 h_1(n(\mathcal C_tP,1))=h_1(\mathcal C_{t+\Omega^{o_{\langle\rangle}(n(P,0))}}\mathcal B_{\forall,\neg\psi}^{\vartheta(t+\Omega^{o_{\langle\rangle}(n(P,0))})} n(P,1))=0\in\mathcal H_{\mathcal C_tP}^\vartheta.
\end{equation*}
Finally, we have seen $o_{\langle\rangle}(n(\mathcal C_tP,0))=\vartheta(t+\Omega^{o_{\langle\rangle}(n(P,0))})\in\mathcal H_{\mathcal C_tP}^\vartheta$ above. Also, in view of $o_{\langle\rangle}(n(P,i))\in\mathcal H_P^\vartheta\subseteq\mathcal H_{t+\Omega^{o_{\langle\rangle}(P)}}^\vartheta$ we have
\begin{equation*}
 o_{\langle\rangle}(n(\mathcal C_tP,1))=\vartheta(t+\Omega^{o_{\langle\rangle}(n(P,0))}+\Omega^{o_{\langle\rangle}(n(P,1))})\in\mathcal H_{t+\Omega^{o_{\langle\rangle}(P)}}^\vartheta=\mathcal H_{\mathcal C_tP}^\vartheta,
\end{equation*}
completing the verification of (H3).

\emph{Case $r_{\langle\rangle}(P)=(\rref,\exists_z\forall_{x\in a}\exists_{y\in z}\theta)$:} Recall that we have only allowed this rule if~$\theta$ is a disjunction. This has the effect that the outer quantifier of the formula $\exists_y\theta$ must be unbounded. Arguing informally, the premise $n(P,0)$ deduces the $\Sigma(\mathbf L_{\omega^\alpha}^u)$-formula $\forall_{x\in a}\exists_y\theta$. By the induction hypothesis we can collapse this preproof to height $\vartheta(t+\Omega^{o_{\langle\rangle}(n(P,0))})<\Omega$. Boundedness yields a deduction of $\forall_{x\in a}\exists_{y\in\mathbb L_{\vartheta(t+\Omega^{o_{\langle\rangle}(n(P,0))})}^u}\theta$. We can use $\mathbb L_{\vartheta(t+\Omega^{o_{\langle\rangle}(n(P,0))})}^u$ as a witness to introduce the existential quantifier over~$z$. Formally we set
\begin{align*}
 r_{\langle\rangle}(\mathcal C_tP)&=(\exists_z,\mathbb L_{\vartheta(t+\Omega^{o_{\langle\rangle}(n(P,0))})}^u,\forall_{x\in a}\exists_{y\in z}\theta),\\
 n(\mathcal C_tP,a)&=\mathcal B_{\exists,\forall_{x\in a}\exists_y\theta}^{\vartheta(t+\Omega^{o_{\langle\rangle}(n(P,0))})}\mathcal C_tn(P,0).
\end{align*}
In view of $l_{\langle\rangle}(n(P,0))\subseteq l_{\langle\rangle}(P)\cup\{\forall_{x\in a}\exists_y\theta\}$ the code $n(P,0)$ satisfies assumptions (i) and (ii). It follows that $o_{\langle\rangle}(\mathcal C_tn(P,0))=\vartheta(t+\Omega^{o_{\langle\rangle}(n(P,0))})$ fits with the superscript of the function symbol $\mathcal B_{\exists,\forall_{x\in a}\exists_y\theta}^{\vartheta(t+\Omega^{o_{\langle\rangle}(n(P,0))})}$. Now let us verify local correctness: As in the previous cases we have $o_{\langle\rangle}(n(\mathcal C_tP,0))+\omega\leq o_{\langle\rangle}(\mathcal C_tP)$. In particular this implies
\begin{equation*}
 |\mathbb L_{\vartheta(t+\Omega^{o_{\langle\rangle}(n(P,0))})}^u|=\vartheta(t+\Omega^{o_{\langle\rangle}(n(P,0))})=o_{\langle\rangle}(n(\mathcal C_tP,0))<o_{\langle\rangle}(\mathcal C_tP),
\end{equation*}
as condition (L) requires at the rule $(\exists_z,\mathbb L_{\vartheta(t+\Omega^{o_{\langle\rangle}(n(P,0))})}^u,\forall_{x\in a}\exists_{y\in z}\theta)$. We also have
\begin{multline*}
 l_{\langle\rangle}(n(\mathcal C_tP,0))=l_{\langle\rangle}(n(P,0))\backslash\{\forall_{x\in a}\exists_y\theta\}\cup\{\forall_{x\in a}\exists_{y\in\mathbb L_{\vartheta(t+\Omega^{o_{\langle\rangle}(n(P,0))})}^u}\theta\}\subseteq\\
\subseteq (l_{\langle\rangle}(P)\cup\{\forall_{x\in a}\exists_y\theta\})\backslash\{\forall_{x\in a}\exists_y\theta\}\cup\{\forall_{x\in a}\exists_{y\in\mathbb L_{\vartheta(t+\Omega^{o_{\langle\rangle}(n(P,0))})}^u}\theta\}\subseteq\\
\subseteq l_{\langle\rangle}(P)\cup\{\forall_{x\in a}\exists_{y\in\mathbb L_{\vartheta(t+\Omega^{o_{\langle\rangle}(n(P,0))})}^u}\theta\}=l_{\langle\rangle}(\mathcal C_tP)\cup\{\forall_{x\in a}\exists_{y\in\mathbb L_{\vartheta(t+\Omega^{o_{\langle\rangle}(n(P,0))})}^u}\theta\}.
\end{multline*}
Condition (C1) does not apply, and (C2) holds by
\begin{equation*}
 d(n(\mathcal C_tP,0))=d(\mathcal B_{\exists,\forall_{x\in a}\exists_y\theta}^{\vartheta(t+\Omega^{o_{\langle\rangle}(n(P,0))})}\mathcal C_tn(P,0))=d(\mathcal C_tn(P,0))=1=d(\mathcal C_tP).
\end{equation*}
As for (H1), due to the new parameter $\mathbb L_{\vartheta(t+\Omega^{o_{\langle\rangle}(n(P,0))})}^u$ in the rule we have
\begin{equation*}
 k_{\langle\rangle}(\mathcal C_tP)\leq\max\{k_{\langle\rangle}(P),\vartheta(t+\Omega^{o_{\langle\rangle}(P)})^*,\vartheta(t+\Omega^{o_{\langle\rangle}(n(P,0))})\}.
\end{equation*}
We can deduce $k_{\langle\rangle}(\mathcal C_tP)\in\mathcal H_{\mathcal C_tP}^\vartheta$ as in the previous case. Condition (H2) holds by
\begin{equation*}
 h_0(n(\mathcal C_tP,0))=h_0(\mathcal C_tn(P,0))=t+\Omega^{o_{\langle\rangle}(n(P,0))}<t+\Omega^{o_{\langle\rangle}(P)}=h_0(\mathcal C_tP).
\end{equation*}
Concerning (H3), as $\exists_z\forall_{x\in a}\exists_{y\in z}\theta$ occurs in $l_{\langle\rangle}(P)$ we have $k(\forall_{x\in a}\exists_y\theta)\leq k_{\langle\rangle}(P)$. We can deduce
\begin{multline*}
 h_1(n(\mathcal C_tP,0))=h_1(\mathcal B_{\exists,\forall_{x\in a}\exists_y\theta}^{\vartheta(t+\Omega^{o_{\langle\rangle}(n(P,0))})}\mathcal C_tn(P,0))=k(\forall_{x\in a}\exists_{y\in\mathbb L_{\vartheta(t+\Omega^{o_{\langle\rangle}(n(P,0))})}^u}\theta)\leq\\
\leq\max\{k(\forall_{x\in a}\exists_y\theta),\vartheta(t+\Omega^{o_{\langle\rangle}(n(P,0))})\}\leq\max\{k_{\langle\rangle}(P),\vartheta(t+\Omega^{o_{\langle\rangle}(n(P,0))})\}.
\end{multline*}
As in the previous cases we see
\begin{equation*}
 \{k_{\langle\rangle}(P),\vartheta(t+\Omega^{o_{\langle\rangle}(n(P,0))})\}\subseteq\mathcal H_{\mathcal C_tP}^\vartheta.
\end{equation*}
Thus we get $h_1(n(\mathcal C_tP,0))\in\mathcal H_{\mathcal C_tP}^\vartheta$, and also
\begin{equation*}
 o_{\langle\rangle}(n(\mathcal C_tP,0))=\vartheta(t+\Omega^{o_{\langle\rangle}(n(P,0))})\in\mathcal H_{\mathcal C_tP}^\vartheta,
\end{equation*}
as required by condition (H3).

\emph{Case $r_{\langle\rangle}(P)=(\rep,b)$:} We set $r_{\langle\rangle}(\mathcal C_tP)=(\rep,b)$ and $n(\mathcal C_tP,a)=\mathcal C_tn(P,a)$. Note that we have $|b|\leq k_{\langle\rangle}(P)$. Thus condition (H1) for $P$ implies $|b|\in\mathcal H_P^\vartheta$, which is equivalent to $\mathcal H_P^\vartheta[|b|]=\mathcal H_P^\vartheta$. By condition (H3) for $P$ we get
\begin{equation*}
 o_{\langle\rangle}(n(P,b))\in\mathcal H_P^\vartheta[|b|]=\mathcal H_P^\vartheta\subseteq\mathcal H_t^\vartheta.
\end{equation*}
With this in mind we can check local correctness as in the previous cases.
\end{proof}

Putting pieces together we obtain the first theorem from the introduction:

\begin{theorem}
 Working in primitive recursive set theory, consider a countable transitive set $u$. If the implication
\begin{equation*}
 \wop(\alpha\mapsto\varepsilon(S_{\omega^\alpha}^u))\rightarrow\exists_\alpha\exists_\vartheta\,\vartheta:\varepsilon(S_{\omega^\alpha}^u)\bh\alpha
\end{equation*}
holds then there is an admissible set $\mathbb A$ with $u\subseteq\mathbb A$.
\end{theorem}
\begin{proof}
 Fix an enumeration $u=\{u_i\,|\,i\in\omega\}$. Aiming at a contradiction, assume that there is no admissible set $\mathbb A$ with $u\subseteq\mathbb A$. By Proposition~\ref{prop:no-admissible-gives-wop} this makes $\alpha\mapsto\varepsilon(S_{\omega^\alpha}^u)$ a well-ordering principle. Then the assumption yields a Bachmann-Howard collapse $\vartheta:\varepsilon(S_{\omega^\alpha}^u)\bh\alpha$, for some ordinal $\alpha$. Now consider the $\mathbf L_{\omega^\alpha}^u$-code $\mathcal C_\Omega\mathcal E^CP_\alpha^u\langle\rangle$, where $C\in\omega$ is as in Proposition~\ref{prop:without-cuts}. Let us check that $\mathcal E^CP_\alpha^u\langle\rangle$ satisfies assumptions (i) and (ii) of Theorem~\ref{thm:collapsing}, with $t=\Omega$: Proposition~\ref{prop:without-cuts} tells us that the sequent $l_{\langle\rangle}(\mathcal E^CP_\alpha^u\langle\rangle)$ is empty, so trivially it contains only $\Sigma(\mathbf L_{\omega^\alpha}^u)$-formulas. The same proposition ensures $d(\mathcal E^CP_\alpha^u\langle\rangle)=2$, as required by assumption (i) of Theorem~\ref{thm:collapsing}. As for assumption (ii), Definition~\ref{def:operators-basic-proofs} and Lemma~\ref{lem:operators-cut-elimination} give $h_0(\mathcal E^CP_\alpha^u\langle\rangle)=h_0(P_\alpha^u\langle\rangle)=\Omega$ and $h_1(\mathcal E^CP_\alpha^u\langle\rangle)=h_1(P_\alpha^u\langle\rangle)=|\langle\rangle|=0$. Together with $\Omega\in\mathcal H_\Omega^\vartheta$ this establishes assumption (ii) of Theorem~\ref{thm:collapsing}. The theorem thus yields
\begin{align*}
 l_{\langle\rangle}(\mathcal C_\Omega\mathcal E^CP_\alpha^u\langle\rangle)&=\langle\rangle,\\
 o_{\langle\rangle}(\mathcal C_\Omega\mathcal E^CP_\alpha^u\langle\rangle)&=\vartheta(\Omega+\Omega^{\varepsilon_{\langle\rangle}}).
\end{align*}
By Corollary~\ref{cor:correct-proofs-from-codes} this means that $[\mathcal C_\Omega\mathcal E^CP_\alpha^u\langle\rangle]$ is an $\mathbf L_{\omega^\alpha}^u$-preproof with empty end-sequent and height $\vartheta(\Omega+\Omega^{\varepsilon_{\langle\rangle}})<\Omega$. However, Proposition~\ref{prop:countable-proof-sound} tells us that such a proof cannot exist. Thus we have reached the desired contradiction.
\end{proof}

It is easy to deduce the implication ``(ii) $\Rightarrow$ (iii)'' in the second theorem of the introduction: To establish (iii), consider an arbitrary set $v$. By the axiom of countability there is a countable transitive set $u$ with $v\in u$. Part (ii) provides the implication
\begin{equation*}
 \wop(\alpha\mapsto\varepsilon(S_{\omega^\alpha}^u))\rightarrow\exists_\alpha\exists_\vartheta\,\vartheta:\varepsilon(S_{\omega^\alpha}^u)\bh\alpha.
\end{equation*}
Thus the previous theorem yields an admissible set $\mathbb A$ with $u\subseteq\mathbb A$. In particular we have $v\in\mathbb A$, as desired.

\section{A Well-Ordering Proof}\label{sect:wo-proof}

In the previous sections we have shown that the existence of a Bachmann-Howard collapse $\vartheta:\varepsilon(S_{\omega^\alpha}^u)\bh\alpha$ for one particular well-ordering principle $\alpha\mapsto\varepsilon(S_{\omega^\alpha}^u)$ entails the existence of an admissible set $\mathbb A$ with $u\subseteq\mathbb A$. The goal of the present section is to establish a converse: Consider some well-ordering principle \mbox{$T^u:\alpha\mapsto T_\alpha^u$} with rank function $s\mapsto|s|_T^u$. Assume that $\mathbb A$ is an admissible set with $u\in\mathbb A$, and write
\begin{equation*}
 o(\mathbb A):=\min\{\alpha\in\ordi\,|\,\alpha\notin\mathbb A\}=\ordi\cap\mathbb A.
\end{equation*}
We will prove that there is a Bachmann-Howard collapse $\vartheta_{\mathbb A}:T_{o(\mathbb A)}^u\rightarrow o(\mathbb A)$. The idea is similar to \cite[Section~4]{rathjen92}: Perform the construction from Remark~\ref{rmk:Bachmann-Howard-semantically} inside $\mathbb A$, with $o(\mathbb A)$ at the place of $\aleph_1$ and ``element of $\mathbb A$'' at the place of ``countable''. This will lead to a $\Sigma$-formula $D_T(u,s,\alpha)$ such that we can set
\begin{equation*}
 \vartheta_{\mathbb A}(s)=\alpha\quad:\Leftrightarrow\quad\mathbb A\vDash D_T(u,s,\alpha)
\end{equation*}
for all $s\in T_{o(\mathbb A)}^u$ and $\alpha<o(\mathbb A)$. As a preparation we need to recover primitive recursive functions, in particular the function $\alpha\mapsto T_\alpha^u$, inside $\mathbb A$:

\begin{lemma}
 For each primitive recursive function $F$ there is a $\Sigma$-formula $\varphi_F$ (in the language of pure set theory, i.e.~without primitive recursive function symbols) such that the following is provable in primitive recursive set theory: For any admissible set $\mathbb A$ and any $\vec x,y\in\mathbb A$ we have
\begin{equation*}
 F(\vec x)=y\quad\Leftrightarrow\quad\mathbb A\vDash\varphi_F(\vec x,y),
\end{equation*}
and indeed $F(\vec x)\in\mathbb A$.
\end{lemma}
\begin{proof}
 We argue by (meta-) induction on the build up of primitive recursive set functions (see e.g.~\cite[Definition~2.1]{rathjen-set-functions}). The basic functions and composition are easily accomodated. To prepare definitions by recursion, consider the notion of transitive closure: Let $\operatorname{TC}(\cdot)$ be the canonical primitive recursive function that computes transitive closures. By the proof of \cite[Theorem~I.6.1]{barwise-admissible} there is a $\Sigma$-formula $\varphi_{\operatorname{TC}}$ such that primitive recursive set theory proves
\begin{equation*}
 \forall_{x,y}(\operatorname{TC}(x)=y\leftrightarrow\varphi_{\operatorname{TC}}(x,y)),
\end{equation*}
as well as
\begin{equation*}
 \mathbb A\vDash\forall_x\exists_y\varphi_{\operatorname{TC}}(x,y)
\end{equation*}
for any admissible set $\mathbb A$. By upward absoluteness of $\Sigma$-formulas $\mathbb A\vDash\varphi_{\operatorname{TC}}(x,y)$ implies $\operatorname{TC}(x)=y$. It follows that $\mathbb A$ is closed under transitive closures, and that we have
\begin{equation*}
 \operatorname{TC}(x)=y\quad\Leftrightarrow\quad\mathbb A\vDash\varphi_{\operatorname{TC}}(x,y).
\end{equation*}
In particular $\varphi_{\operatorname{TC}}$ is a $\Delta$-formula from the viewpoint of $\mathbb A$, namely
\begin{equation*}
 \mathbb A\vDash\varphi_{\operatorname{TC}}(x,y)\leftrightarrow\forall_{y'}(y'\neq y\rightarrow\neg\varphi_{\operatorname{TC}}(x,y')).
\end{equation*}
Misusing notation we will use the function symbol $\operatorname{TC}(\cdot)$ in formulas of pure set theory: For example $\mathbb A\vDash z\in\operatorname{TC}(x)$ abbreviates either the $\Sigma$-formula
\begin{equation*}
 \mathbb A\vDash\exists_y(\varphi_{\operatorname{TC}}(x,y)\land z\in y)
\end{equation*}
or the $\Pi$-formula
\begin{equation*}
 \mathbb A\vDash\forall_y(\varphi_{\operatorname{TC}}(x,y)\rightarrow z\in y),
\end{equation*}
depending on the context. After this preparation, consider a function
\begin{equation*}
 F(z,\vec x)=H(\bigcup\{F(w,\vec x)\,|\,w\in z\},z,\vec x)
\end{equation*}
defined by recursion. We repeat the usual proof of $\Sigma$-recursion (see~\cite[Theorem~I.6.4]{barwise-admissible}) inside $\mathbb A$: Let $\varphi_H$ be the $\Sigma$-definition of $H$ provided by the induction hypothesis. Set $\varphi_F(z,\vec x,y):\equiv\exists_f\theta(z,\vec x,y,f)$ where
\begin{multline*}
 \theta(z,\vec x,y,f):\equiv\text{``$f$ is a function with domain $\operatorname{TC}(z)$''}\land\\
\land\forall_{w\in\operatorname{TC}(z)}\varphi_H(\bigcup\rng(f\!\restriction_w),w,\vec x,f(w))\land\varphi_H(\bigcup\rng(f\!\restriction_z),z,\vec x,y).
\end{multline*}
The claims of the lemma are established by induction on $\operatorname{TC}(z)$ (see~\cite[Theorem~I.6.3]{barwise-admissible}; note that our induction statement is primitive recursive): Concerning ``$\Leftarrow$'', assume that we have $\mathbb A\vDash\varphi_F(z,\vec x,y)$, i.e.~$\mathbb A\vDash\theta(z,\vec x,y,f)$ for some~$f\in\mathbb A$. By the definition of $\theta$ we get $\mathbb A\vDash\theta(w,\vec x,f(w),f\!\restriction_{\operatorname{TC}(w)})$ for all $w\in z$. This implies $\mathbb A\vDash\varphi_F(w,\vec x,f(w))$, so that the induction hypothesis yields $F(w,\vec x)=f(w)$. Furthermore, from $\mathbb A\vDash\theta(z,\vec x,y,f)$ we get $\mathbb A\vDash\varphi_H(\bigcup\rng(f\!\restriction_z),z,\vec x,y)$. Using the claim for $H$ we obtain
\begin{equation*}
 y=H(\bigcup\rng(f\!\restriction_z),z,\vec x)=H(\bigcup\{F(w,\vec x)\,|\,w\in z\},z,\vec x)=F(z,\vec x),
\end{equation*}
as required for direction ``$\Leftarrow$'' of the lemma. As for direction ``$\Rightarrow$'', the induction hypothesis gives $F(w,\vec x)\in\mathbb A$ and $\mathbb A\vDash\varphi_F(w,\vec x,F(w,\vec x))$ for all $w\in\operatorname{TC}(z)$. Since direction ``$\Leftarrow$'' provides unicity we obtain
\begin{equation*}
 \mathbb A\vDash\forall_{w\in\operatorname{TC}(z)}\exists !_y\varphi_F(w,\vec x,y).
\end{equation*}
Now $\Sigma$-replacement (see~\cite[Theorem~I.4.6]{barwise-admissible}) inside $\mathbb A$ yields a function $f\in\mathbb A$ with domain $\operatorname{TC}(z)$ and
\begin{equation*}
\mathbb A\vDash\forall_{w\in\operatorname{TC}(z)}\varphi_F(w,\vec x,f(w)).
\end{equation*}
Direction ``$\Leftarrow$'' tells us $f=F(\cdot,\vec x)\!\restriction_{\operatorname{TC}(z)}$. By the claim for $H$ the value $F(z,\vec x)=H(\bigcup\rng(f\!\restriction_z),z,\vec x)$ lies in $\mathbb A$. It remains to show
\begin{equation*}
 \mathbb A\vDash\theta(z,\vec x,H(\bigcup\rng(f\!\restriction_z),z,\vec x),f).
\end{equation*}
The first conjunct of $\theta(z,\vec x,H(\bigcup\rng(f\!\restriction_z),z,\vec x),f)$ is immediate. The third conjunct, i.e.~the statement
\begin{equation*}
 \mathbb A\vDash\varphi_H(\bigcup\rng(f\!\restriction_z),z,\vec x,H(\bigcup\rng(f\!\restriction_z),z,\vec x)),
\end{equation*}
holds by the claim for $H$. Similarly, the second conjunct
\begin{equation*}
 \mathbb A\vDash\forall_{w\in\operatorname{TC}(z)}\varphi_H(\bigcup\rng(f\!\restriction_w),w,\vec x,f(w))
\end{equation*}
reduces to
\begin{equation*}
 H(\bigcup\rng(f\!\restriction_w),w,\vec x)=f(w).
\end{equation*}
This is the same as
\begin{equation*}
 H(\bigcup\{F(u,\vec x)\,|\,u\in w\},w,\vec x)=F(w,\vec x),
\end{equation*}
which is the defining clause for $F$.
\end{proof}

In formulas of pure set theory (which are not supposed to contain symbols for primitive recursive functions) we will use $F(\vec x)=y$ as an abbreviation for $\varphi_F(\vec x,y)$. As we have seen in the case of transitive closure, the lemma implies that $\varphi_F$ is a $\Delta$-formula from the viewpoint of any admissible set. We will also use functional notation, e.g.~writing $z\in F(\vec x)$ for either of the formulas \mbox{$\exists_y(y=F(\vec x)\land z\in y)$} or $\forall_y(y=F(\vec x)\rightarrow z\in y)$, which are equivalent in any admissible set. The formula $D_T(u,\alpha,s)$ that we have described above will be constructed via the second recursion theorem, applied inside our admissible set:

\begin{lemma}
 Let $C(x_1,\dots,x_n,\vec y,R)$ be a $\Sigma$-formula involving an $n$-ary relation symbol $R$ which only occurs positively. Then there is a $\Sigma$-formula $D(x_1,\dots,x_n,\vec y)$ such that we have
\begin{equation*}
 \mathbb A\vDash\forall_{x_1,\dots,x_n,\vec y}(D(x_1,\dots,x_n,\vec y)\leftrightarrow C(x_1,\dots,x_n,\vec y,\{x_1,\dots,x_n\,|\,D(x_1,\dots,x_n,\vec y)\}))
\end{equation*}
for any admissible set $\mathbb A$.
\end{lemma}
\begin{proof}
 This is the second recursion theorem (see~\cite[Theorem~V.2.3]{barwise-admissible}). In the meta-theory we need to construct certain proofs in Kripke-Platek set theory. This is done by induction on formulas, and clearly feasible in primitive recursive set theory (and in much weaker theories).
\end{proof}

In the following we assume that $\alpha\mapsto T_\alpha^u$ is a well-ordering principle with rank function $s\mapsto|s|_T^u$. In particular this means that these functions are primitive recursive. Recall that $T^u=\bigcup_{\alpha\in\ordi}T_\alpha^u$ is a primitive recursive class (namely $s\in T^u$ precisely if $s\in T_{|s|_T^u+1}^u$). We observe that
\begin{equation*}
 \mathbb A\cap T^u=T_{o(\mathbb A)}^u
\end{equation*}
holds for any admissible set $\mathbb A$: If we have $s\in\mathbb A\cap T^u$ then the rank $|s|_T^u$ of $s$ lies in $\mathbb A$ as well, by the closure of admissible sets under primitive recursive functions. By the properties of the rank we have $s\in T_{|s|_T^u+1}^u\subseteq T_{o(\mathbb A)}^u$. On the other hand, $s\in T_{o(\mathbb A)}^u$ implies $s\in T_\alpha^u$ for some $\alpha\in\mathbb A$. Again by closure under primitive recursive functions we get $T_\alpha^u\in\mathbb A$. As $\mathbb A$ is transitive this implies $s\in T_\alpha^u\subseteq\mathbb A$. Using the lemma we can construct a $\Sigma$-formula $D_T(u,s,\alpha)$ such that we have
\begin{align*}
\mathbb A\vDash D_T(u,s,\alpha)\leftrightarrow{} & s\in T^u\land\alpha\in\ordi\,\land\\
& \begin{aligned}
\exists_a(&\text{``$a:\omega\rightarrow\ordi$ is a function"}\,\land\\
& a(0)=|s|_T^u+1\,\land\\
& \begin{aligned}\forall_{n\in\omega}\exists_d(&\text{``$d:\{t\in T_{a(n)}^u\,|\,t<_{T^u} s\}\rightarrow\ordi$ is a function"}\,\land\\
		&\forall_{t\in\dom(d)}D_T(u,t,d(t))\,\land\\
		&a(n+1)=\sup\{d(t)+1\,|\,t\in\dom(d)\})\,\land\end{aligned}\\
& \alpha=\textstyle\sup_{n\in\omega} a(n))
\end{aligned}
\end{align*}
for any admissible set $\mathbb A$ with $u\in\mathbb A$. Let us begin with uniqueness:

\begin{lemma}
 Let $\mathbb A\ni u$ be an admissible set. Then we have
\begin{equation*}
 \mathbb A\vDash D_T(u,s,\alpha_0)\land D_T(u,s,\alpha_1)\rightarrow\alpha_0=\alpha_1
\end{equation*}
for all $s\in T_{o(\mathbb A)}^u$.
\end{lemma}
\begin{proof}
 We argue by induction on $s$ (i.e.~induction over the well-ordering~$<_{T_{o(\mathbb A)}^u}$). Assume that $\mathbb A\vDash D_T(u,s,\alpha_0)$ and $\mathbb A\vDash D_T(u,s,\alpha_1)$ are witnessed by functions $a_0,a_1:\omega\rightarrow\ordi$ with $\alpha_i=\sup_{n\in\omega}a(n)$. To establish the claim we show $a_0(n)=a_1(n)$ by induction on $n$. The base $n=0$ is immediate. Concerning the step, we have
\begin{equation*}
a_i(n+1)=\sup\{d_i(t)+1\,|\,t\in\dom(d_i)\}
\end{equation*}
for some functions $d_i:\{t\in T_{a_i(n)}^u\,|\,t<_{T^u} s\}\rightarrow\ordi$ which satisfy
\begin{equation*}
\mathbb A\vDash D_T(u,t,d_i(t))
\end{equation*}
for all $t\in\dom(d_i)$. By induction hypothesis we have $a_0(n)=a_1(n)$, so that the domains of $d_0$ and $d_1$ are equal. As $t\in\dom(d_i)$ implies $t<_{T_{o(\mathbb A)}^u} s$ the main induction hypothesis yields $d_0(t)=d_1(t)$ for all such $t$. This clearly implies $a_0(n+1)=a_1(n+1)$, as required.
\end{proof}

After uniqueness we establish existence:

\begin{proposition}
 Let $\mathbb A\ni u$ be an admissible set. For all $s\in T_{o(\mathbb A)}^u$ there is an ordinal $\alpha<o(\mathbb A)$ with $\mathbb A\vDash D_T(u,s,\alpha)$.
\end{proposition}
\begin{proof}
Again we argue by induction on $s$. To establish the claim for $s$ we construct functions $a_m:m+1\rightarrow\ordi$ in $\mathbb A$ such that we have $a_m(0)=|s|_T^u+1$ and
\begin{multline*}
\mathbb A\vDash\exists_d(\text{``$d:\{t\in T_{a_m(n)}^u\,|\,t<_{T^u} s\}\rightarrow\ordi$ is a function"}\,\land\\
		\forall_{t\in\dom(d)}D_T(u,t,d(t))\,\land a_m(n+1)=\sup\{d(t)+1\,|\,t\in\dom(d)\})
\end{multline*}
for all $n<m$. The function $a_0$ is simply the pair $\langle 0,|s|_T^u+1\rangle$. To extend $a_m$ to $a_{m+1}$ it suffices to construct a value $a_{m+1}(m+1)$ which satisfies the above condition. First, observe that the set $T_{a_m(m)}^u$ lies in $\mathbb A$ by closure under primitive recursive functions. Using $\Delta$-separation in $\mathbb A$ we see that $\{t\in T_{a_m(m)}^u\,|\,t<_{T^u} s\}$ is an element of $\mathbb A$. By the induction hypothesis and the previous lemma we get
\begin{equation*}
\mathbb A\vDash\forall_{t\in\{t\in T_{a_m(m)}^u\,|\,t<_{T^u} s\}}\exists!_\alpha D_T(u,t,\alpha).
\end{equation*}
Then $\Sigma$-replacement (see~\cite[Theorem~4.6]{barwise-admissible}) in the admissible set $\mathbb A$ provides a function $d:\{t\in T_{a_m(m)}^u\,|\,t<_{T^u} s\}\rightarrow\ordi$ in $\mathbb A$ such that we have $\mathbb A\vDash D_T(u,t,d(t))$ for all $t\in\dom(d)$. Setting
\begin{equation*}
a_{m+1}:=a_m\cup\{\langle m+1,\sup\{d(t)+1\,|\,t\in\dom(d)\}\rangle\}
\end{equation*}
completes the induction step. Using the previous lemma one checks that the functions $a_m\in\mathbb A$ are unique. Thus, by $\Sigma$-replacement, the function $m\mapsto a_m$ lies itself in $\mathbb A$. Finally, it follows that the admissible set $\mathbb A$ contains the function $a:\omega\rightarrow\ordi$ defined by $a(n):=a_n(n)$. This function witnesses
\begin{equation*}
\mathbb A\vDash D_T(u,s,\textstyle\sup_{n\in\omega}a(n)),
\end{equation*}
which establishes the claim for $s$.
\end{proof}

The previous two results justify the following:

\begin{definition}
Consider a well-ordering principle $\alpha\mapsto T_\alpha^u$ and an admissible set $\mathbb A$ with $u\in\mathbb A$. We define a function $\vartheta_{\mathbb A}:T_{o(\mathbb A)}^u\rightarrow o(\mathbb A)$ by setting
\begin{equation*}
\vartheta_{\mathbb A}=\{\langle s,\alpha\rangle\in T_{o(\mathbb A)}^u\times o(\mathbb A)\,|\,\mathbb A\vDash D_T(u,s,\alpha)\}.
\end{equation*} 
\end{definition}

Finally, we verify that we have indeed constructed a Bachmann-Howard collapse:

\begin{theorem}
Assume that $\alpha\mapsto T_\alpha^u$ is a well-ordering principle, and that $\mathbb A$ is an admissible set with $u\in\mathbb A$. Then the function $\vartheta_{\mathbb A}:T_{o(\mathbb A)}^u\rightarrow o(\mathbb A)$ is a Bachmann-Howard collapse.
\end{theorem}
\begin{proof}
We must verify the two conditions from Definition~\ref{def:bh-collapse}. Condition (i) requests $|s|_T^u<\vartheta_{\mathbb A}(s)$ for all $s\in T_{o(\mathbb A)}^u$. To see that this is satisfied, let $a:\omega\rightarrow\ordi$ be a function which witnesses $\mathbb A\vDash D_T(u,s,\vartheta_{\mathbb A}(s))$. Then we have
\begin{equation*}
|s|_T^u<|s|_T^u+1=a(0)\leq\textstyle\sup_{n\in\omega} a(n)=\vartheta_{\mathbb A}(s).
\end{equation*}
Condition (ii) asks us to deduce $\vartheta_{\mathbb A}(s)<\vartheta_{\mathbb A}(t)$ from $s<_{T_{o(\mathbb A)}^u} t$ and $|s|_T^u<\vartheta_{\mathbb A}(t)$. Let $a:\omega\rightarrow\ordi$ be a witness for $\mathbb A\vDash D_T(u,t,\vartheta_{\mathbb A}(t))$. In particular we have $\vartheta_{\mathbb A}(t)=\sup_{n\in\omega}a(n)$, and thus $|s|_T^u<a(n)$ for some $n\in\omega$. So we see
\begin{equation*}
s\in \{r\in T_{a(n)}^u\,|\,r<_{T^u} t\}.
\end{equation*}
The definition of $D_T(u,t,\vartheta_{\mathbb A}(t))$ yields $a(n+1)=\sup\{d(r)+1\,|\,r\in\dom(d)\}$ for some function $d:\{r\in T_{a(n)}^u\,|\,r<_{T^u} t\}\rightarrow\ordi$ which satisfies $\mathbb A\vDash D_T(u,r,d(r))$ for all $r\in\dom(d)$. The latter means that $d$ coincides with $\vartheta_{\mathbb A}$, so that we get
\begin{equation*}
\vartheta_{\mathbb A}(s)=d(s)<a(n+1)\leq\textstyle\sup_{n\in\omega} a(n)=\vartheta_{\mathbb A}(t),
\end{equation*}
as required.
\end{proof}

This establishes ``(iii) $\Rightarrow$ (ii)'' of the theorem in the introduction.

\bibliographystyle{alpha}
\bibliography{Higher_Bachmann-Howard_Principle.bib}

\end{document}